%% file: ArXiV-main.tex
\documentclass{amsart}

\newcommand{\tempnewpage}{}

\newcommand{\real}{\realSp}
\newcommand{\realR}{\realSp^*}
\newcommand{\realSp}{\ensuremath\text{\upshape\textsc{real}}}

\DeclareMathOperator{\RelO}{\ensuremath\text{\upshape\textsc{RelO}}}

\DeclareMathOperator{\mreal}{\ensuremath\text{\upshape\textsc{mreal}}}

\newcommand{\VF}{\ensuremath{V\!F}}

\newcommand{\Y}{\ensuremath{\mathrm{Y}}}

\newcommand{\DY}{\ensuremath{\Delta\kern-1.2pt \Y}}
\newcommand{\YD}{\ensuremath{\Y\kern-1.2pt\Delta}}

\input{header}

\begin{document}

	
	
\title[Deformations and second-order rigidity of polytopes]%
{Deformations and second-order rigidity of polytopes}
		
\author[M.\ Adrian-Himmelmann]{Matthias Adrian-Himmelmann}
\address{Institut für Analysis und Algebra, Technische Universität Braunschweig, Braunschweig, Germany}
\email{matthias.himmelmann@tu-braunschweig.de}

\author[M.\ Winter]{Martin Winter}
\address{Nonlinear Algebra Group, Max-Planck-Institute for Mathematics in the Sciences, Leipzig, Germany}
\email{martin.winter@mis.mpg.de}

\author[\;Z.\ Zhang]{Zhen Zhang}
\address{Yau Mathematical Sciences Center, Tsinghua University, Haidian District, Beijing, China}
\email{zz628@tsinghua.edu.cn}

		
\date{\today}
\begin{abstract}


We study deformations of polytopes that preserve edge lengths and face coplanarities.
This gives rise to a notion of rigidity, for which we develop a \textit{second-order theory} together with an effective algorithm for testing second-order rigidity symbolically.
The strength of these new tools is demonstrated on several polytope classes, including the previously intractable regular dodecahedron, which we find to be rigid.
We also find that a single tested polytope evades our techniques and will therefore provide a simple test case for future developments of tools of even higher order.

This paper also contains a study of so-called \textit{edge-length perturbations}. We point out connections between rigidity and the ability to realize polytopes with slightly perturbed edge lengths.
To explore these connections in practice, we dedicate a section to another case study of the regular dodecahedron.
We investigate the local behavior of its realization space with a view towards edge-length perturbations, singularities and generic global rigidity.


\end{abstract}


\vspace*{-1.5em}
\maketitle

\input{sec/introduction}

\par\bigskip
{
\parindent 0pt
\textbf{Funding.} 
Martin Winter was supported as Dirichlet Fellow by the Berlin Mathematics Research Center MATH\raisebox{0.25ex}{$+$} and the Berlin Mathematical School, funded by the Deutsche Forschungsgemeinschaft (DFG, German Research Foundation) under Germany’s Excellence Strategy (EXC-2046/1, project ID 390685689), and furthermore by the SPP 2458 ``Combinatorial Synergies'' (project ID 539851419), funded by the Deutsche Forschungsgemeinschaft (DFG, German Research Foundation).
}

\par\bigskip
{
\parindent 0pt
\textbf{Acknowledgements.} We are grateful to Bob Connelly and Herman Servatius for helpful discussions about computational and constructive approaches for determining the rigidity of geometric constraint systems. 
We also thank Bob Connelly and Simon Guest for granting us permission to use their visualization of the flat dodecahedron.
The authors thank Louis Theran for his ideas and discussions about an energy-based approach that can potentially let us determine the higher-order rigidity of the truncated dodecahedron.
We are also grateful to Bernd Schulze for helping pave the way for demonstrating the rigidity of the regular dodecahedron by contributing various useful ideas and constructions.
}

\bibliographystyle{abbrv}
\bibliography{references}

\vspace*{1em}
\addresseshere

\newpage

\input{sec/appendix}

\end{document}

%% file: header.tex
\usepackage[T1]{fontenc}
\usepackage[utf8]{inputenc}
\usepackage[USenglish]{babel}
\usepackage[
	bookmarks=true,
	plainpages=false,
	linktocpage,
	colorlinks=true,
	citecolor=green!80!black,
	linkcolor=red!70!black,
	filecolor=magenta,
	urlcolor=magenta,
	breaklinks,
	pdfauthor={Martin Winter},
]{hyperref}
\usepackage{calc, mathtools} 

\usepackage{float}

\usepackage{amssymb} 
\usepackage{footnote}
\usepackage[dvipsnames]{xcolor}
\definecolor{DarkGreen}{rgb}{0,0.7,0}
\definecolor{DarkRed}{rgb}{0.825,0,0}
\definecolor{DarkOrange}{rgb}{1,0.55,0}
\definecolor{SteelBlue}{rgb}{0.27, 0.51, 0.71}
\definecolor{NiceBlue}{rgb}{0.15, 0.2, 0.75}
\definecolor{NiceBlueLight}{RGB}{249, 249, 253}
\definecolor{Lightbackgroundred}{RGB}{252, 247, 247}

\newcommand{\Def}[1]{\textcolor{NiceBlue!90}{\textit{#1}}}
\newcommand\cincludegraphics[2][]{\hspace*{-3mm}
\makecell[c]{~\\[-4.05mm]
\colorbox{white!00}{\includegraphics[#1,height=0.12\linewidth]{#2}}}\hspace*{-1mm}}

\usepackage{makecell}
\usepackage{pifont}
\newcommand{\yes}{\textcolor{DarkGreen}{\ding{51}}}
\newcommand{\no}{\textcolor{DarkRed}{\ding{55}}}

\def\yesBut{\textcolor{Gray}{\ding{51}${}^*$}}
\def\noBut{\textcolor{Gray}{\ding{55}${}^*$}}

\newcommand{\Ques}{\textcolor{DarkOrange}{$\textbf{?}$}}
\newcommand{\SymOne}{\textcolor{SteelBlue}{\ding{67}}}
\newcommand{\SymTwo}{\textcolor{SteelBlue}{\ding{115}}}
\newcommand{\SymThree}{\textcolor{SteelBlue}{\ding{95}}}

\newcommand{\Star}{\textcolor{orange}{\ding{81}}}

\usepackage{centernot} 
\usepackage{array} 
\usepackage{enumitem,moreenum} 

\usepackage[noadjust]{cite} 
\usepackage[nameinlink,capitalize,noabbrev]{cleveref}
\usepackage{nicefrac}

\usepackage[font=small,labelfont=bf]{caption}

\usepackage[
    colorinlistoftodos,
    backgroundcolor=yellow,
    textsize=footnotesize 
]{todonotes}

\usepackage{newfloat}

\DeclareFloatingEnvironment[
  fileext=lor,
  listname={List of Roadmaps},
  name=Roadmap,
  placement=htbp
]{roadmap}

\makeatletter
\renewcommand{\fnum@roadmap}{Roadmap}
\makeatother

\crefformat{roadmap}{#2roadmap#3}
\Crefformat{roadmap}{#2Roadmap#3}

\usepackage{accents}
\newlength{\dhatheight}

\newcommand*{\Scale}[2][4]{\scalebox{#1}{$#2$}}

\usepackage{tikz-cd}
\usetikzlibrary{fit,backgrounds}
\usetikzlibrary{decorations.markings}

\tikzset{
  myarrow/.style={-{Triangle[length=1.1mm,width=1.3mm]}}
}

\tikzset{
  over/.style={
    postaction={
      decorate,
      decoration={
        markings,
        mark=at position 0.67675 with {
          \draw[white,line width=2pt] (-1.4pt,0) -- (1.4pt,0);
        }
      }
    }
  }
}

\tikzset{
  over2/.style={
    postaction={
      decorate,
      decoration={
        markings,
        mark=at position 0.63 with {
          \draw[white, line width=2pt] (-1.4pt,0) -- (1.4pt,0);
        }
      }
    }
  }
}

\newcommand{\RR}{\mathbb{R}}    
                   
\newcommand{\bs}[1]{\boldsymbol{#1}}
\newcommand{\bsdot}[1]{\dot{\boldsymbol{#1}}}
\newcommand{\bsddot}[1]{\ddot{\boldsymbol{#1}}}

\newcommand{\tildebs}[1]{\tilde{\boldsymbol{{#1}}}}
\newcommand{\tilbs}[1]{\tildebs{#1}}

\newcommand{\hatij}{{\ensuremath{\kern1pt\widehat{\kern-1pt\imath\kern-2.3pt\jmath}}}}

\makeatletter
\newcommand{\subalign}[1]{%
  \vcenter{%
    \Let@ \restore@math@cr \default@tag
    \baselineskip\fontdimen10 \scriptfont\tw@
    \advance\baselineskip\fontdimen12 \scriptfont\tw@
    \lineskip\thr@@\fontdimen8 \scriptfont\thr@@
    \lineskiplimit\lineskip
    \ialign{\hfil$\m@th\scriptstyle##$&$\m@th\scriptstyle{}##$\hfil\crcr
      #1\crcr
    }%
  }%
}
\makeatother

\newcommand{\Tsymb}{\top}
\newcommand{\T}{^{\Tsymb}}

\def\:{\colon}

\newcommand{\labelstyle}[1]{\upshape(\textit{#1})}
\newcommand{\mylabel}{\labelstyle{\roman*}}
\newenvironment{myenumerate}{\begin{enumerate}[label=\mylabel]}{\end{enumerate}}

\def\nlspace{\nolinebreak\space}
\def\nls{\nlspace}

\newcommand{\freespace}{\kern.07em}

\newcommand{\customdoubleline}{\cline{1-3}\cline{5-7}
\multicolumn{7}{c}{} \\[-4.75mm]
\cline{1-3}\cline{5-7}}

\newcommand{\customdoublelinetwo}{\cline{1-3}\cline{5-7}
\multicolumn{5}{c}{} \\[-2.5mm]
\cline{1-3}\cline{5-7}}

\newcommand{\customsingletwo}{\cline{1-3}
\multicolumn{4}{c}{} \\[-2.5mm]
\cline{1-3}\cline{5-7}}

\usepackage{ragged2e}
\usepackage{amsmath,amsthm}

\numberwithin{equation}{section}

\newtheoremstyle{mythmstyle} 
    {\parsep}                    
    {\parsep}                    
    {\itshape}                   
    {}                           
    {\bfseries\scshape}          
    {.}                          
    {.5em}                       
    {}  
\newtheoremstyle{mydefstyle} 
    {\parsep}                    
    {\parsep}                    
    {}                   
    {}                           
    {\mdseries\scshape}          
    {.}                          
    {.5em}                       
    {}  

\theoremstyle{plain}
\newtheorem{theorem}{Theorem}[section]
\newtheorem{corollary}[theorem]{Corollary}
\newtheorem{lemma}[theorem]{Lemma}
\newtheorem{proposition}[theorem]{Proposition}

\theoremstyle{definition}
\newtheorem{definition}[theorem]{Definition}
\newtheorem{example}[theorem]{Example}
\newtheorem{remark}[theorem]{Remark}
\newtheorem{question}[theorem]{Question}

\crefname{theorem}{Theorem}{Theorems}
\crefname{proposition}{Proposition}{Propositions}
\crefname{lemma}{Lemma}{Lemmas}
\crefname{corollary}{Corollary}{Corollaries}
\crefname{remark}{Remark}{Remarks}
\crefname{example}{Example}{Examples}
\crefname{definition}{Definition}{Definitions}
\crefname{problem}{Problem}{Problems}
\crefname{observation}{Observation}{Observation}
\crefname{construction}{Construction}{Construction}

\theoremstyle{plain}
\newtheorem{innercustomgeneric}{\customgenericname}
\providecommand{\customgenericname}{}
\newcommand{\newcustomtheorem}[2]{%
  \newenvironment{#1}[1]
  {%
   \renewcommand\customgenericname{#2}%
   \renewcommand\theinnercustomgeneric{##1}%
   \innercustomgeneric
  }
  {\endinnercustomgeneric}
}

\newcustomtheorem{theoremX}{Theorem}
\newcustomtheorem{lemmaX}{Lemma}
\newcustomtheorem{conjectureX}{Conjecture}

\DeclareMathOperator{\conv}{conv}

\DeclareMathOperator{\OO}{O}

\DeclareMathOperator{\GL}{GL}

\DeclareMathOperator{\rank}{rank}
\DeclareMathOperator{\corank}{corank}

\DeclareMathOperator{\Int}{int}

\DeclareMathOperator{\Iso}{Iso}

\let\<=\langle
\let\>=\rangle
\let\x=\times

\def\...{...}
\newcommand{\shortStyle}{\text}
\newcommand{\ie}{\shortStyle{i.e.,}}

\newcommand{\eg}{\shortStyle{e.g.}}

\newcommand{\cf}{\shortStyle{cf.}}

\makeatletter
\renewcommand*{\eqref}[1]{%
  \hyperref[{#1}]{\textup{\tagform@{\ref*{#1}}}}%
}
\makeatother

\allowdisplaybreaks

\makeatletter
\newcommand{\addresseshere}{%
  \enddoc@text\let\enddoc@text\relax
}
\makeatother

%% file: sec/introduction.tex

\vspace*{-1em}
\section{Introduction}
\label{sec:introduction}

In \cite{himmelmannschulzewinter2025rigiditypolytopesedgelength} the authors initiated the study of convex polytopes as point-hyperplane frameworks and asked under what conditions a polytope admits a continuous deformation that preserves both its edge lengths and face coplanarities.
If such deformations exist, the polytope is said to be flexible, otherwise rigid.
Classical results in polyhedral rigidity (such as the rigidity theorems by Cauchy \cite{cauchy1813} or Dehn \cite{dehnstheorem}) enforce the stronger constraint of fixed face shapes, 
and then find that convexity generally implies rigidity.
This is in contrast to our setting, in which flexible convex polytopes do exist.
Still, only
few techniques for
their construction are
known.

\begin{figure}[h!] 
    \centering
\includegraphics[width=0.235\linewidth]{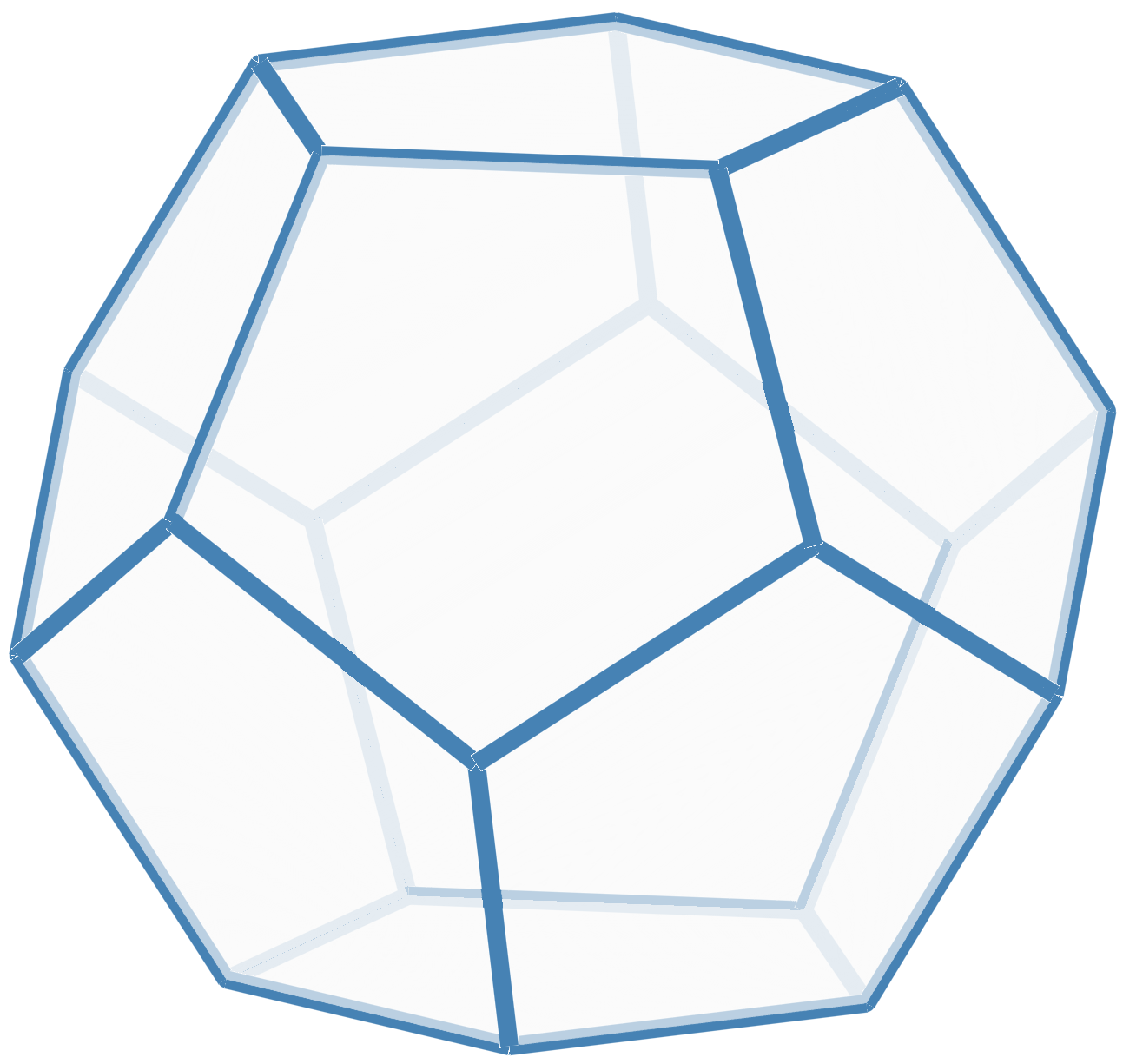}
\hspace{0.75em}
\includegraphics[width=0.235\linewidth]{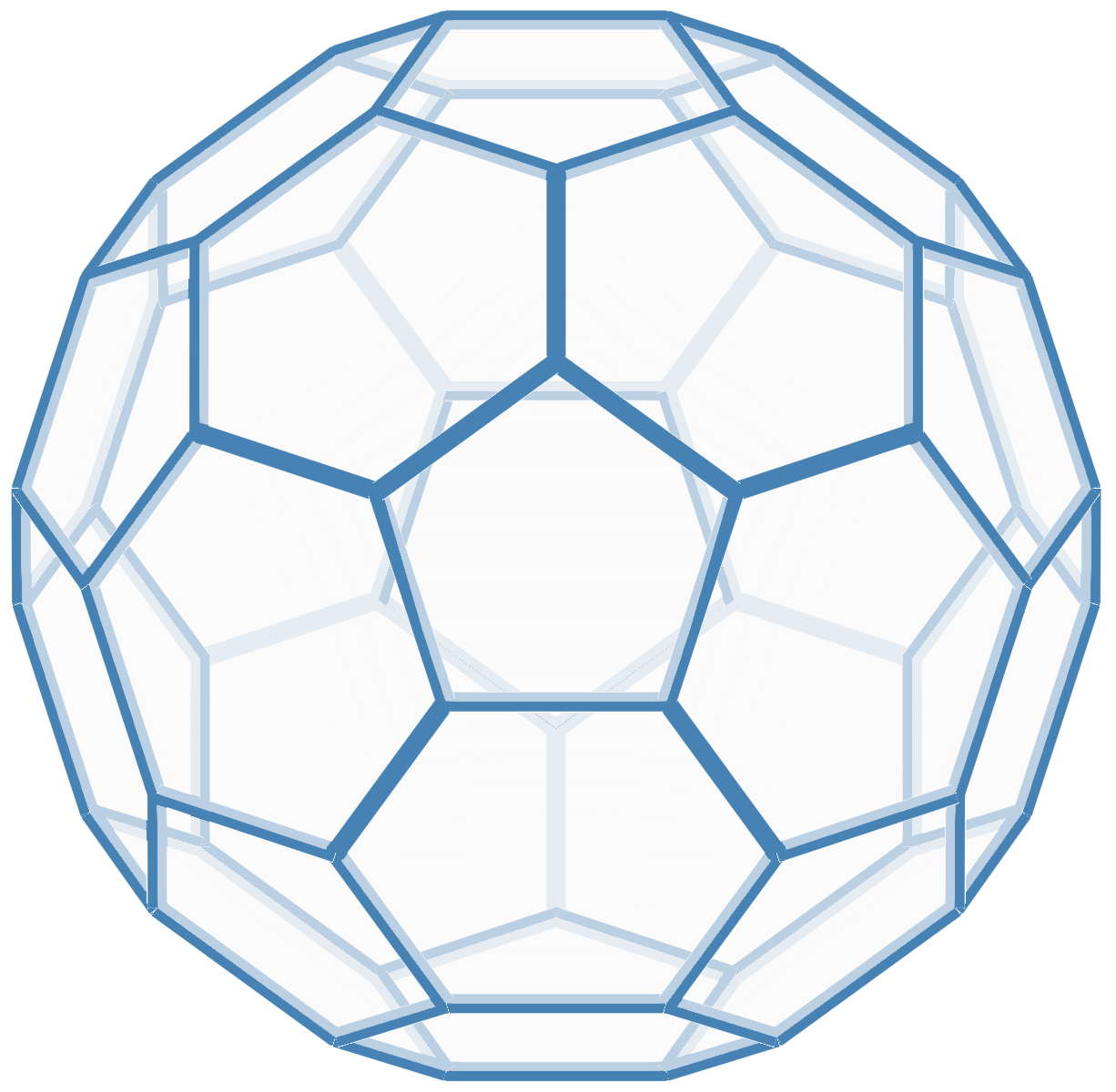}
\hspace{0.75em}
\includegraphics[width=0.2425\linewidth]{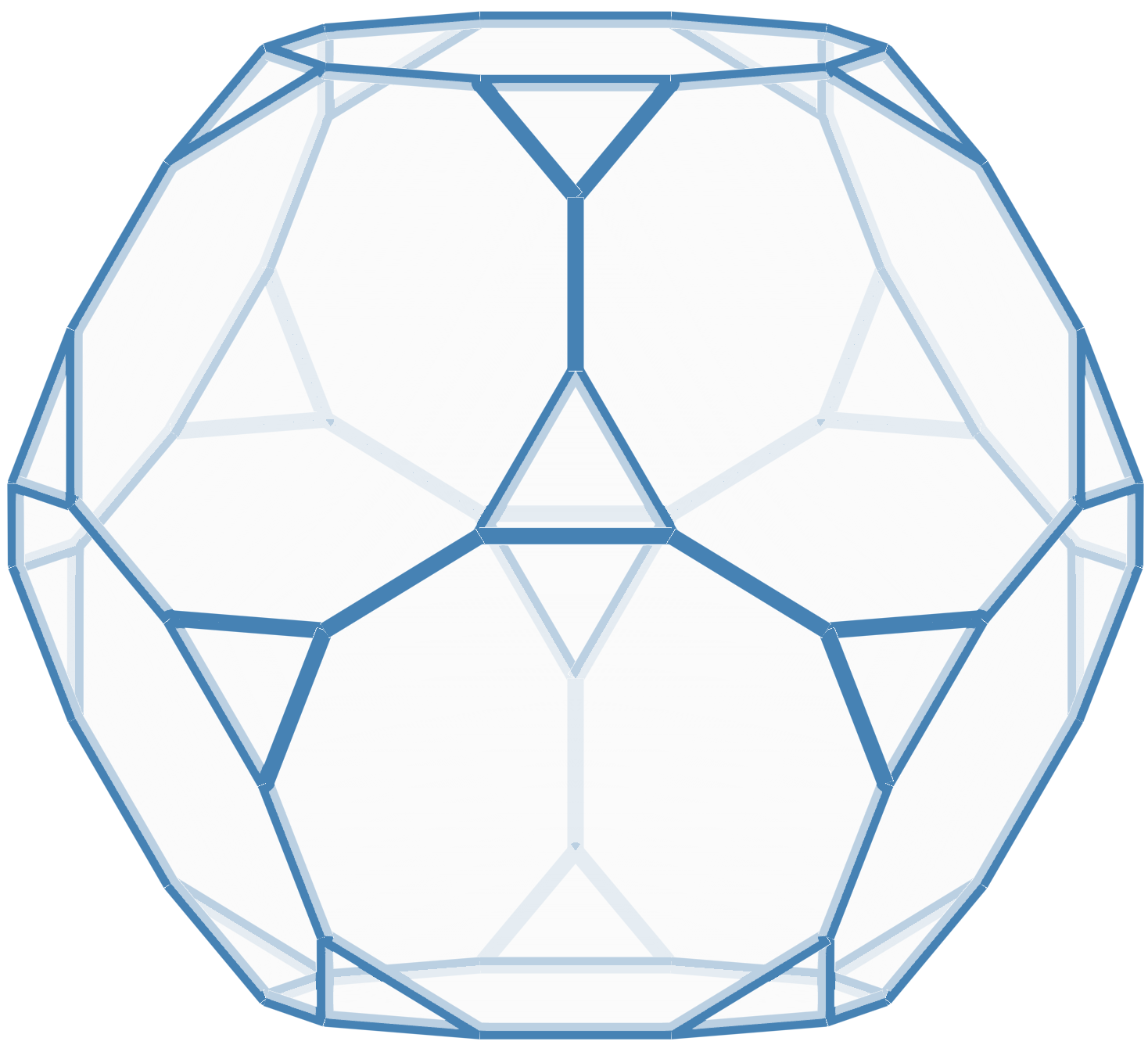}
    \caption{The regular dodecahedron (left), the truncated icosahedron (center) and the truncated dodecahedron (right) provide simple examples of polytopes with fascinating rigidity properties.
}
    \label{fig:Dod_TIco_TDod}
\end{figure}
As a consequence, the sentiment arose that this sort of flexibility is still rather exceptional, not only among combinatorial types of polytopes, but also among~realizations of a fixed type.
In \cite[Conjecture 1.1]{himmelmannschulzewinter2025rigiditypolytopesedgelength} the authors conjecture that~in~dimension at least three, a \textit{generic} convex polytope of a given combinatorial type~is rigid. This claim was shown to hold 
in dimension three \cite[Theorem 1.2]{himmelmannschulzewinter2025rigiditypolytopesedgelength}, but remains open for higher dimensions.

Analogous to the classical setting of bar-joint frameworks, deciding rigidity for a given polytope is computationally hard.
Once a polytope is not first-order rigid, the question of rigidity is often out of reach for direct techniques.
This already holds for polytopes with relatively few vertices and faces.
For example, in \cite[Section~6.2]{himmelmannschulzewinter2025rigiditypolytopesedgelength} the authors singled out the \textit{regular dodecahedron} (\Cref{fig:Dod_TIco_TDod}, left) as a guiding example of a relatively small polytope with unclear rigidity status.

In this article, we develop novel tools for the analysis of realization spaces and rigidity properties of polytopes. 
This includes a suitable second-order theory as~well as an effective algorithm for testing second-order rigidity of concrete examples.
For a demonstration, we analyze the rigidity properties of the Platonic, Archimedean and Catalan solids (see \cref{tab:comparison} and \cref{tab:catalan-comparison}). 
In particular, we decide the previously unknown rigidity properties of the \textit{regular dodecahedron} (\cref{fig:Dod_TIco_TDod}, left).
Solely the rigidity of the \emph{truncated dodecahedron} (\cref{fig:Dod_TIco_TDod}, right) cannot be determined by~our techniques, but seems to require techniques of even higher order. It will thus~serve as a simple test case for future theoretical and computational developments.
We find it remarkable that such varied rigidity behaviors can already be observed in classical examples.

This article is divided into theoretical, methodological, and experimental results,
of which we give an overview below (also refer to the \Cref{roadmap}).

\subsection{Theoretical results}
\label{sec:intro-theory}

In \cref{sec:combinatorics-of-polytopes} we recall the basics of polytope geometry, combinatorics, and rigidity from \cite{himmelmannschulzewinter2025rigiditypolytopesedgelength}.
In \cref{sec:first_order} we present the first-order theory. This includes definitions for the rigidity matrix, first-order motions and~flexes,
%
%
and equilibrium stresses.
We prove that first-order rigidity implies rigidity.
In \cref{sec:second-order-theory} we develop the second-order theory for polytope rigidity.
We provide analogues for the following notions and implications from framework rigidity:
$$\text{first-order rigid} \;\xRightarrow{\ref{res:prestress}}\; \text{prestress stable} \;\xRightarrow{\ref{res:prestress}}\; \text{second-order rigid} \;\xRightarrow{\ref{res:second_order}}\;
\text{rigid}.$$
We generalize the second-order stress test by Connelly and Whiteley \cite{connelly1996second} to~the~polytope setting (\cf\ \cref{res:sotest}).

\cref{section:edge-length-perturbations} conducts a local study of polyhedral realization spaces motivated by the following train of thought:
the classical Legendre-Steinitz theorem states that the realization space of a 3-dimensional polytope modulo isometries is homeomorphic to $\RR^{E}$, where $E$ is the number of edges \cite{rastanawi2021dimensions}.
This suggests that the shape of a 3-polytope can be controlled by one free and independent parameter per edge.\nls
If the polytope is flexible, however, these parameters cannot map to edge lengths directly, since there are different realizations with the same lengths. 
This suggests that the geometric freedoms somehow distribute between flexes and other deformations that do change edge length. The latter we call \emph{edge length perturbations}, and we explore their relations to the various notions of rigidity. 

\subsection{Methodological results}
\label{sec:intro-methods}
In \cref{section:second-order-rigidity-check}, we provide a concrete method based on real algebraic geometry for certifying second-order rigidity of polytopes. It relies on transforming the condition for second-order rigidity into a polynomial system of homogeneous equations, which can be symbolically solved using Gröbner bases. If this system has no non-trivial solutions, then the polytope is second-order rigid.

\subsection{Computational results}
\label{sec:intro-results}

Our newly developed first- and second-order rigidity certificates apply to any polytope.
For demonstration purposes, we benchmark them on the Platonic, Archimedean and Catalan solids for which combinatorial~data and exact vertex and normal coordinates are readily available.

In \cref{sec:so_rigidity_dod} we give a step-by-step account of how to apply our second-order~rigidity certificate to the regular dodecahedron. We thereby verify that it is indeed rigid (\cf\ \cref{dodecahedronrig}), answering a question posed in \cite[Section 6.2]{himmelmannschulzewinter2025rigiditypolytopesedgelength}.

Analogous computations are performed for the other test polytopes and are~documented in \cref{tab:comparison} and \cref{tab:catalan-comparison}.
We highlight two results: We show that, besides the dodecahedron, also the \textit{truncated icosahedron} (\cref{fig:Dod_TIco_TDod}, center) is not first- but second-order rigid. We also show that the \textit{truncated dodecahedron} (\cref{fig:Dod_TIco_TDod}, right) is \emph{not} second-order rigid, and thus its rigidity status remains open.
A detailed explanation of the tables and a discussion of the results follow in \cref{sec:experimental_results}.
Further theoretical background on the observed flexes and flex space dimensions is given in \cref{sec:dimension_estimates}.
All code and additional animations of the flexible polytopes are included in the Supplementary Material~\cite{zenodo_suppMaterial}.

Due to its previous status as the guiding example,
we dedicate \cref{sec:dodecahedron} to~another case study of the regular dodecahedron. 
We investigate the local behavior of its realization space with the goal of understanding the (potentially singular) nature of its very special configuration.
Using tools from Riemannian geometry and homotopy continuation,
we compute edge length perturbations and deformation paths (with details given in \cref{sec:approx_continuous_motions}).
We also conclude that the dodecahedron is not generically globally rigid.




{
\renewcommand{\thefigure}{}
\hypersetup{hidelinks}
\begin{roadmap}[h!]
    \centering
\begin{tikzpicture}[
    scale=0.725,
    transform shape,
    every node/.style={
        draw,
        rounded corners,
        minimum width=2.5cm,
        minimum height=0.9cm,
        align=center
    },
    >=stealth
]

\node[fill=white] (intro) at (21,1.65) {Intro};

\node[fill=white] (app) at (21,-1.65) {Appendix};

\node[fill=white] (c10) at (12,1.65) {\Cref{sec:dimension_estimates}};

\node[fill=white] (c3) at (7.5,1.65) {Chapter 3};

\node[fill=white] (c4) at (8.25,-0.825)  {Chapter 4};
\node[fill=white] (c5) at (12,-1.65) {Chapter 5};

\node[fill=white] (c6) at (12,1.65)  {Chapter 6};
\node[fill=white] (c8) at (16,-0.7) {Chapter 8};
\node[fill=white] (c7) at (16,0.825)  {Chapter 6};
\node[draw=none] (anchor1) at (16,1.65) {};
\node[draw=none] (anchor2) at (16,-1.65) {};

\node[fill=none] (c9) at (20.15,0) {\Cref{sec:so_rigidity_dod}};
\node[
    draw=NiceBlue!40,
    fill=NiceBlue!20,
    rounded corners,
    fit=(c3)(c4)(c5)(c6),
    inner sep=0.2cm,
    fill opacity=0.15,
    label=below:{Theory (\cf\ \Cref{sec:intro-theory})}
] {};
\node[
    draw=DarkRed!40,
    fill=DarkRed!20,
    rounded corners,
    fit=(anchor1)(anchor2),
    inner sep=0.2cm,
    fill opacity=0.15,
    label=below:{Methods (\cf\ \Cref{sec:intro-methods})}
] {};
\node[
    draw=DarkGreen!40,
    fill=DarkGreen!20,
    rounded corners,
    fit=(app)(intro)(c9),
    inner xsep=0.3cm,
    inner ysep=0.2cm,
    fill opacity=0.15,
    label=below:{Results (\cf\ \Cref{sec:intro-results})}
] {};

\node[fill=white] (tables) at (21,1.65) {\Cref{sec:archimedean}};

\node[fill=NiceBlueLight] (c10) at (12,1.65) {\Cref{sec:dimension_estimates}};

\node[fill=white] (app) at (21,-1.65) {\Cref{sec:dodecahedron}};

\node[fill=white] (c3) at (7.5,1.65) {\Cref{sec:notation}};

\node[fill=white] (c4) at (8.25,-0.825)  {\Cref{sec:first_order}};
\node[fill=white] (c5) at (12,-1.65) {\Cref{section:edge-length-perturbations}};

\node[fill=white] (c6) at (12,0)  {\Cref{sec:second_order}};
\node[fill=Lightbackgroundred] (c8) at (16,-0.7) {\Cref{sec:approx_continuous_motions}};
\node[fill=white] (c7) at (16,0.825)  {\Cref{section:second-order-rigidity-check}};

\node[fill=white] (c9) at (20.15,0) {\Cref{sec:so_rigidity_dod}};

\draw[myarrow] (c4) -- (c5);


\draw[myarrow] (c3) -- (c4);

\draw[myarrow] (c4) -- (c6);

\draw[myarrow] (c6) -- (c7);
\draw[myarrow] (c7) -- (tables);

\draw[myarrow] (c5) -- (app);

\draw[myarrow, dashed] (c7) -- (c9);
\draw[myarrow] (c8) -- (app);

\draw[myarrow, dashed] (c9) -- (tables);
\draw[myarrow, dashed] (c9) -- (app);
\draw[myarrow] (c10) -- (tables);

\end{tikzpicture}
\vspace*{-1em}
\caption{This article's overall structure. Necessary prerequisites and logical dependencies are indicated by arrows.}
\label{roadmap}
\vspace*{-2em}
\end{roadmap}
}

\tempnewpage

\section{Basic notions}
\label{sec:notation}
\label{sec:combinatorics-of-polytopes}

Throughout 
this article, let $P\subset\RR^d$ denote a $d$-dimensional \textit{convex}
polytope with non-empty interior and \Def{vertices} $p_1,...,p_n$, \ie\ $P$ is the~convex hull $\conv\{p_1,...,p_n\}$.
We often focus on polytopes in dimension three, which we call \Def{polyhedra}. 
Below~we recall the relevant facts about polytopes; for a general introduction we refer to \cite{ziegler2012lectures}.

\pagebreak
\subsection{Combinatorial types}

The combinatorics of a polytope can be described using its vertex-facet incidence structure: a \Def{combinatorial type}
$\mathcal P=(V,F,\sim)$ is~a triple consisting of an (abstract) \Def{vertex set} $V$, an (abstract) \Def{facet set} $F$ and~a~\Def{vertex-facet incidence relation} ${\sim}\subseteq V\times F$.
To each convex polytope $P\subset\RR^d$ we can assign such a combinatorial type in the obvious way.
Note that the edge graph is encoded within $\mathcal P$ too: we say that vertices $i,j\in V$ are \Def{adjacent} in $\mathcal P$, and denote this by $i\sim j$, if there are facets $\sigma_1,...,\sigma_r\in F$ so that $i$ and $j$ are the only two vertices incident to all of them.\nls
This defines a graph structure on $V$ which we call the \Def{edge graph} $G_{\mathcal P}=(V,E)$.
To emphasize the presence of this graph structure we shall denote a combinatorial type also as the quadruple $\mathcal P = (V,E,F,\sim)$.

\subsection{Polytopal realizations}

The polytopes that share a combinatorial type $\mathcal P$ form the \emph{polytopal realizations} of $\mathcal P$.
For computational convenience, a $d$-dimensional \Def{polytopal realization} of $\mathcal P$ will be given as a pair $(\bs p,\bs a)$ consisting of a \Def{vertex map} $\bs p\: V\to\RR^d$ and a \Def{facet normal map} $\bs a\: F\to\RR^d$, so that for each $\sigma\in F$ the points (or \Def{vertices}) $\{p_i\mid i\sim\sigma\}$ lie on a common affine hyperplane (the \Def{facet hyperplane}) with normal vector $a_\sigma$ (the \Def{facet normal}).
This can be expressed as
$$\<p_i,a_\sigma\>=1,\quad\text{whenever $i\sim \sigma$}.$$
We say that a realization is \Def{strictly convex} (or just \Def{convex} for short) if additionally
$$\<p_i,a_\sigma\><1,\quad\text{whenever $i\not\sim\sigma$}.$$
Each convex polytope $P$ with combinatorial type $\mathcal P$ and $0\in\Int(P)$ gives rise~to~a convex realization in this sense with the facet normals pointing outwards.
We~often identify a polytope with the corresponding realization and use $P$ and $(\bs p,\bs a)$~interchangeably.

All convex realizations of $\mathcal P$ taken together form its \Def{(convex) realization space}

\begin{align*}
\real(\mathcal P)
&:=\left\{\!\!
\begin{array}{l}
    \bs p\:V\to\RR^d
    \\
    \bs a\:F\to\RR^d
\end{array}
\Bigg\vert
    \begin{array}{rcl}
    \<p_i,a_\sigma\> &\!\!\!\!=\!\!\!\!\!& 1\text{ whenever $i\sim \sigma$}\,\\
    \<p_i,a_\sigma\> &\!\!\!\!<\!\!\!\!\!& 1\text{ whenever $i\not\sim\sigma$}\,
    \end{array}
\!\right\},
\end{align*}
With these definitions, $\real(\mathcal P)$ is a semi-algebraic set in $\RR^{dV}\oplus\RR^{dF}$\!. 
Note 
that~we use certain set variables, such as $V$, $E$ and $F$, to also stand in for their~cardinalities.

\begin{remark}\label{res:new_notion}
    The realization spaces defined in \cite[Section 2]{himmelmannschulzewinter2025rigiditypolytopesedgelength} are slightly different from the ones used here.
    Our constraint $\<p_i,a_\sigma\>=1$ for $i\sim \sigma$ prevents facet~hyperplanes from passing through the origin.
    This formulation is more convenient for our purpose.
    The corresponding constraint from \cite{himmelmannschulzewinter2025rigiditypolytopesedgelength} is $\<p_i-p_j,a_\sigma\> = 0$ for $i,j\sim\sigma$ and does not entail this restriction.
    Both definitions lead to spaces that are birationally equivalent in a neighborhood of each convex realization that has the origin in its interior.
    Since our analysis will be primarily local and since every affinely-spanning convex polytope can be translated so that it contains the origin in its interior, this distinction can be ignored.
\end{remark}

The \Def{Lie group of Euclidean isometries $\Iso$} 
(\ie\ rotations, translations and reflec\-tions) acts on $\real(\mathcal P)$.
Factoring it out yields the \Def{reduced~realization space}:
\begin{align*}
\realR(\mathcal P)
    :=\realSp(\mathcal{P})\,/\,\Iso . 
\end{align*}


\subsection{Polytope rigidity}
\label{sec:rigidity}

We examine continuous deformations of $P$ that preserve edge lengths, analogous to bar-joint framework rigidity. Unlike in the framework setting, these motions~must also preserve the polytope's structure, that is, facet planarities and the combinatorial type.
Below we make this precise.

    Two realizations $(\bs p,\bs a),(\tilbs p,\tilbs a)\in\real(\mathcal P)$ are 
\begin{align*}
    \text{\Def{equivalent} if } \|p_j-p_i\|&=\|\tilde p_j-\tilde p_i\| \text{ for each edge $ij\in E$},
    \\
    \text{\Def{congruent} if } \|p_j-p_i\|&=\|\tilde p_j-\tilde p_i\|\text{  for each pair $i,j\in V$}.
\end{align*}
A \Def{(continuous) motion} of $(\bs p,\bs a)$ is a continuous map $(\bs p^t,\bs a^t)\:[0,1]\to\real(\mathcal P)$~so that $(\bs p^0,\bs a^0)=(\bs p,\bs a)$ and $(\bs p^t,\bs a^t)$ is equivalent to $(\bs p,\bs a)$ for all $t\in[0,1]$.

We define the \Def{metric realization space}
$$\mreal(\mathcal P, P) := \big\{(\bs p, \bs a) \in \real(\mathcal P)\mid \text{$(\bs p,\bs a)$ is equivalent to $P$}\big\},$$
which contains all realizations of the polytope $\mathcal{P}$, whose edge lengths are prescribed according to the realization $P$. 
The equivalence of edge lengths can be expressed as $\smash{\|p_i-p_j\|^2=\|\tilde p_i-\tilde p_j\|^2}$, that is, by a system of quadratic polynomial equations. Therefore, $\mreal(\mathcal{P},P)$ also is a semi-algebraic set in $\mathbb{R}^{dV}\oplus \mathbb{R}^{dF}$. 
A motion can~also be understood as a continuous curve $\bs p^t\:[0,1]\to\mreal(\mathcal P,P)$.

A motion is \Def{trivial} if $(\bs p^t,\bs a^t)$ is congruent to $(\bs p,\bs a)$ for all $t\in[0,1]$.
If non-trivial, we call it a \Def{flex} of $P$.
If a polytope has a flex, we call it \Def{flexible}, otherwise \Def{rigid}.
As before, the Lie group of Euclidean isometries $\Iso$ acts on $\mreal(\mathcal P, P)$ and can be factored out to obtain the \Def{reduced metric realization space}
$$\mreal^*(\mathcal P, P) = \mreal(\mathcal P,P)/\Iso.$$
A flex can be understood as a non-constant curve $\bs p^t:[0,1]\to\mreal^*(\mathcal P,P)$.
For constructions of flexible polytopes, see \cref{sec:flexible_polytopes} or \cite[Section 4]{himmelmannschulzewinter2025rigiditypolytopesedgelength}.

\begin{remark}
Evidently, these definitions apply verbatim to the more general class of point-hyperplane frameworks \cite{EJNSTW}.
Consequently, much of what we say~here applies more generally.
Caution is necessary when we assume certain sub-configurations to be affinely spanning, which is natural in the polytope setting, but does not apply to general point-hyperplane frameworks.    
\end{remark}

\tempnewpage 

\section{First-order theory of polytope rigidity}
\label{sec:first_order}

The first-order theory of a geometric constraint system studies the deformations that preserve the constraints up to first order. 
It is a linear theory and hence~provides heuristics and criteria that can be efficiently checked in practice.

Given a polytope $P=(\bs p,\bs a)\subset\RR^d$, its \Def{rigidity matrix} $\mathcal R(P)=\mathcal R(\bs p,\bs a)$ has the following structure:

\begin{figure}[h!]
    \centering \includegraphics[width=0.97\textwidth]{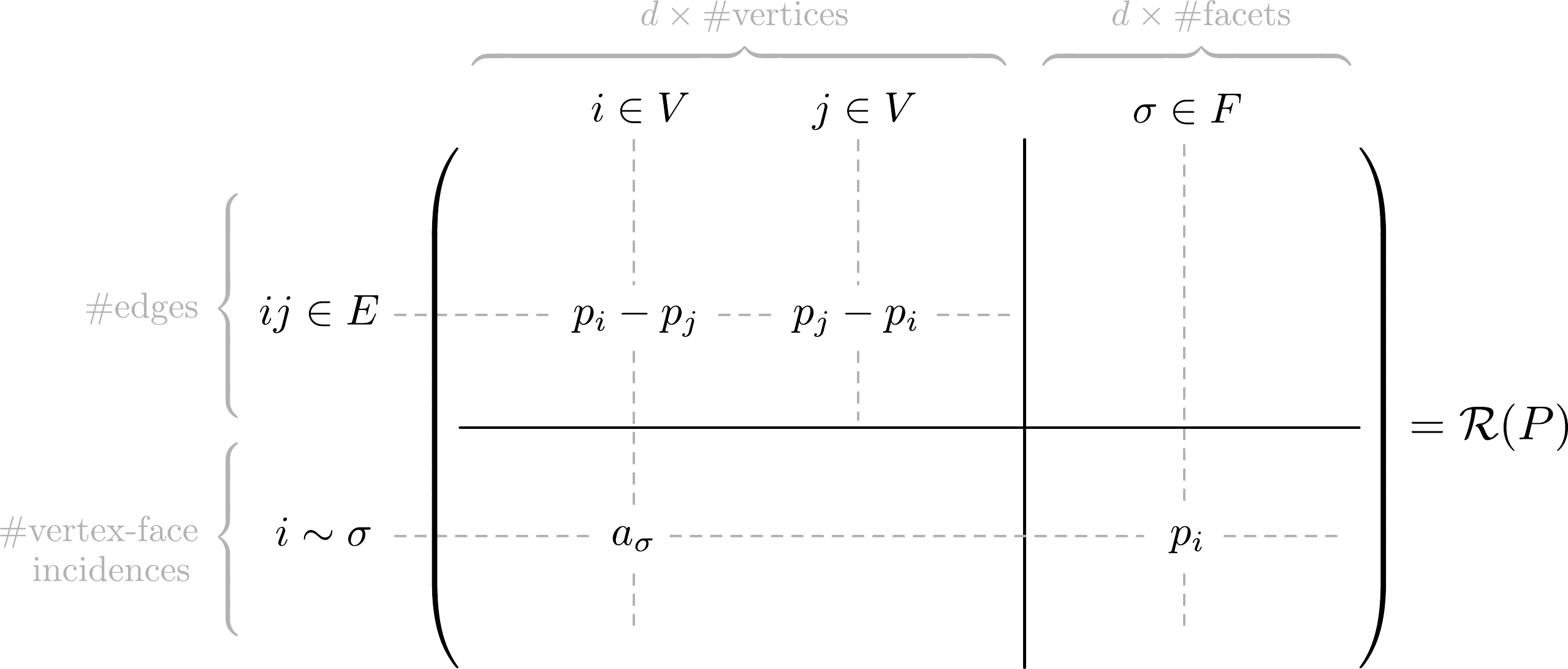}
\end{figure}

\noindent
All displayed entries are to be interpreted as $1\times d$ sub-matrices. 
The dimensions~of $\mathcal R(P)$ are therefore $(E+\VF)\times(dV+dF)$, where \Def{$\VF$} denotes the set of incident vertex-facet pairs. The rigidity matrix can be interpreted as (a scaled~version~of)~the Jacobian corresponding to the polynomial system of equations that define the metric realization space $\mreal(\mathcal{P},P)$.

The first-order theory of rigidity is essentially the study of the linear algebra~pro\-perties of $\mathcal R(P)$.
All first-order notions can be derived from the rigidity~matrix.

\subsection{First-order rigidity}
\label{sec:first-order-theory}

The elements of the kernel of $\mathcal R(P)$ are called \Def{first-order motions} of $P$. We denote them as pairs $\smash{\dot P=(\bsdot p,\bsdot a)}$ consisting of maps $\bsdot p\:V\to\RR^d$ and $\bsdot a\:F\to\RR^d$ that satisfy the following equations:
\begin{align}
    \<p_i-p_j,\,\dot p_i-\dot p_j\> &= 0 ,\quad\text{whenever $ij\in E$},
    \label{eq:flex_at_vertex}
    \\
    \<p_i,\dot a_\sigma\> + \<\dot p_i, a_\sigma\> &= 0 ,\quad\text{whenever $i\sim\sigma$}.
    \label{eq:flex_at_vertex_face}
\end{align}
Similar to the framework setting, the interpretation is that $\bsdot p$ and $\bsdot a$ are infinitesimal changes in the values of $\bs p$ and $\bs a$. In particular, \eqref{eq:flex_at_vertex} and \eqref{eq:flex_at_vertex_face} can be obtained~by differentiating the equations for edge length and coplanarity constraints.
Note~also that $\bsdot a$ is already determined by $\bsdot p$ via \eqref{eq:flex_at_vertex_face} (for details see the proof of \cref{res:trivial_dim}).

First-order motions are called \Def{trivial} if they satisfy
\begin{align}
    \<p_i-p_j,\,\dot p_i-\dot p_j\> &= 0 ,\quad\text{for any $i,j\in V$},
    \label{eq:trivial_motion_vertex}
\end{align}
Separate constraints on $\bsdot a$ are not necessary since they are already~determined via \eqref{eq:flex_at_vertex_face}.
Non-trivial first-order motions are called \Def{first-order flexes}. If a \mbox{first-order flex} exists, the polytope is called \Def{first-order flexible}, otherwise \Def{first-order rigid}.

The trivial first-order motions of $P$ form a linear subspace $\mathcal T_P\subseteq\ker \mathcal R(P)$.\nls
The first-order flexes do not form a linear space, though for matters of dimension counting we can consider the quotient $\ker \mathcal R(P)/\mathcal T_P$. We shall call it the \Def{first-order flex space} of $P$.
This space is trivial if and only if $P$ is first-order rigid.
The dimensions of the first-order flex space for the Platonic, Archimedean and Catalan solids are included in \cref{tab:comparison} and \cref{tab:catalan-comparison}.

\begin{lemma}
    \label{res:trivial_dim}
    $\dim\mathcal T_P=\binom{d+1}{2}$.
    In particular, $\corank \mathcal{R}(P)\geq \binom{d+1}{2}$.
\end{lemma}
\begin{proof}
    The equations 
    \eqref{eq:trivial_motion_vertex} are precisely the constraints of trivial first-order motions for bar-joint frameworks.
    In this setting, trivial first-order motions have the form $\bsdot p = S \bs p + t$ for a \mbox{skew-sym}\-metric matrix $S\in\RR^{d\x d}$ and $t\in\RR^d$, forming a vector space $\mathcal T_P'$ of dimension $\binom{d+1}{2}$ \cite{Asimow1978TheRO}.
    Note that this assumes that $\bs p$ is affinely spanning, which follows from our assumption that $P$ has nonempty interior.
    
    It remains to verify that each $\bsdot p\in\mathcal T_P'$ extends to a polytopal first-order motion of $P$.
    For this, fix $\sigma\in F$ and choose a basis $p_{i_1},...,p_{i_k}$ among the $p_i,i\sim\sigma$.\nls
    This~is possible because the face spans a $(d-1)$-dimensional affine subspace.
    Let~$\dot a_\sigma$~be~the unique solution of the full-rank linear system $\<\dot a_\sigma, p_{i_k}\> = -\<a_\sigma,\dot p_{i_k}\>$~where $k$ ranges over $\{1,...,d\}$. This ensures \eqref{eq:flex_at_vertex_face} for the $p_{i_k}$, and it remains to verify \eqref{eq:flex_at_vertex_face} for a general $p_i,i\sim\sigma$.
    Since the $p_i,i\sim\sigma$ are coplanar, we can write $p_i = \sum_{k}\alpha_k p_{i_k}$ for coefficients $\alpha_k$ that satisfy $\sum_k \alpha_k=1$. Using $\dot p_i = S p_i + t$, we obtain
    \begin{align*}
        \<\dot a_\sigma, p_i\> 
            &= \Big\< \dot a_\sigma, \sum_k \alpha_k p_{i_k}\Big\>
            = \sum_k \alpha_k \< \dot a_\sigma, p_{i_k}\>
            = -\sum_k \alpha_k \< a_\sigma, \dot p_{i_k}\>
            \\&= -\sum_k \alpha_k \< a_\sigma, Sp_{i_k}+t\>
            = - \Big\< a_\sigma, S\sum_k \alpha_k p_{i_k} + t\sum_k \alpha_k\Big\>
            \\&= - \< a_\sigma, Sp_i + t\>
            = - \< a_\sigma, \dot p_i\>.
    \end{align*}
    We conclude $\mathcal T_P\simeq\mathcal T_P'$, and the claim follows.
\end{proof}

\begin{corollary}
    \label{res:corank_rigid}
    A polytope is first-order
    rigid if and only if $\corank \mathcal R(P) = \binom{d+1}{2}$
\end{corollary}
\begin{proof}
    Let $\mathcal T_P\subseteq\ker\mathcal R(P)$ be the space of trivial first-order motions of $P$.
    Then
    \[ \text{$P$ is first-order rigid}\; \overset{\text{def}}\Longleftrightarrow\; \ker\mathcal R(P)=\mathcal T_P \;\overset{\ref{res:trivial_dim}}\Longleftrightarrow\; \corank \mathcal R(P)=\textstyle\binom{d+1}{2}.     \qedhere
\]
\end{proof}

First-order rigidity provides an efficiently computable certificate for rigidity:

\begin{theorem}
    \label{res:first_order}
    If a polytope is first-order rigid, then it is rigid.
\end{theorem}

This is a consequence of the second-order results \cref{res:second_order} and \cref{res:prestress}.
Below we give an independent proof to highlight the similarities between the first- and second-order theory (in particular, compare to the proof of \cref{res:second_order}).

\begin{proof}[Proof of \cref{res:first_order}]
    We prove the contrapositive: we assume that $P$ is flexible, and show that it is first-order flexible.
    For this, consider the algebraic subset $M\subseteq \mreal(\mathcal P,P)$ where we restrict sufficiently many vertices of $P$ to subspaces so as to precisely cancel the trivial motions. 
    If $P$ is flexible, then $P$ is not an isolated point of $M$.
    Hence, using the \emph{curve selection lemma} (see \eg\ \cite[Lemma 18.3]{algebraicapproximationcurves}) we obtain a non-constant analytic path $P^t$ in $M$ that ends at $P$.
    In other words, $P^t$ is a flex of $P$ that has a Taylor series expansion that converges to $P^t$.

    By adding trivial motions to $P^t$ we can cancel some of the lower order terms~of the series expansion.
    Since $P^t$ is non-trivial and analytic, there is a largest~integer $k$ so that we can simultaneously cancel all orders $1,...,k-1$.

    We now consider the $k$-th derivatives of the first-order motion constraints. We write $(\frac{\mathrm d}{\mathrm dt})^m_{t=0}$ for the operator that takes the $m$-th derivative and evaluates at $t=0$.
    For the edge length constraint \eqref{eq:flex_at_vertex} at $ij\in E$ we obtain
    \begin{align*}
        0 
        = 
        \Big(\frac{\mathrm d}{\mathrm dt}\Big)^{\!k}_{\!t=0}\, \|p_i-p_j\|^2 &= 
        \sum_{\ell=0}^k\binom{k}{\ell} \Big\<\Bigl(\frac{\mathrm d}{\mathrm dt}\Bigr)^{\!\ell}_{\!t=0}(p_i-p_j), \Bigl(\frac{\mathrm d}{\mathrm dt}\Bigr)^{\!k-\ell}_{\!t=0}(p_i-p_j)\Big\>
        \\&\overset{\mathclap{(*)}}=
        2\Big\<p_i-p_j, \Bigl(\frac{\mathrm d}{\mathrm dt}\Bigr)^{\!k}_{\!t=0}(p_i-p_j)\Big\>,
    \end{align*}
    where in $(*)$ we use that most terms are zero by the choice of $k$.
    Analogously, for the vertex-facet incidence constraint \eqref{eq:flex_at_vertex_face} at $i\sim \sigma$ we obtain
    \begin{align*}
        0 
        = 
        \Big(\frac{\mathrm d}{\mathrm dt}\Big)^{\!k}_{\!t=0} \<p_i,a_\sigma\>
        &= 
        \sum_{\ell=0}^k\binom{k}{\ell} \Big\< \Bigl(\frac{\mathrm d}{\mathrm dt}\Bigr)^{\!\ell}_{\!t=0}\, p_i, \Bigl(\frac{\mathrm d}{\mathrm dt}\Bigr)^{\!k-\ell}_{\!t=0} a_\sigma\Big\>
        \\&=
        \Big\<p_i, \Bigl(\frac{\mathrm d}{\mathrm dt}\Bigr)^{\!k}_{\!t=0}a_\sigma\Big\>+\Big\<\Bigl(\frac{\mathrm d}{\mathrm dt}\Bigr)^{\!k}_{\!t=0}\,p_i, a_\sigma\Big\>.
    \end{align*}
    Observe that these two identities precisely express that $\dot P := (\tfrac{\mathrm d}{\mathrm dt})^k_{t=0} P$ satisfies the equations of a first-order motion.
    It remains to show that $\smash{\dot P}$ is non-trivial. 
    In fact, if $\dot P$ were trivial, then $\dot P= \tfrac{\mathrm d}{\mathrm dt} Q$ for some trivial motion $Q^t$.
    One can check that then the reparameterization $\hat Q^t:= Q^{t^k/k!}$ is still trivial and cancels the $k$-th order term of $P^t$, in contradiction to the choice of $k$.
\end{proof}

\begin{remark}
    \label{res:first_order_no_reverse}
    The converse of \cref{res:first_order} does not hold in general.
    For example, the regular dodecahedron is first-order flexible (it has a 5-dimensional space of~first-order flexes), but is second-order rigid, and hence rigid (\cf\ \cref{tab:comparison}).    
\end{remark}

\subsection{Equilibrium stresses}

The elements of the cokernel of the  rigidity matrix $\mathcal R(P)$ (\ie\ the kernel of $\mathcal R(P)\T$) are called the \Def{equilibrium stresses} (or just \Def{stresses}) of a system.
Stresses form the second, and equally important, first-order notion of rigidity theory.
We denote a stress by a pair $(\bs\omega,\bs\eta)$ consisting of maps $\bs\omega\:E\to\RR$ and $\bs\eta\:\VF\to\RR$. 
They satisfy the following two \Def{stress equilibrium conditions}:
\begin{align}
    \sum_{\mathclap{j:ij\in E}} \omega_{ij}(p_i-p_j) + \sum_{\mathclap{\sigma:i\sim\sigma}} \eta_{i\sigma} a_
    \sigma&= 0,\quad\text{for all $i\in V$}, 
    \label{eq:stress_at_vertex}
    \\
    \sum_{\mathclap{i:i\sim\sigma}} \eta_{i\sigma} p_i &= 0,\quad\text{for all $\sigma\in F$}.
    \label{eq:stress_at_face}
\end{align}
We write $\bs\xi=({\bs \omega}, {\bs \eta})$ for sake of brevity and call it a \Def{polytope stress}.

\begin{figure}[h!]
    \centering
    \includegraphics[width=0.7\linewidth]{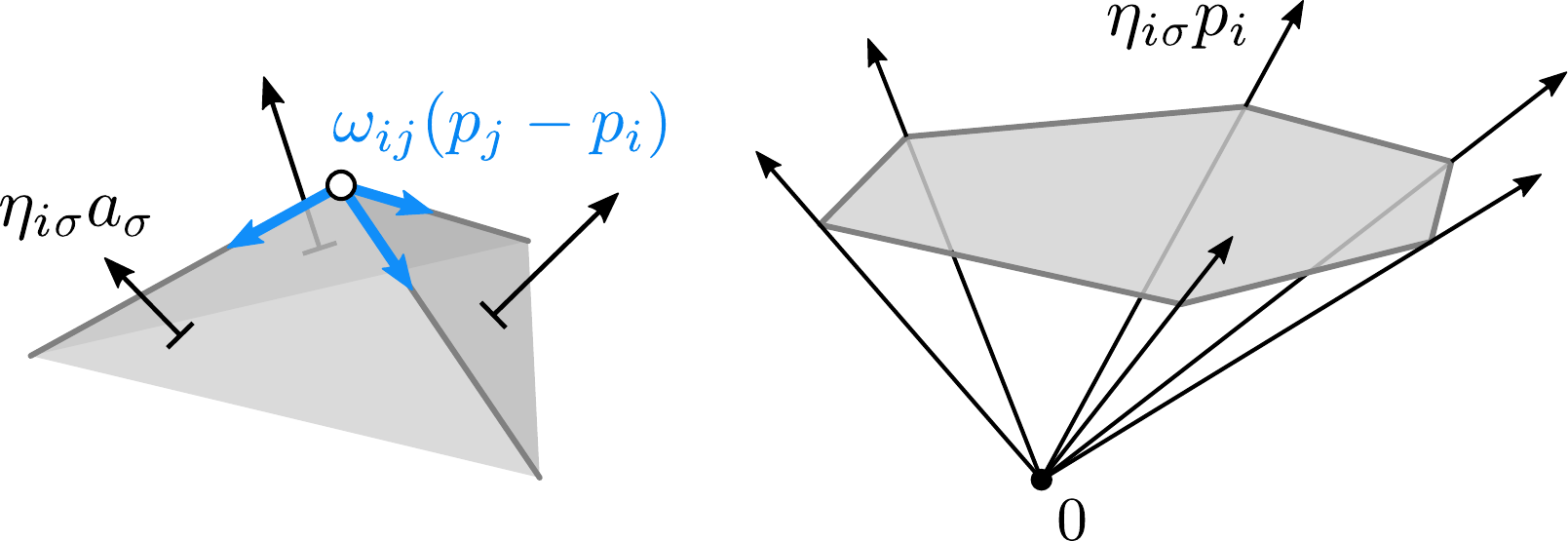}
    \caption{Visualization of a polytope stress $(\bs\omega,\bs\eta)$.}
    \label{fig:stress_on_face}
\end{figure}

\begin{remark}
\label{remark:stresses}
    Stresses in frameworks can be given a physical interpretation: each edge of a framework is thought of as a Hookean spring with a spring constant given
    by the stress $\omega_{ij}$ at $ij\in E$. The spring constant determines the magnitude of the force with which the edge pushes or pulls on its end vertices along the direction of the edge. Whether it pulls or pushes is determined by the sign of $\omega_{ij}$. The stress equilibrium condition (\ie\ \eqref{eq:stress_at_vertex} with second term removed)
    states that the forces that act on a vertex $i\in V$ cancel each other out.

    The spring constant interpretation persists for the $\bs\omega$-part of the polytope stress $(\bs\omega,\bs\eta)$, but now they do not need to cancel out at the vertices on their own.
    Instead, they need to add up to be equal and opposite to the second term -- the $\bs\eta$-part of the stress.
    This is complicated by the fact that $\bs\eta$ contributes to two equilibrium conditions, \eqref{eq:stress_at_vertex} and \eqref{eq:stress_at_face}, for which we try to provide interpretations below.
     
    First, the right term in \eqref{eq:stress_at_vertex} can be thought of as forces pushing a vertex along the normals of its incident facets (see \cref{fig:stress_on_face}, left).
    One may think of this term as a force resulting from pressure from within the polytope, though this analogy is mostly non-physical: the forces are not proportional to the volumes of the facets and may even have varying signs from facet to facet.

    Second, the stress condition \eqref{eq:stress_at_face} for a facet $\sigma\in F$ can be interpreted as a force equilibrium at the origin: each vertex $i\sim\sigma$ pushes or pulls the origin along their connecting line (\cf\ \cref{fig:stress_on_face}, right).
\end{remark}

Like first-order motions, stresses form a linear space -- the \Def{stress space}.
Stresses are particularly relevant in second-order theory (see \cref{sec:second_order}), but also play a special role in the first-order study of polyhedra (\ie\ 3-dimensional polytopes).

\begin{proposition}
    \label{res:flex_eq_stress}
    If $P \subset\RR^3$ is a polyhedron, then the dimension of its first-order flex space equals the dimension of the space of equilibrium stresses.
    In particular, a polyhedron is first-order rigid if and only if it has no non-zero stress.
\end{proposition}

\begin{proof}

    For each vertex, the number of incident faces equals the number of~incident edges. Thus, the total number of vertex-face incidences in $P$ (denoted by $\VF$)~equals the total number of vertex-edge incidences, which (by the handshaking lemma) is equal to $2E$.
    Using Euler's polyhedral formula $V-E+F=2$, we find the following relation between the number of rows and columns of the rigidity matrix $\mathcal{R}(P)$:
    \begin{align*}
        (*)\quad\text{\#rows}=\VF+E=2E+E=3E&=3(V+F-2)
        \\&=3V+3F-6=\text{\#columns}-6.
    \end{align*}
    The claim now follows from
    \[
        \dim(\text{first-order flexes}) \overset{\ref{res:trivial_dim}}= \corank\mathcal R(P) - 6 \overset{(*)}= \corank \mathcal R(P)\T = \dim(\text{stresses}).
    \qedhere \]
\end{proof}


The advantage of checking first-order rigidity via stresses is that stresses have~no analogue of the trivial first-order flexes that have to be treated separately.
This~fact is important in Dehn's proof of first-order rigidity for simplicial polytopes \cite{dehnstheorem}, and with \cref{res:flex_eq_stress} we highlight that it holds for non-simplicial polytopes as well.

\section{Second-order theory}
\label{sec:second_order}
\label{sec:second-order-theory}

Second-order rigidity theory investigates deformations that preserve constraints up to second order.
It is mainly employed for the analysis of systems that fail the first-order rigidity test.
For that, we follow the work of Connelly and Whiteley \cite{connelly1996second}. 

\subsection{Second-order rigidity}

A tuple $(\dot P,\ddot P) = (\bsdot p,\bsdot a;\bsddot p,\bsddot a)$ of maps $\bsdot p,\bsddot p\:V\to\RR^d$ and $\bsdot a,\bsddot a\:F\to\RR^d$ is called a \Def{second-order motion} of $P$ if $(\bsdot p,\bsdot a)$ forms a first-order motion and additionally the following equations hold:
\begin{align}
    \label{eq1}
    \<p_i-p_j,\,\ddot{p}_i-\ddot{p}_j\>+\<\dot{p}_i-\dot{p}_j,\,\dot{p}_i-\dot{p}_j\> & =0, \quad\text{whenever $ij\in E$},
    \\
    \label{eq2}
    \<p_i,\ddot{a}_\sigma\>+2\< \dot{p}_i,\dot{a}_\sigma\>+\<\ddot p_i,a_\sigma\> &= 0,\quad\text{whenever $i\sim \sigma$}.
\end{align}
We call a second-order motion \Def{trivial} if $(\bsdot p,\bsdot a)$ is trivial as a first-order motion. 
A non-trivial second-order motion is called a \Def{second-order flex}.
If $P$ has some second-order flex, then it is called \Def{second-order flexible}, and \Def{second-order rigid} otherwise.

Second-order rigidity provides a certificate for rigidity:

\begin{theorem}         \label{res:second_order_rigid_impl_rigid}
    \label{res:second_order}
    \label{res:second-order-implies-rigid}
    If a polytope is second-order rigid, then it is rigid. 
\end{theorem}
\begin{proof}
    The proof is largely analogous to the proof of \cref{res:first_order}. 
    As before, we prove the contrapositive: we assume that $P$ is flexible and show that it is second-order flexible.
    Through restricting vertices to suitable subspaces and invoking the curve selection lemma \cite[Lemma 18.3]{algebraicapproximationcurves}, we obtain an analytic flex $P^t$.
    There is then a largest integer $k$ so that adding a suitably chosen trivial motion simultaneously cancels all terms of order $1,...,k-1$ in the Taylor expansion of $P^t$.

    This time we consider the $2k$-th derivatives of the first-order motion constraints.
    For the edge length constraint \eqref{eq:flex_at_vertex} at $ij\in E$ we obtain
    \begin{align*}
        0 
        &= 
        \Big(\frac{\mathrm d}{\mathrm dt}\Big)^{\!2k}_{\!t=0} \|p_i-p_j\|^2 
        = 
        \sum_{\ell=0}^{2k}\binom{2k}{\ell} \Big\<\Bigl(\frac{\mathrm d}{\mathrm dt}\Bigr)^{\!\ell}_{\!t=0}(p_i-p_j), \,\Bigl(\frac{\mathrm d}{\mathrm dt}\Bigr)^{\!2k-\ell}_{\!t=0}(p_i-p_j)\Big\>
        \\&\overset{\mathclap{(*)}}=\;
        2\Big\<p_i-p_j,\, \Bigl(\frac{\mathrm d}{\mathrm dt}\Bigr)^{\!2k}_{\!t=0}(p_i-p_j)\Big>
            + \binom{2k}{k} \Big\<\Bigl(\frac{\mathrm d}{\mathrm dt}\Bigr)^{\!k}_{\!t=0}(p_i-p_j),\, \Bigl(\frac{\mathrm d}{\mathrm dt}\Bigr)^{\!k}_{\!t=0}(p_i-p_j)\Big\>
    \end{align*}
    where in $(*)$ we use that most terms are zero.
    Analogously, for the vertex-facet~incidence constraint \eqref{eq:flex_at_vertex_face} at $i\sim \sigma$ we obtain
    \begin{align*}
        0 
        &= 
        \Big(\frac{\mathrm d}{\mathrm dt}\Big)^{\!2k}_{\!t=0} \<p_i,a_\sigma\>
        = 
        \sum_{\ell=0}^{2k}\binom{2k}{\ell} \Big\< \Bigl(\frac{\mathrm d}{\mathrm dt}\Bigr)^{\!\ell}_{\!t=0} \,p_i,\, \Bigl(\frac{\mathrm d}{\mathrm dt}\Bigr)^{\!2k-\ell}_{\!t=0} a_\sigma\Big\>
        \\&=
        \Big<p_i,\, \Bigl(\frac{\mathrm d}{\mathrm dt}\Bigr)^{\!2k}\!\!a_k\Big\>
            + \Big\<\Bigl(\frac{\mathrm d}{\mathrm dt}\Bigr)^{\!2k}_{\!t=0} \,p_i,\, a_k\Big\>
            + \binom{2k}{k}\Big\<\Bigl(\frac{\mathrm d}{\mathrm dt}\Bigr)^{\!k}_{\!t=0} \,p_i,\,
                \Bigl(\frac{\mathrm d}{\mathrm dt}\Bigr)^{\!k}_{\!t=0} a_k\Big\>.
    \end{align*}
    Observe that these two identities express that
    $$\dot P := \Big(\frac{\mathrm d}{\mathrm dt}\Big)^k_{t=0} P
    \quad\text{and}\quad
    \ddot P := 2 \binom{2k}{k}^{-1}\Big(\frac{\mathrm d}{\mathrm dt}\Big)^{2k}_{t=0} P$$ 
    together form a second-order motion of $P$.
    It remains to show that $\dot P$ is non-trivial.
    Analogously to the proof of \cref{res:first_order}, this follows from the choice of $k$.
\end{proof}

The notion of a second-order flex is not easy to apply directly.
In contrast~to~first-order theory, the defining equations \eqref{eq1} and \eqref{eq2} are quadratic rather than linear in the unknowns $(\bsdot p,\bsddot p;\bsdot a,\bsddot a)$.
%
%
Given a candidate first-order flex $(\bsdot p,\bsdot a)$, the system becomes linear in $(\bsddot p,\bsddot a)$ and we can check efficiently if $(\bsdot p,\bsdot a)$ extends to a second-order flex.
This allows us to check second-order rigidity at least for systems whose space of first-order flexes is 1-dimensional.
If the dimension is higher, there is no obvious way to choose the candidate first-order flexes from which to start.

\subsection{The second-order stress test}
\label{sec:stress_test}

In \cite[Section 1.4]{connelly1996second} Connelly and Whiteley developed a second-order rigidity test that reinterprets second-order rigidity as the study of the bilinear pairing between first-order motions and equilibrium stresses. We adapt this technique to our setting.

\begin{theorem}[Stress test] \label{res:sotest}
    A polytope is second-order rigid if and only if for every first-order flex $\dot P$ there is a \Def{blocking} stress, that is, a stress $\bs\xi$ with
    \begin{eqnarray}
    \label{block}
    {\bs \xi}\T \mathcal  R(\dot P)\dot P\not=0.\end{eqnarray}
\end{theorem}

\begin{proof}
    Fix a first-order flex $\dot P=(\bsdot p,\bsdot a)$. We shall show that it can be extended~to~a second-order flex if and only if there is no blocking stress.

    Recall that, after fixing $\dot P$, the equations \eqref{eq1} and \eqref{eq2} become linear in $\ddot P=(\bsddot p,\bsddot a)$.
    The linear system can be compactly expressed as
    \begin{equation}\label{eq:compact}
        \mathcal R(P) \ddot{P}= -\mathcal R(\dot{P}) \dot{P}.
    \end{equation}
    This system has a solution in $\ddot P$ if and only if $\dot P$ extends to a second-order flex.

    The existence of a blocking stress $\bs\xi$ for $\dot P$ can be expressed by the system
    \begin{equation}\label{eq:blocking_system}
        \mathcal R(P)\T\bs\xi = 0,\qquad {\bs \xi}\T \mathcal  R(\dot P)\dot P\not=0.
    \end{equation}

    It remains to observe that the systems \eqref{eq:compact} and \eqref{eq:blocking_system} are dual in the sense of Farkas' lemma: if we set $A:=\mathcal R(P)$ and $b:=-\mathcal R(\dot P)\dot P$, then either there is~a~$\ddot P$~with $A\ddot P = b$, or there is a $\bs\xi$ with $A\T\bs\xi=0$ and $\bs\xi\T b\not=0$.
    \qedhere

\end{proof}

Following \cite{connelly1996second}, we also introduce the notion of prestress stability: a polytope is \Def{prestress stable} if there is a single stress $\bs\xi$ that blocks every first-order flex $\dot P$. By reversing the quantifiers from the criterion for second-order rigidity in \Cref{res:sotest}, this creates a stronger sufficient condition for rigidity. 



\begin{corollary}
    \label{res:prestress}
    first-order rigid $\implies$ prestress stable $\implies$ second-order rigid.
\end{corollary}
\begin{proof}
    If $P$ is first-order rigid, then there are no first-order flexes, and the condition of prestress stability is trivially satisfied.
    The second implication is clear~from~\cref{res:sotest} and the definition of prestress stability.
\end{proof}

\tempnewpage

\section{Edge Length Perturbations}
\label{section:edge-length-perturbations}

For this section, we restrict ourselves to polyhedra, \ie\ 3-dimensional polytopes, which come with various theoretical advantages over the general case (see \eg\ \cite{himmelmannschulzewinter2025rigiditypolytopesedgelength}). Here, we investigate 
$\real(\mathcal{P})$ locally by considering the space of all realizations that arise from perturbing some of the edge lengths of $P$. We establish topological techniques that provide novel conditions for the rigidity and flexibility of polyhedra, but also provide insights into the geometry of $\mreal(\mathcal{P},P)$ by providing an intuition for how first-order flexes are related to 
curves in $\real({\mathcal{P}})$.

According to the Legendre-Steinitz theorem \cite{rastanawi2021dimensions}, the realization space $\real({\mathcal{P}})$ is a $(E+6)$-dimensional manifold. At every polytopal realization, the trivial first-order motions form a 6-dimensional subspace of the space of first-order motions (\cf\ \Cref{res:trivial_dim}) which is completely determined by the 6-dimensional Lie group of Euclidean isometries $\Iso$.  
This turns $\realR(\mathcal{P})$ into a smooth manifold of dimension
\[\dim\left(\realR(\mathcal{P})\right)=\dim(\real(\mathcal{P}))-\dim(\mathrm{Iso})=(E+6)-6=E.
\]
In other words, all polyhedra have $E$ degrees of freedom when pinning a vertex, constraining an edge to a line and a face to a plane (\cf\ \cite{gortler2025higherorderrigidityenergy}). 
This modification factors out the trivial motions from $\real(\mathcal{P})$. 
The fact that $\dim(\real(\mathcal{P}))$ is exactly equal to the number of edges of $\mathcal{P}$ suggests that the degrees of freedom exactly correspond to perturbations of edge lengths. 
Since rigid realizations correspond to isolated points in $\mreal^*(\mathcal{P},P)$, the degrees of freedom in $\realR(\mathcal{P})$ must concentrate on the deformations that change the edge lengths. This line of thought would imply that we can freely perturb the edge lengths locally around rigid realizations (see \Cref{conj:edge_length_perturbations}). We formalize this intuition in the following definition.
    
\begin{definition}
\label{def:edge-length-perturbations}
The
 \Def{edge length vector} of $P$ is given by
 \[{\bs\ell}(P)\,:=\,\left(\,\frac{1}{2}\|p_i-p_j\|^2 ~\bigg\lvert~ij\in E\right).\]
This gives rise to the \Def{edge length map} $\bs\ell:\real(\mathcal{P})\rightarrow \mathbb{R}^E$. We say that a vector ${\bs\ell}'\in \mathbb{R}^{E}$ is \Def{realized} by $P$ if ${\bs\ell}'={\bs\ell}(P)$. The \Def{edge length realization space} $\mathcal{L(P)}$ of $P$ consists of all possible vectors of edge lengths corresponding to convex realizations of $\mathcal P$.
Moreover, $P$ \Def{has all edge length perturbations} if the edge length map $\bs\ell$ is locally surjective in a neighborhood of $P$. 
If that is not the case, 
we say that \Def{some of $P$'s edge length perturbations are not realizable}.
\end{definition}

Note that \Cref{def:edge-length-perturbations} can analogously be stated for $d$-dimensional polytopes; however, our focus on 3-dimensional polytopes is justified in that most results of this section use the Legendre-Steinitz theorem in an essential way.

Let us first consider an example of the first-order rigid regular tetrahedron. For this polyhedron, all edge length perturbations are realizable.

\begin{example}
    \label{ex:tetrahedron-edge-length-perturbations}
    Consider the regular tetrahedron with realization $P$.
    As a polytope, it is first-order rigid according to Dehn's theorem \cite{dehnstheorem}. The facet planarity constraints are superfluous in this case, as all facets are triangles. As the edge graph of the tetrahedron is the complete graph $K_4$, this polyhedron is first-order rigid, in fact, even minimally first-order rigid in $\mathbb{R}^3$, as a bar-joint framework, meaning that the removal of any edge renders it flexible. Given any sufficiently small perturbation ${\bs\ell}'\in \mathbb{R}^{E}$ of ${\bs\ell}(P)$, we can thus remove the length constraint from any edge to make the tetrahedron flexible, and then deform it until the particular edge length attains the corresponding value in $\bs\ell'$.
    We can iteratively repeat this process until we find a tetrahedral realization $Q$ sufficiently close to $P$ which satisfies ${\bs\ell}'={\bs\ell}(Q)$.
\end{example}

\Cref{def:edge-length-perturbations} can be used to formulate a relationship between a polyhedron's rigidity and the topological properties of its metric realization space $\mreal(\mathcal{P},P)$. With this language, we can show that first-order rigid polyhedra concentrate their degrees of freedom on deformations induced by edge length perturbations.

\begin{proposition}
    \label{prop:inf-rigidity-edge-perturbations}
    A first-order rigid polyhedron $P$ has all edge length perturbations.
\begin{proof}
    Consider the following commutative diagram:
    \[\begin{tikzcd}
    	{\RR^{dV}\oplus\RR^{dF}} && {\RR^E\oplus \RR^{V\!F}} \\
    	\\
    	{\real(\mathcal P)} && {\RR^E}
    	\arrow["f", from=1-1, to=1-3]
    	\arrow["\imath", hook, from=3-1, to=1-1]
    	\arrow["{\boldsymbol{\ell}}", from=3-1, to=3-3]
    	\arrow[hook, from=3-3, to=1-3, "\jmath"]
    \end{tikzcd}\]
    %
    The vertical arrows $\imath$ and $\jmath$ are the canonical embeddings.
    The arrow $\bs \ell$ is the edge length map. The arrow $f$ maps an arbitrary pair $(\bs p,\bs a)$ to the vector of lengths $\frac{1}{2}\|p_i-p_j\|^2$ for $ij\in E$ and coplanarity-defects $\<p_i,a_\sigma\>$ for $i\sim \sigma$.

    We observe that the Jacobian of $f$ at the realization $P$ is equal to the rigidity matrix $\mathcal R(P)$, and since $P$ is first-order rigid, \cref{res:corank_rigid} yields that
    \[\rank f = \rank\mathcal R(P)=dV+dF-\corank\mathcal R(P)=dV+dF-6,\]
    where the rank of a map is given by the rank of its Jacobian at the realization $P$.
    The map $\imath$ has rank $E+6$ by the Legendre-Steinitz theorem. 
    From the commutative diagram and basic results about the rank of matrix products, we obtain    
    \begin{align*}
        \rank(\bs\ell)\ge \rank(\jmath\circ \bs\ell) &= \rank(f\circ\imath)
        \\
        &\ge \rank(f) + \rank(\imath) - \dim(\RR^{dV}\oplus\RR^{dF})
        \\
        &= (dV+dF-6) + (E+6) - (dV+dF) 
        \\ &= E.
    \end{align*}
    Since $\bs\ell$ maps into $\RR^E$, it follows that the map has the maximal $\rank(\bs\ell)= E$. Finally, $\bs\ell$ is smooth, so
    we conclude that $\bs \ell$ is a local submersion (\cf\ \cite[Proposition 4.1]{lee-smooth-manifolds}) and hence locally surjective (\cf\ \cite[Proposition 2.2]{lang1999}). 
\end{proof}
\end{proposition}

A major result of \cite{himmelmannschulzewinter2025rigiditypolytopesedgelength} is that almost all realizations of polyhedra are first-order rigid.  
This result requires a suitable definition of \emph{genericity}, which is omitted here for the sake of brevity (see \cite[Definition 5.3.]{himmelmannschulzewinter2025rigiditypolytopesedgelength} for details). Intuitively, we can obtain a generic realization by drawing a random convex realization from $\real(\mathcal{P})$. We can apply the statement on the generic rigidity of polyhedra to \Cref{prop:inf-rigidity-edge-perturbations} in order to obtain a genericity result on the edge length perturbation space. 


\begin{corollary}
    Generic polyhedral realizations have all edge length perturbations.
    \begin{proof}
        Take any generic polyhedral realization $P$ of a combinatorial type $\mathcal P$.
        Accor\-ding to \cite[Theorem 1.2]{himmelmannschulzewinter2025rigiditypolytopesedgelength} $P$ is first-order rigid. Hence, \Cref{prop:inf-rigidity-edge-perturbations} is applicable, showing that $P$ has all edge length perturbations.
    \end{proof}
\end{corollary}

In contrast to \Cref{ex:tetrahedron-edge-length-perturbations}, we now consider a flexible polyhedron in the following example. This example highlights that not all realizations of a polyhedron have all edge length perturbations.

\begin{example}
\label{ex:cube-edge-length-perturbations}
    Consider the regular cube, whose vertices are realized by the points in $\{1,-1\}^3$. This is a flexible polyhedron, as it is a zonotope (see \Cref{sec:dimension_estimates}). It is not possible to continuously perturb a single edge without violating the planarity of the facets; the minimum requirement for realizing an edge length perturbation is to simultaneously adjust two 
    antipodal edges. 
\end{example}

We can turn the observation from this example, that not all polyhedra admit all edge length perturbations, into a general statement about the edge length perturbation space. 
However, we need to assert a slightly stronger condition than the local surjectivity of the edge length map $\bs\ell$, which is that we can construct a continuous, locally injective map $\psi:\mathbb{R}^E\rightarrow \real(\mathcal{P})$ that maps vectors in $\mathbb{R}^E$ sufficiently close to the polyhedron's edge length vector to distinct polyhedra sufficiently close to $P$ of the same combinatorial type such that $\psi(\bs\ell(P))=P$. If that is the case, we say that \Def{all edge length perturbations of $P$ are realizable via a continuous, locally injective map}.

\begin{theorem}
\label{thm:rigidity-and-edge-length-perturbations}
    If all edge length perturbations of $P$ are realizable via a continuous, locally injective map, then $P$ is rigid. 
    \begin{proof}
        Let $\mathcal P$ denote the combinatorial type of the polyhedron $P$. By a previous observation, the reduced realization space $\realR(\mathcal{P})$ of $\mathcal{P}$ is a smooth $E$-dimensional manifold with the local homeomorphism $\varphi:U_P\rightarrow \mathbb{R}^{E}$, which is defined in an open neighborhood $U_P\subset \realR(\mathcal{P})$ of $P$. 

        Assume now that all of $P$'s edge length perturbations are realizable via 
        the continuous map $\psi:\mathbb{R}^{E}\rightarrow \realR(\mathcal{P})$. This map takes a list of edge lengths and produces a realization of $\mathcal{P}$ with $\psi(\bs\ell(P))=P$ and it is injective in a sufficiently small open neighborhood $U_{{\bs\ell}(P)}\subset\mathbb{R}^{E}$ of ${\bs\ell}(P)$. After sufficiently shrinking $U_{\bs\ell(P)}$, we can assume that $\psi: U_{\bs\ell(P)}\rightarrow U_P$ is continuous and injective.
        
	\begin{figure}[h!]
		\centering
		\begin{tikzpicture}[node distance=3.5cm, auto]
			
			\node (R1) { $\mathbb{R}^{E}$};
            \node (UR) [right=-0.1cm of R1] { $\supset U_{{\bs\ell}(P)}$};
			\node (L) [below=-0.075cm of UR] { \rotatebox[origin=c]{90}{$\in$}};
			\node (L2) [below=-0.075cm of L] { 
			\small${\bs\ell}(P)$};
			\node (CP) [right=of R1] {~\small$U_P$~};
            \node (L3) [below=-0.075cm of CP] { \rotatebox[origin=c]{90}{$\supset$}};
			\node (L4) [below=-0.075cm of L3] { 
			\small$\realR(\mathcal{P})$};
			\node (R2) [right=of CP] {$\mathbb{R}^{E}$};
			
			\draw[->, thick] (UR) to node[above]{ $\psi$} (CP);
			\draw[->, thick] (CP) to node[above]{ $\varphi$} node[below]{\small $\mathrm{Legendre}$-$\mathrm{Steinitz}$} (R2);
			
			\draw[->, thick, bend left=20] ([xshift=0.1cm,yshift=0.035cm]UR.north) to node[above=0.1cm] {\small $\mathrm{Invariance}$ $\mathrm{of}$ $\mathrm{Domain}$} ([xshift=-0.1cm,yshift=0.035cm]R2.north);
			
		\end{tikzpicture}
        \caption{A diagram for the topological argument used in this proof}
        \label{fig:proofhelper-invariance-of-domain}
	\end{figure}
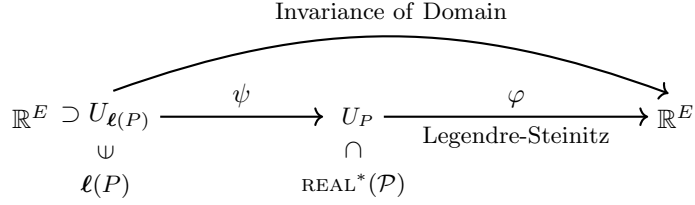
    
        We visualize the following topological arguments in \Cref{fig:proofhelper-invariance-of-domain}. By the Invariance of Domain Theorem \cite{invariance-of-domain}, the continuous injection $\varphi\circ \psi: U_{{\bs\ell}(P)}\rightarrow V\subset \mathbb{R}^{E}$ for \mbox{$V=\varphi\circ \psi (U_{{\bs\ell}(P)})$} is a homeomorphism on its domain and $V$ is open. This lets us pull back the open set $V$ to the open set $\varphi^{-1}(V)$ in $\realR(\mathcal{P})$, since $\varphi$ is a local homeomorphism. 
        Hence, $\psi:U_{{\bs\ell}(P)}\rightarrow \varphi^{-1}(V)$ is a homeomorphism, since $\psi$ can be expressed as $\varphi^{-1}\circ(\varphi\circ \psi)$ and since $\varphi^{-1}$ and $\varphi\circ \psi$ are both homeomorphisms.
        
        Consequently, there exists an open neighborhood around $P$ in the manifold $\realR(\mathcal{P})$, in which all realizations come from edge length perturbations. In other words, there is no curve in $\realR(\mathcal{P})$ through the realization $P$ which preserves the edge lengths ${\bs\ell}(P)$, so $P$ is rigid. 
    \end{proof}
\end{theorem}
The assumption that $\psi$ is continuous is equivalent to the local continuity of $\psi$ in the neighborhood $U_{\bs\ell(P)}$: We can find a sufficiently small ball around $\bs\ell(P)$ inside the open set $U_{\bs\ell(P)}$, which lets us continuously extend $\psi$ to all of $\mathbb{R}^E$ since an open ball is homeomorphic to $\mathbb{R}^E$. Moreover, we show in the proof of \Cref{thm:rigidity-and-edge-length-perturbations} that $\psi$ must be a local homeomorphism under these assumptions. 

Since \Cref{thm:rigidity-and-edge-length-perturbations} provides a sufficient condition for the rigidity of polyhedra, its contrapositive yields information about the reduced realization space $\realR(\mathcal{P})$. This observation is sufficiently useful to state explicitly as a corollary. Albeit, it is currently unclear whether checking the theorem's preconditions offers any advantage over directly verifying a polyhedron's rigidity or flexibility.

\begin{corollary}
    If a polyhedron $P$ of combinatorial type $\mathcal{P}$ is flexible, then there exists no continuous map $\psi:\mathbb{R}^E\rightarrow \realR(\mathcal{P})$ that is locally injective around $\bs\ell(P)$ with $\psi(\bs\ell(P))=P$.
\end{corollary}
Note that the local surjectivity of $\bs\ell$ is a strictly weaker condition than the existence of a locally continuous map $\psi$ that is asserted in \Cref{thm:rigidity-and-edge-length-perturbations}. 
To explain that the condition is strictly weaker, let us consider a counterexample. 
\begin{remark}
\label{res:square_root_branching}
Consider $f:\mathbb{R}^2\rightarrow \mathbb{R}^2$ with $(x,y)\mapsto (x^2-y^2,\,2xy)$. This mimics the holomorphic function $\mathbb{C}\rightarrow \mathbb{C}$ that maps $z\mapsto z^2$ with values over $\mathbb{R}$. The map $f$ is continuous and, since over $\mathbb{C}$ square roots do exist, it is also locally surjective at the point $(0,0)$. 
However, the square root function is set-valued, so there exists no (locally) continuous square root function over $\mathbb{C}$ due to the branching at $(0,0)$.
\end{remark}

However, it may still be possible that the local surjectivity of $\bs\ell$ is sufficient for the statement of \Cref{thm:rigidity-and-edge-length-perturbations}. With the example of the square root in mind, we may need to show that the fiber size of $\bs\ell$ is constant outside of the realization $P$. 

The correspondence between the edge length realization space of a polyhedron and its rigidity properties is fascinating. Even though \Cref{prop:inf-rigidity-edge-perturbations} and \Cref{thm:rigidity-and-edge-length-perturbations} almost paint a complete picture of this correspondence, we were unable to resolve the case where the polyhedron is rigid, but not first-order rigid. In our experiments (\cf\ \Cref{section:dodec-edge-length-perturbations}), we find that this case seems to have particularly interesting algebraic and topological properties. For that reason, we formulate the following two questions, which we hope will inspire future research in this topic. 

\begin{question}
\label{conj:edge_length_perturbations_implies_rigid}
\label{conj:edge_length_perturbations}
    If all edge length perturbations of a convex polyhedron $P$ are realizable, is $P$ then necessarily rigid? Conversely, does every rigid realization~of a polyhedron have all edge length perturbations?
\end{question}


\tempnewpage

\section{Certifying Second-order rigidity}
\label{section:second-order-rigidity-check}
\label{sec:so_certificate}

As laid out in \Cref{sec:second-order-theory}, any first-order flex of a polytope 
that has the chance to extend to a continuous motion must necessarily not be blocked by an equilibrium stress. 
That is, according to \Cref{res:sotest}, there exists a first-order flex $\dot P$ so that for all equilibrium stresses $\bs \xi$ we have $\bs \xi\T
\mathcal{R}(\dot{P}) \dot{P}=0$.
Conversely, if for each $\dot{P}$ there exists a $\bs \xi$ such that $\bs \xi\T
\mathcal{R}(\dot{P}) \dot{P}\neq 0$, then the polytope is 
second-order rigid, which guarantees that it is rigid (\cf\ \Cref{res:second_order_rigid_impl_rigid} and \ref{res:sotest}). This criterion can be checked by parameterizing the first-order flex space via a basis
${\dot{P}}=({\dot{P}}_1,\dots,{\dot{P}}_r)$ with variables $\lambda=(\lambda_1,\dots, \lambda_r)$ and the equilibrium stress space via a basis $\bs\xi=({\bs \xi}_1,\dots,{\bs \xi}_s)$ with variables $\mu=(\mu_1, \dots, \mu_s)$. 
This turns the second-order condition \eqref{block} into the cubic polynomial

\[E_{{\dot{P}},\,\bs \xi}(\lambda,\mu)\,=\,\left(\sum_{j=1}^s \mu_j{\bs \xi}_j\right) \cdot \mathcal{R}\left(\sum_{i=1}^r \lambda_i{\dot{P}}_i\right)\cdot \left(\sum_{i=1}^r \lambda_i{\dot{P}}_i\right).\]

More precisely, the homogeneous polynomial $E_{\dot{P},\,\bs\xi}$ is linear in the variables $\mu$ and quadratic in the variables $\lambda$. For asserting second-order rigidity, it suffices to show that every first-order flex has a blocking stress according to \Cref{res:sotest}. In other words, second-order rigidity can be checked by finding a $\mu$ for each $\lambda$ 
so that $E_{{\dot{P}},\,\bs \xi}(\lambda,\mu)\neq 0$.
By expanding $E_{\dot{P},\,{\bs\xi}}(\lambda,\mu)$ in terms of $\mu_j$, we obtain 
\begin{eqnarray}
\label{eqn:stress_polynomial_Qi}
E_{{\dot{P}},\,\bs \xi}(\lambda,\mu)\,=\,\sum_{j=1}^s \mu_j \cdot Q_j(\lambda)
\end{eqnarray}
for quadratic forms $Q_j\in \mathbb{R}_2[\lambda_1,\dots,\lambda_r]$.

The crucial observation is the following: If any of the $Q_j$ is non-zero for a given $\lambda$, then we can find $\mu$ such that $\sum_{j=1}^s\mu_j\cdot {\bs\xi}_j$ blocks $\sum_{i=1}^r\lambda_i\cdot{\dot{P}}_i$. 
Therefore, the first-order flexes that are not blocked by any stresses can be parametrized by the solutions of the 
polynomial system $Q_1=\dots= Q_s=0$, which form an algebraic set. 
This system can also be represented by the gradient of $E_{\dot{P},{\bs\xi}}$ with respect to $\mu$ to obtain $\nabla_\mu E_{\dot{P},{\bs\xi}}(\lambda)=0$.

We will later see that the zero-dimensionality of the algebraic set $\mathcal{V}(Q_1,\dots,Q_s)$ is sufficient for second-order rigidity. As a criterion for that, we formulate the following. It is a standard result in computational algebraic geometry and can be found in textbooks such as \cite[Theorem 5.3.6]{cox-little-oshea}. 
    \begin{lemma}\label{gbasis}
        Fix a field 
        $\Bbbk$ and let $\mathcal{I}=\langle f_1,...,f_s\rangle$ be an ideal in $\Bbbk[x_1,..,x_r]$. Let $\mathcal{G}$ be a Gröbner basis of $\mathcal{I}$. Then, $\mathcal{I}$ is zero dimensional if and only if for every $i\in\{1,...,r\}$ some power of $x_i$ is the leading monomial of an element of $\mathcal{G}$.
    \end{lemma}

Of course, there are also other methods based on Gröbner bases to verify that an ideal is zero-dimensional; however, for our purposes, \Cref{gbasis} proves to be sufficient. If the ideal $\mathcal{I}=\langle Q_1,\dots,Q_s \rangle$ is zero-dimensional, then the algebraic set $\mathcal{V}(\mathcal{I})$ contains only finitely many points \cite[Proposition 1.7]{hartshorne}. Since the quadratic forms $Q_i$ are homogeneous, it holds that \[\left(Q_1(c\cdot \lambda^*),\,\dots,\,Q_s(c\cdot \lambda^*)\right)=\left(c^2\cdot Q_1(\lambda^*),\,\dots,\,c^2\cdot Q_s(\lambda^*)\right)= 0\] 
for any $c\in \mathbb{C}$ and any $\lambda^*\in \mathcal{V}(\mathcal{I})$. In other words, all scalar multiples $c\cdot\lambda^*$ must also lie in $\mathcal{V}(\mathcal{I})$. The zero-dimensionality of $\mathcal{V}(\mathcal{I})$ is therefore equivalent to $\mathcal{V}(\mathcal{I})=\{0\}$.
In our context, this translates to the non-existence of non-blocked flexes. We summarize our observations in the following:
\begin{proposition}
\label{prop:stress-energy-condition}
    If the homogeneous polynomial system 
    $\nabla_\mu E_{\dot{P}, \bs \xi} (\lambda) = 0$ defined in \eqref{eqn:stress_polynomial_Qi} only has $\lambda_1=\dots=\lambda_r = 0$ as a real solution, then all first-order flexes are blocked by equilibrium stresses. Conversely, any flex that is not blocked by an equilibrium stress corresponds to a non-trivial, real zero of $\nabla_\mu E_{\dot{P}, \bs \xi} (\lambda)=0$.
\end{proposition}

The approach based on Gröbner bases described in this section is sufficient to decide the second-order rigidity of all polytopes tested in \cref{sec:experimental_results} (\cf\ \Cref{tab:comparison} and \cref{tab:catalan-comparison}). 
Still, 
our result using Gröbner bases is clearly not necessary for second-order rigidity, as we only determine the emptiness of $\mathcal{V}(\mathcal{I})$ over $\mathbb{C}$. Therefore, it is possible that $\mathcal{V}(\mathcal{I})=\{0\}$ over the real numbers, even when we find non-trivial solutions over $\mathbb{C}$. A simple counterexample is given by the equation $x^2+y^2=0$, which only has $x=y=0$ as a real solution, but has the line $x=iy$ as a complex solution. A weaker certificate is given by the \Def{Euclidean distance degree problem} \cite{euclideandistancedegree} for a generic point $p\in \mathbb{R}^r$, which guarantees finitely many solutions on all connected components of $\mathcal{V}_\mathbb{R}(\mathcal{I})$ (\cf\ \cite{deformationpaths}). By appending the equation $\sum_{i=1}^r\lambda_i^2-1$ to the ideal $\mathcal{I}$ which describes the unit hypersphere, we guarantee that this approach only finds non-trivial solutions. We also do not lose any solutions beyond the origin due to the invariance of solutions under scalar multiplication.

Analogous to second-order rigidity, we can decide prestress stability (\cf\ \Cref{res:prestress}). For that, we need to decide if there exists a $\mu\in \mathbb{R}^s$ defining a 
stress ${\bs\xi}=\sum_{j=1}^s \mu_j\cdot {\bs\xi}_j$ such that $E_{\dot{P},\,{\bs\xi}}(\lambda,\mu)\not= 0$ for all $\lambda$. This is a standard semidefinite programming (SDP) problem: we need to show that $\sum_{j=1}^s \mu_j \cdot Q_j(\lambda)$ is positive definite for some choice of $\mu$. As this is a homogeneous quadratic form, we can alternatively check whether the matrix
$$\sum_{j=1}^s \mu_j \cdot \left(\frac{\partial^2}{\partial \lambda_k\lambda_\ell}Q_j(\lambda)\right)_{k,\ell=1}^r$$
is positive definite for some choice of $\mu$. 

\tempnewpage

\section{Second-order rigidity of the regular dodecahedron}
\label{sec:so_rigidity_dod}

Among the Platonic solids, there is only one for which the question of rigidity is not answered immediately.
The regular tetrahedron, octahedron and icosahedron are simplicial, hence rigid by Dehn's theorem \cite{dehnstheorem}. 
The cube is the classical example of~a~flexible polytope.
Solely the regular dodecahedron remains to be understood and initially eluded direct computation.

Symbolic computations reveal a 5-dimensional space of first-order flexes and a~5-dimensional space of stresses.\footnote{The specific embedding and orderings of vertices, edges and faces are provided in \cref{appendix:comp-data-dodecahedron}.}
Note that these dimensions must be identical~due~to \cref{res:flex_eq_stress}.
It is tempting to somehow relate the dimensions of the flex and stress spaces to the 5-fold symmetries of the dodecahedron.
We have no evidence~for such a relation, and instead found rather mild counter-evidence in the case of the truncated~dode\-cahedron (see \cref{sec:selected_rigidity}).

Since the regular dodecahedron is not first-order rigid, it is a promising test case for the second-order certificate from \cref{sec:so_certificate}.
Concretely, \cref{prop:stress-energy-condition} asserts that it suffices to compute the zeros of a certain polynomial system. 
We~proceed~as follows:
\begin{enumerate}
    \setlength{\itemsep}{0.7ex}
    \item 
    Symbolically compute a basis of the space of first-order flexes $(\dot{P}_1,\dots,\dot{P}_5)$ and a basis of the space of equilibrium stresses $(\bs \xi_1,\dots, \bs\xi_5)$. 
    Using the~bases, we parametrize the spaces by variables $\lambda_1,\dots,\lambda_5$ and $\mu_1,\dots,\mu_5$.
    
    \item From the stress energy $E_{\dot{P},\bs\xi}(\lambda, \mu)$, we extract the polynomial system $$\nabla_\mu E_{\dot{P},\bs\xi}(\lambda)=(Q_1(\lambda),\,\dots,\, Q_5(\lambda))=0$$ consisting of quadratic forms $Q_i$. 
    
    \item Use the equivalence in \cref{gbasis} to determine whether the system~has~fi\-nitely many (real) zeros.
    This is the case for the regular dodecahedron.
    
    \item Use \Cref{prop:stress-energy-condition} to infer second-order rigidity.
\end{enumerate}

    The Gröbner basis $\mathcal{G}$ for the ideal $\mathcal{I}=\langle Q_1,\,\dots,\,Q_5\rangle$ in the lexicographic ordering consists of 31 polynomials and is provided in the Supplementary Material \cite{zenodo_suppMaterial}.
    For more specifics on the relevant \textsc{Mathematica} files, see \cref{appendix:comp-data-dodecahedron}.
    With the help of this Gröbner basis, we obtain the following results:

    \begin{theorem}\label{dodecahedronrig}
        The regular dodecahedron is rigid. 
    \end{theorem}
    \begin{proof}
        The monomials $\lambda_1^2$, $\lambda_2^3$, $\lambda_3^3$, $\lambda_4^5$,  $\lambda_5^6$, appear as leading monomials in the Gröbner basis $\mathcal{G}$ for the ideal $\mathcal{I}=\langle Q_1,\,\dots,\, Q_5\rangle$.
        \Cref{gbasis} therefore implies that $\mathcal{I}$ is zero-dimensional. 
        Since $\mathcal{I}$ is a homogeneous ideal, the only solution is $\lambda_1=\lambda_2=\cdots=\lambda_5=0$. 
        By \cref{prop:stress-energy-condition}, for any non-zero linear combination of the five first-order flexes, we can find a stress so that inequality \eqref{block} holds. 
        By \Cref{res:sotest}, $P$ is second-order rigid. 
        By \cref{res:second-order-implies-rigid}, $P$ then is rigid. 
    \end{proof}

Having understood the rigidity properties of the Platonic solids, a natural move is to extend the search to the next best well-documented classes of polytopes -- the Archimedean and Catalan solids.
A detailed account of the findings is given~in~\cref{sec:archimedean} (see also \cref{tab:comparison} and \cref{tab:catalan-comparison}).

Among the 31 polytopes checked, only three are not first-order rigid, but also not predictably flexible (\cf\ \cref{sec:dimension_estimates}): the regular dodecahedron, the truncated icosahedron and the truncated dodecahedron (\cf\ \cref{fig:Dod_TIco_TDod}).
The procedure above was applied to all of them, and our findings are documented in \cref{sec:selected_results}.

\tempnewpage

\section{Computational results: Platonic, Archimedean and Catalan solids}
\label{sec:experimental_results}
\label{sec:archimedean}

We apply our techniques to the Platonic, Archimedean, and Catalan~solids.\nls
For these classes, combinatorial data and algebraic vertex and normal coordinates are readily available.%
%
%
\footnote{The coordinates, orderings of vertices, edges and facets are taken from the \textsc{Mathematica} polyhedral database. The computations are provided in the Supplementary Material \cite{zenodo_suppMaterial}.}
The results are shown in \cref{tab:comparison} (for the Platonic and Archime\-dean solids) and \cref{tab:catalan-comparison} (for the Catalan solids).
An explanation of the content of the tables is given in \cref{sec:explain_table}.
A key takeaway is that even simple examples can show a wide range of rigidity behaviors.

We were able to decide the rigidity of the tested polytopes in all but one case. 
In particular, we determined the previously unknown rigidity properties of the \textit{regular dodecahedron}.
Solely the rigidity of the \emph{truncated dodecahedron} cannot be decided by our techniques (marked with \Ques\ in \cref{tab:comparison}). 
Solving this question appears to~re\-quire techniques of even higher order and can thus serve as a simple test case for future theoretical and computational developments.
See \cref{sec:selected_results,sec:rigidity_comments,sec:selected_flexibility}~for a discussion of the results.

\subsection{Explanation of table contents}
\label{sec:explain_table}

Each table cell contains the name and~a~visualization of the polytope in question.
The parentheses contain the following data from left to right:

\begin{enumerate}
\setlength{\itemsep}{1ex}

    \item Whether the polytope is \textit{first-order rigid} (\yes) or \textit{first-order flexible} (\no). This is computed \textit{symbolically} based on the description in \cref{sec:first_order}.

    \item Whether the polytope is \textit{prestress stable} (\yes) or \textit{not} (\no). If not implied~by other symbolic first- or second-order computations, this is computed \textit{numerically} based on the description in \cref{sec:second_order} and using our novel techniques developed in \cref{section:second-order-rigidity-check}, in this case reducing to a semidefinite program.\nls
    We use \yesBut\kern-1pt/\kern1pt\noBut\ to indicate numerical results.
    A dual symbolic approach is described in \cite[Lemma 6.3.2]{Lovasz2003SDP}, but we do not currently have a way to produce such a certificate for our setting.
    
    \item Whether the polytope is \textit{second-order rigid} (\yes) or \textit{second-order flexible}~(\no). This is computed \textit{symbolically} based on the description in \cref{sec:second_order} and using our new techniques developed in \cref{section:second-order-rigidity-check}.
    
    \item Whether the polytope is \textit{rigid} (\yes) or \textit{flexible}~(\no). This is inferred from the previous symbolic computations, from the polytope's combinatorial data (see \Cref{sec:dimension_estimates}) or indicated as unknown (\Ques).
    
    \item The \textit{dimension of the first-order flex space} as introduced in \cref{sec:first_order}, which we compute \emph{symbolically}.
    In some cases we have a good theoretical understanding of how the dimension of the flex space comes about. 
    In other cases we can derive theoretical lower bounds on the dimension.
    %
    %
    If further explanations are available, this is indicated~by a symbol \SymOne, \SymTwo, \SymThree\ (see below), and further theoretical background is provided in \cref{sec:dimension_estimates}.
\end{enumerate}

The following symbols in the table indicate further explanations regarding their rigidity properties and flex space dimensions:
\begin{itemize}
\setlength{\itemsep}{1ex}

    \item[\SymTwo] This is a \textit{simplicial polytope} (\ie\ all faces are triangles). 
    Hence the polytope is first-order rigid by Dehn's rigidity theorem, as well as globally rigid by Cauchy's rigidity theorem (for details on the classical rigidity theorems see \cite[Remark 1.3]{himmelmannschulzewinter2025rigiditypolytopesedgelength} or \cite[Chapter 26+]{pak2010lectures}).
    
    \item[\SymThree] This polytope is a \emph{generic Minkowski sum}, and hence flexible according~to \cite[Section 4.1]{himmelmannschulzewinter2025rigiditypolytopesedgelength}.
    The dimension of the first-order flex space decomposes~according to~the two types of flexes that exist for such polytopes~(summand flexes and reorientation flexes; see \cref{sec:generic_Msum} for details).
    
    \item[\SymOne] This is a \textit{zonotope}, and therefore flexible according to \cite[Section 4.2]{himmelmannschulzewinter2025rigiditypolytopesedgelength}. See \cref{sec:zonotopes} for an explanation of the observed flex space dimensions.\nls If the zonotope is generic (marked by \SymOne\ \SymThree), the dimension of the first-order flex space can be computed exactly (\cf\ \cref{res:generic_zonotopes}).
\end{itemize}
%

\pagebreak
\vspace*{6mm}

\input{sec/table1}
\newpage


\subsection{Selected results}
\label{sec:selected_results}
\label{sec:selected_rigidity}

We highlight the following findings from \cref{tab:comparison} (the corresponding table cells are marked with \Star):

\begin{itemize}
\setlength{\itemsep}{0.7ex}
    \item The \Def{regular dodecahedron} and the \Def{truncated icosahedron} have identical~rigidity properties: they are neither first-order rigid (both have a 5-dimensional space of first-order flexes) nor prestress stable; yet, \emph{both are second-\mbox{order} rigid}. 
    In particular, they are rigid. 
    This was previously stated as an open question for the regular dodecahedron in
    \cite[Section 6.2]{himmelmannschulzewinter2025rigiditypolytopesedgelength}. 
    With these rigidity properties, both polytopes are remarkably natural geometric constraint systems which demonstrate that prestress stability and second-order rigidity are distinct concepts.
    Previous examples of this distinction were compa\-ratively ad hoc (\eg\ \cite[Figures 9b and 17a]{connelly1996second}).
    
    \item The \Def{truncated dodecahedron} is not first-order rigid (it~has~a 4-dimensional space of first-order flexes), not prestress stable and \emph{not~second-order rigid}.
 In fact, for all quadratic forms in \eqref{eqn:stress_polynomial_Qi} we verified symbolically that $Q_i=0$. 
 Second-order flexibility then followed via \cref{prop:stress-energy-condition}.
    Since at this time no 
    methods exist to test rigidity beyond the second order, we do not know whether the truncated dodecahedron is in fact rigid.
    If it is rigid, it would provide a remarkably natural geometric constraint system demonstrating that rigidity does not imply second-order rigidity. 
    Previous examples of~this distinction have been comparatively ad hoc (\eg\ \cite[Figure 22]{connelly1996second}). 
    If it is flexible, its flex would be fundamentally different from all polytope flexes encountered before (\cf\ \cite[Question 4.2]{himmelmannschulzewinter2025rigiditypolytopesedgelength}).
\end{itemize}

Numerical experiments suggest that the first-order flexibility of the three polytopes discussed above is sensitive to small changes in their geometry.
First, applying a generic linear transformation results in first-order rigid polytopes. This is in contrast to how all known flexible polytopes preserve their flexibility under linear transformations (\cf\ \cite[Question 4.3]{himmelmannschulzewinter2025rigiditypolytopesedgelength}).
Second, changing the height of truncation in the truncated variants results in first-order rigid polytopes. 
At this time, it is unclear why exactly the realizations with uniform edge lengths are the ones~that end up being first-order flexible.
In particular, these properties cannot be a consequence of the polytopes' symmetry alone.

\subsection{Comments on rigidity}
\label{sec:rigidity_comments}

Dehn's theorem \cite{dehnstheorem} states that simplicial polytopes (\ie\ where all faces are simplices) are first-order rigid.
\Cref{tab:comparison} and \cref{tab:catalan-comparison} contain plenty of non-simplicial polytopes that are first-order rigid, for example, the~trunca\-ted tetrahedron (\cref{tab:comparison}) or deltoidal icositetrahedron (\cref{tab:catalan-comparison}).
Currently no theo\-retical tools are available that would identify these as rigid a priori.
We therefore formulate the following question:

\begin{question}
    Can Dehn's theorem be extended to some classes of non-simplicial polytopes? 
    That is, can we formulate combinatorial/geometric
    %
    %
    criteria for the first-order rigidity of polytopes? 
\end{question}

Perhaps, polytopes whose only faces are triangles and 4-gons form a class simple enough to admit a complete characterization.
See also \cite[Question 4.4.]{himmelmannschulzewinter2025rigiditypolytopesedgelength}.

\subsection{Comments on flexibility}
\label{sec:selected_flexibility}
\label{sec:flexibility_comments}

Save the polytopes of \cref{sec:selected_rigidity}, we have partial explanations for the first-order flex space dimensions listed in \cref{tab:comparison} and \cref{tab:catalan-comparison}.
In all cases, these first-order flexes are explained by actual flexes.
We know~this both from theoretical considerations (see \cref{sec:dimension_estimates}) and the numerically generated flexes provided in the Supplementary Material \cite{zenodo_suppMaterial}.

Remarkably, the lower bounds on the dimensions of the first-order flex spaces obtained in \cref{sec:dimension_estimates} match the computationally obtained values in the tables exactly.
We can conclude that, at least locally at their respective realizations, we have a full understanding of the flexes of these polytopes.

\input{sec/table2}


\pagebreak

\section{Case study: the regular dodecahedron}
\label{sec:dodecahedron}

Due to its previous status as the guiding example,
we single out the dodecahedron once again for a closer inspection.
Exploring its realization space close to the regular dodecahedron helps us better understand the source of its exceptional behavior.
From that, we draw conclusions about the global rigidity of the polytope.

\subsection{Deformation paths}
\label{section:dodec-edge-length-perturbations}

By the Legendre-Steinitz theorem the reduced realization space $\realR(\mathcal{P})$ of the dodecahedron is a smooth $30$-dimensional manifold. 
The dodecahedron is rigid by \Cref{dodecahedronrig}, suggesting that these 30 degrees of freedom translate directly into changes in the lengths of its 30 edges.
Although we formulate the hope that all edge length perturbations of rigid polytopes are realizable in \Cref{conj:edge_length_perturbations}, even for the regular dodecahedron it is completely unclear whether this statement is true.

Due to the dodecahedron's high symmetry (all edges are the same), one might suggest to deform a single edge and, if this is possible, to conclude that arbitrary perturbations are possible.
In this section, we investigate the deformations of the regular dodecahedron induced by shrinking or extending a single edge. 
By approximating such deformation paths using methods explained in \Cref{sec:approx_continuous_motions}, we probe a neighborhood of this polytope $P$ in $\realR(\mathcal P)$. 
However, whether this implies that all edge length perturbations are possible remains unanswered.

For this experiment, we keep all but one edge length constant at $\ell=1$ throughout the deformation. Denote the non-constant edge length by $\ell^*(t) = 1+t$ with $t\in (-\delta,+\delta)$ for some $\delta>0$ such that $\lvert \delta \lvert~ < 1$. 
Sampling from the realization space intersected with the real algebraic set of prescribed edge lengths (see also \cite{deformationpaths}) reveals \textit{six} distinct arcs starting at the regular dodecahedron corresponding to the shrinking of a single edge and \textit{two} corresponding to its expansion. 

        \begin{figure}[ht]
            \centering
            \includegraphics[width=0.75\linewidth]{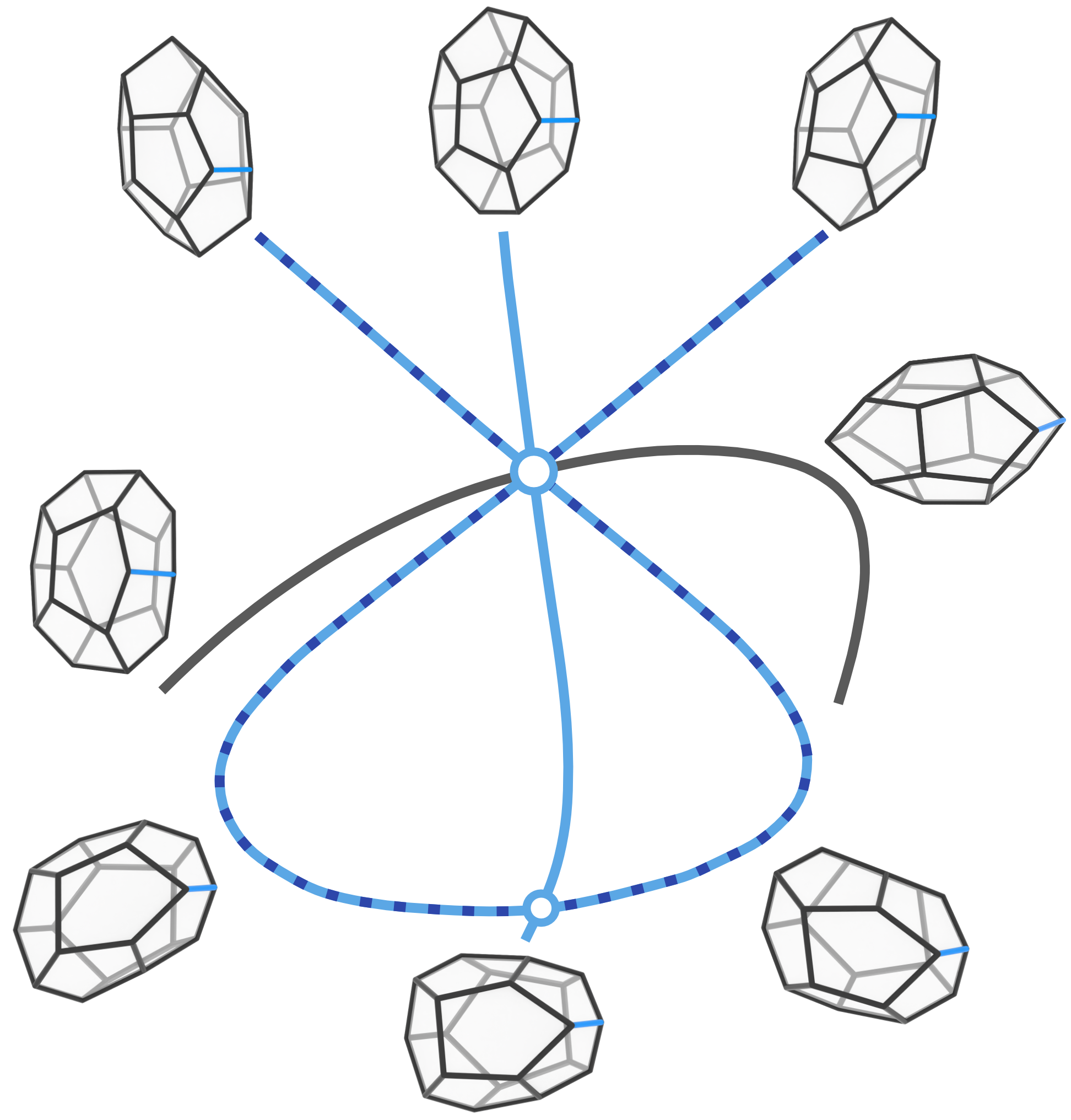}
            \caption{Realizations obtained by modifying the length of a single edge of the dodecahedron (central blue circle). Blue curves correspond to edge contractions, whereas gray curves correspond to edge expansions. On these curves, the eight depicted 
            realizations are located at $\ell^*(\pm 0.25)$. Four realizations are connected by a nodal cubic (dark blue dashes).}
            \label{fig:edge-contraction-deformation-space}
        \end{figure}

%
%
In \Cref{fig:edge-contraction-deformation-space}, we depict a random projection of the deformations induced by edge contractions 
(blue) and expansions (gray), alongside eight realizations at $\ell^{\star}(\pm 0.25)$.
Remarkably, it appears that some of these arcs rejoin smoothly. 
For instance, the gray arcs corresponding to an edge expansion have the same tangent in the regular dodecahedron, and so do suitable pairs of arcs of the blue curves. 
Moreover, the arcs corresponding to the four lower-symmetry realizations can be combined to form a nodal cubic curve (dark blue dashed curve).
It is conceivable that for these reasons \Cref{thm:rigidity-and-edge-length-perturbations} does not apply to the regular dodecahedron since there are multiple realizations with the same edge lengths arbitrarily close to this polytopes. 
Hence, it is possible that in this case there cannot exist a continuous, locally injective map $\psi:\mathbb{R}^E\rightarrow \realR(\mathcal{P})$ such that $\psi(\bs\ell(P))=P$ (\cf\ \Cref{res:square_root_branching}). 

In total, three curves are passing through the regular dodecahedron that are induced by length changes of a single edge. 
Numerically computing the derivative in the regular dodecahedron confirms this observation: though naturally, we expect that our numerical calculations are not exact, the respective sum of the tangent vectors is sufficiently close to zero. 
We attach an animation of the smooth deformation passing through the regular dodecahedron on the solid blue path from $\ell^*(-0.25)$ via $\ell^*(0)$ to $\ell^*(-0.25)$ in the Supplementary Material \cite{zenodo_suppMaterial} of this article.

Additionally, we find that in the case of edge contractions there are two types of deformations: those that react to the edge contraction by flattening horizontally (see \Cref{fig:edge-contraction-deformation-space}, top) and those that flatten vertically (see \Cref{fig:edge-contraction-deformation-space}, bottom). Both types contain one realization with two reflection symmetries, and two realizations with only one reflection symmetry that are mutual mirror images of each other. 
        Intriguingly, not all dodecahedra persist until the edge collapses to length zero. 
        The configurations in the bottom left and bottom right are confluent with the configuration in the bottom center at $\approx\ell^{\star}(-0.37194\dots)$, leaving only four configurations below this particular edge length. 
        All of them persist until $\ell^*(-1)=0$.

        Conversely, expanding an edge only produces two realizations with distinct reflection axes. 
        They are depicted on the left and right of \Cref{fig:edge-contraction-deformation-space}. 
        Despite the dodecahedron's high level of regularity and symmetry, there are more deformation paths for contracting an edge than for expanding it. This behavior is surprising and should be the object of further investigation.

\subsection{Global rigidity}
\label{sec:dodecahedron-global-rigidity}

The results of
\Cref{section:dodec-edge-length-perturbations} provide insight into the concept of global rigidity for convex polytopes.
Recall that a polytope is called \Def{globally rigid} if there exists no other realization of the polytope with the same edge lengths.
Clearly, the realizations shown in \cref{fig:edge-contraction-deformation-space} (either on the blue or on the gray curve) are equivalent, that is, have identical edge lengths. None of them is therefore~globally rigid.
This observation
leaves open the possibility of generic global rigidity. 

A polytope is called \Def{generically globally rigid} if every generic 
realization is globally rigid \cite{Connelly2005}. Once again, we refer to the notion of genericity introduced in \cite[Definition 5.3]{himmelmannschulzewinter2025rigiditypolytopesedgelength}.
First, we confirmed numerically that all of the equivalent realizations which are 
depicted in \cref{fig:edge-contraction-deformation-space} are first-order rigid. 
Using \cref{prop:inf-rigidity-edge-perturbations}, each of these realizations can therefore be perturbed to achieve the same edge lengths as a sufficiently close generic realization of the dodecahedron.
Hence, generic realizations are also not necessarily globally rigid.
Notably, the most concrete case remains open:

\begin{question}
    \label{q:global_rigid}
    Is the regular dodecahedron globally rigid? Conversely, does there exist a convex realization of the dodecahedron with all edges of the same length, besides the regular dodecahedron?
\end{question}

It would be surprising, though also not absurd, if other realizations with equal edge lengths exist. Consider, for example, the near miss shown in \cref{fig:dodecahedron_variation} (left, center), whose edge lengths are almost all equal.
An alternative realization, for which all edges have exactly the same length, is the ``flat dodecahedron'' shown in \cref{fig:planar-dodec} (right). It does not lie in $\real(\mathcal{P})$, since $\langle p_i,\,a_\sigma\rangle = 1$ for all $i\sim \sigma$.

\begin{figure}[h!]
    \centering
    \begin{minipage}[c]{0.575\textwidth}
    \includegraphics[width=\linewidth]{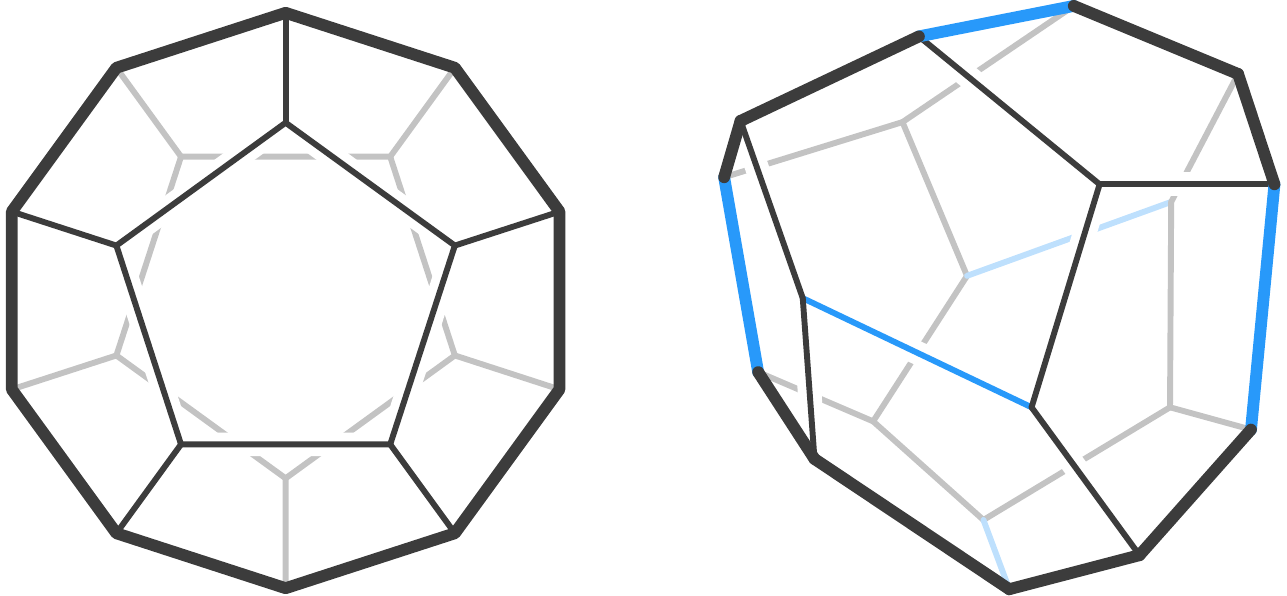}
    \end{minipage}\hspace*{12mm}
    \begin{minipage}[c]{0.32\textwidth}
    \includegraphics[width=\linewidth]{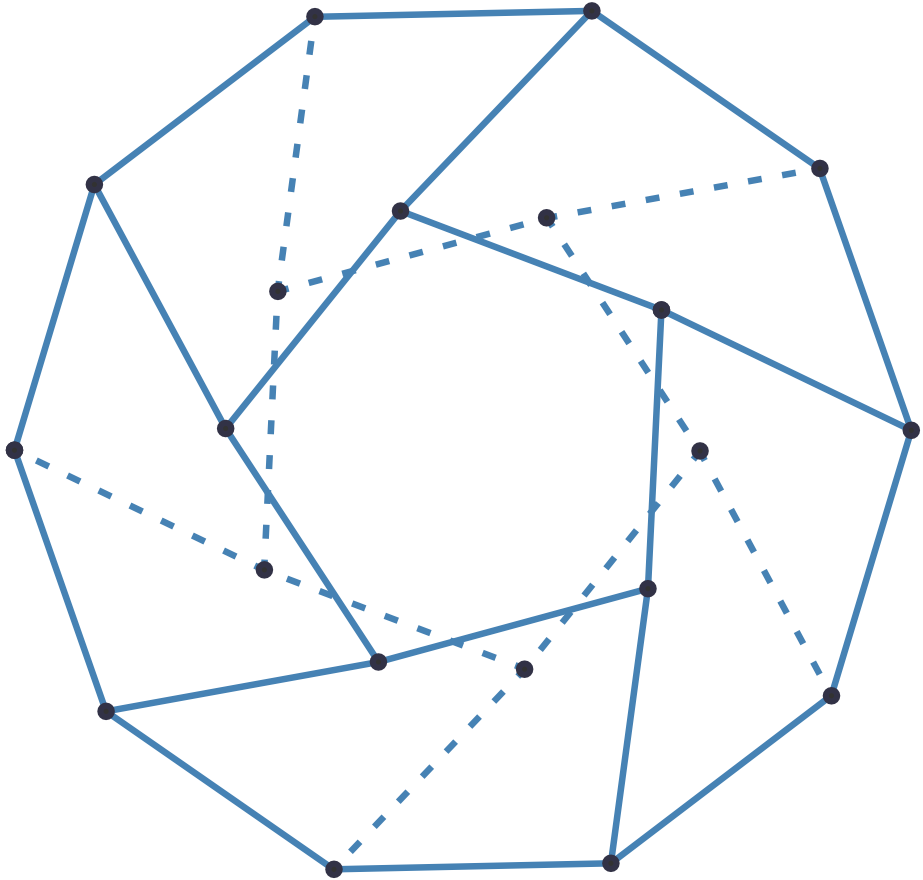}
    \end{minipage}
    \caption{The regular dodecahedron (left) and another realization, known as the ``Pyritohedron'', for which all edges have \emph{almost} the same length (center). The blue and black edges have slightly different lengths. Moreover, we depict a 2-dimensional drawing of the dodecahedral graph for which all edges have the same length (right). The flat dodecahedron is a reproduction and is courtesy to Bob Connelly and Simon Guest. }
    \label{fig:dodecahedron_variation}
    \label{fig:planar-dodec}
\end{figure}

\tempnewpage 

\section{Conclusion}
\label{sec:outlook}

In this article, we investigate deformations and develop a second-order rigidity theory for polytopes with edge length and coplanarity constraints.

\subsection{Main results}

We carefully develop a first-order and second-order theory for the rigidity of polytopes in \Cref{sec:first_order,sec:second-order-theory}. 
Using these tools, we demonstrate in \Cref{res:second_order} that second-order rigidity is sufficient for the rigidity of polytopes. 
Using equilibrium stresses, we subsequently develop an effective algorithm for checking second-order rigidity in \Cref{section:second-order-rigidity-check}. 
A well-developed second-order theory is important since it is right at the boundary of what is practically possible: analogous statements are no longer true for naive higher-order rigidity criteria (see \Cref{sec:higher_order_rigidity}). 

For 3-dimensional polytopes, we also investigate the topological properties of the realization space $\real(\mathcal{P})$ by considering the edge length perturbation space. This lets us generate novel conditions for rigidity (\eg\ \Cref{prop:inf-rigidity-edge-perturbations} and \Cref{thm:rigidity-and-edge-length-perturbations}). By applying the numerical curve-tracking scheme from \Cref{sec:approx_continuous_motions}, we generate animations of flexes and the curves induced by edge length modifications.

Using these tools, we characterize the rigidity properties of the five Platonic and 13 Archimedean solids in \Cref{tab:comparison}, and of the 13 Catalan solids in \Cref{tab:catalan-comparison}. Notably, we demonstrate that the regular dodecahedron and the truncated icosahedron  are second-order rigid, even though they are neither first-order rigid, nor prestress stable. Therefore, they represent examples of rigid polytopes, for which the first-order theory is insufficient. 
In \Cref{sec:so_rigidity_dod}, we explicitly demonstrate a strategy~for proving the second-order rigidity of the regular dodecahedron as a test case. Moreover, we investigate deformations induced by modifying the length of a single edge in \Cref{sec:dodecahedron}, which reveals a highly interesting local geometry and also serves as a counterexample to the generic global rigidity of polytopes.

\subsection{The Truncated Dodecahedron}

Among the 31 polytopes we considered, the truncated dodecahedron is the only one whose rigidity we are unable to determine. We are able to show that this polytope's space of first-order flexes is 4-dimensional, so it is not first-order rigid.
We also determine that it is not second-order rigid, since the associated system of homogeneous quadratic equations $Q_1=\dots=Q_4=0$ arising from \eqref{eqn:stress_polynomial_Qi} has a non-trivial solution. 
In fact, we 
find that all quadratic forms $Q_i$ are identically zero, and hence no first-order flex can be blocked by any stress. 
Still, with our current tools we are not able to distinguish whether it is flexible or rigid.

In \Cref{fig:TDod_flex_conc}, we depict samples from a numerically generated deformation of the truncated dodecahedron.
It is constructed from deforming along first-order flexes of high symmetry and subsequently restoring length and coplanarity constraints by gradient descent with reasonable numerical tolerances.
%
\begin{figure}[h!]
    \centering
    \includegraphics[width=0.9\linewidth]{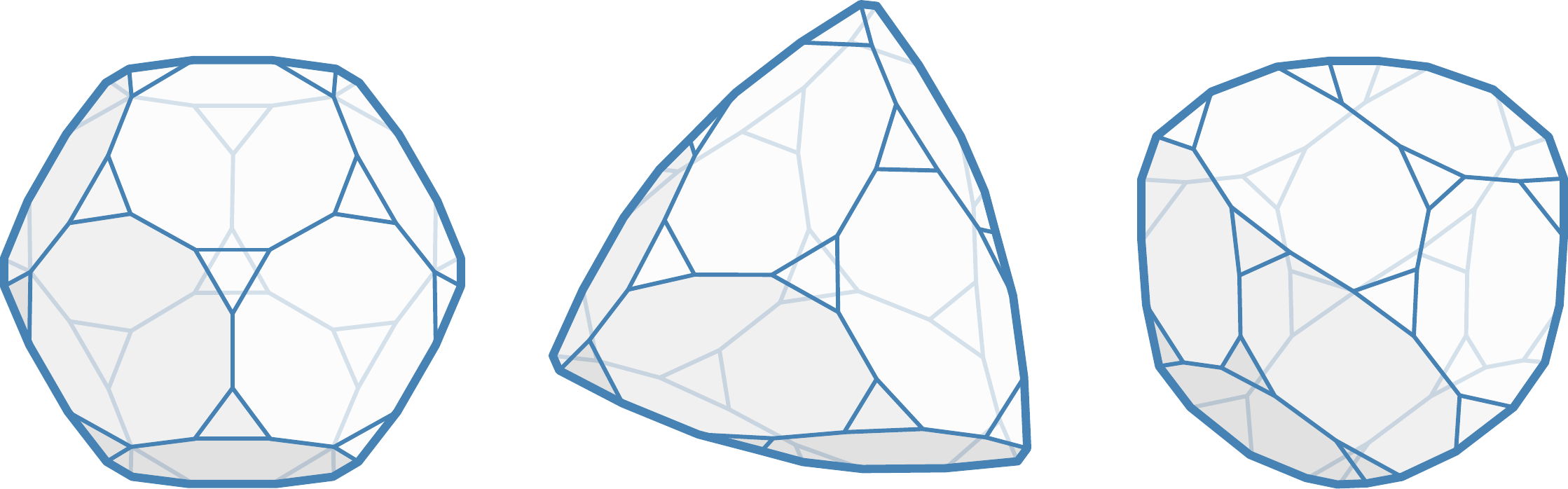}
    \caption{The Archimedean truncated dodecahedron (left) and two of its potential flexes (middle and right) constructed numerically.}
    \label{fig:TDod_flex_conc}
\end{figure}
However, when we apply a more robust numerical scheme (\cf\ \Cref{sec:approx_continuous_motions}) with a tolerance close to machine precision, all of the realizations from the flex visualized in \Cref{fig:TDod_flex_conc} converge back to the maximally symmetric truncated dodecahedron. Therefore, this symmetric deformation is likely not a true flex. 

\subsection{Higher-Order Rigidity}
\label{sec:higher_order_rigidity}
Initial experiments using the energy-based approach for higher-order rigidity of Gortler, Holmes-Cerfon and Theran \cite{gortler2025higherorderrigidityenergy} suggest that the truncated dodecahedron is instead third-order rigid. 
This result would help explain why we are able to find a deformation of the truncated dodecahedron in \Cref{fig:TDod_flex_conc}, which 
satisfies the edge length and facet planarity constraints with reasonably small numerical tolerance but does not hold up to more robust numerical tests. In \cite{gortler2025higherorderrigidityenergy}, the authors observe that the energy landscape around an $N$-th order rigid realization locally looks like the 
polynomial curve $y=x^{2N}$. This implies that, for larger $N$, one would have to deviate relatively far from the original realization to see a noticeable increase in the energy.

If the dimension of the space of first-order flexes is greater than one, there currently exists no effective and conclusive algorithm for checking higher-order rigidity, and we would have to solve infinitely many polynomial optimization problems. It is a problem for future research to apply hierarchies from the polynomial optimization literature (\eg\ Lasserre hierarchy) to limit the number of cases that need to be checked. Solving this issue would lead to an algorithmic framework which would be applicable to various geometric constraint systems across rigidity theory, and not just polytopes.

%% file: sec/table1.tex
\def\arraystretch{1.5} 

\begin{table}[h!]
\centering
\hspace*{-9mm}
\begin{tabular}{|clc|c|clc|}
\cline{1-3}\cline{5-7}

\makecell{\\[7mm]\SymTwo\hspace*{-5mm}} & \makecell{Tetrahedron\\[2.5mm] (\yes , \yes, \yes, \yes , 0)} & \cincludegraphics{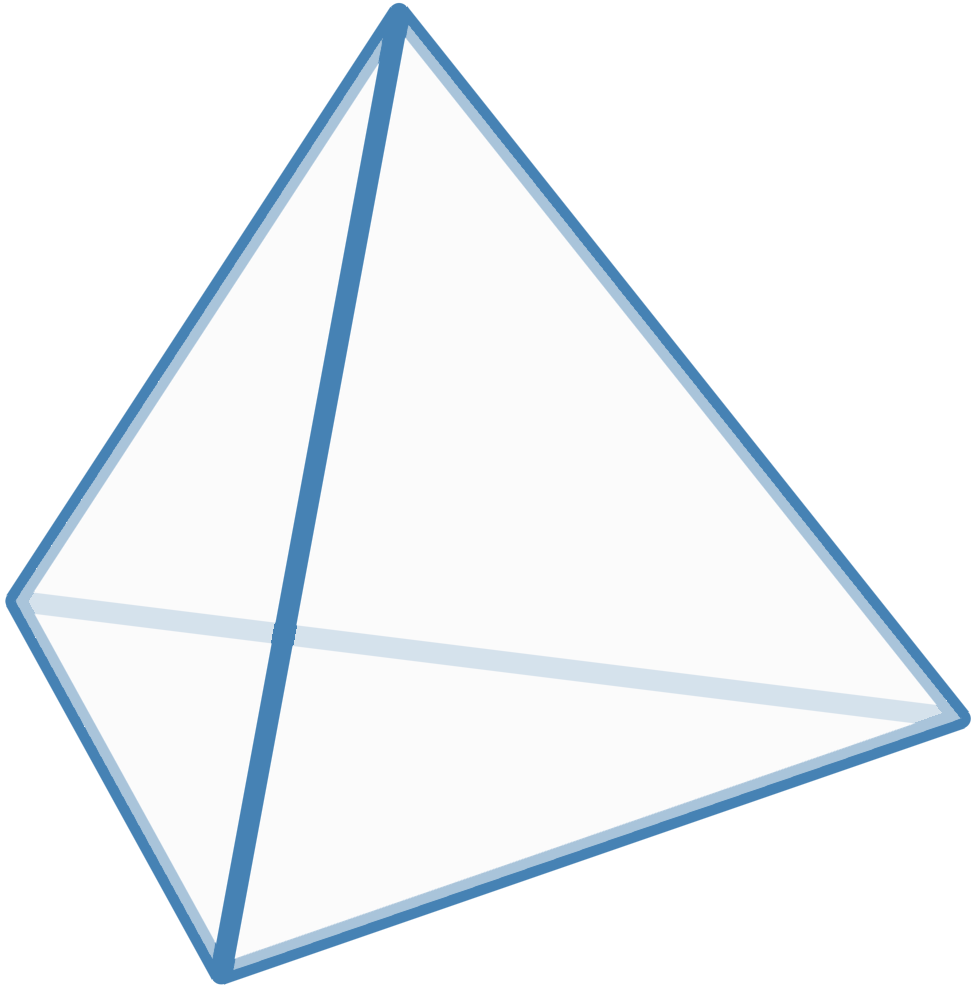} & &
 & \makecell{Truncated Tetrahedron\\[2.5mm](\yes , \yes, \yes , \yes , 0)} & \cincludegraphics{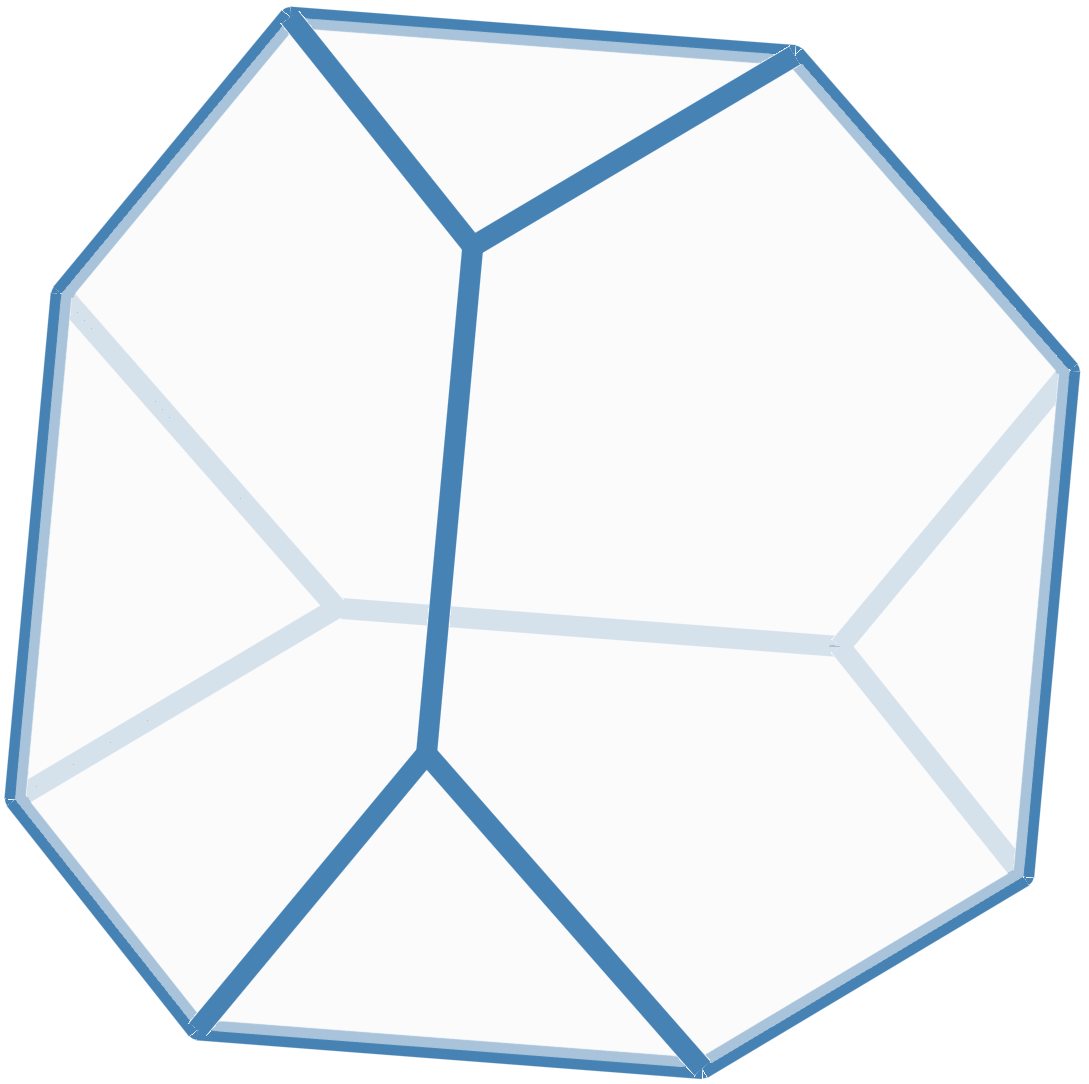} \\
\customdoubleline

\makecell{\\[7mm]\SymOne\rlap{\;\SymThree}\hspace*{-5mm}} & \makecell{Cube\\[2.5mm](\no , \no, \no , \no , 3) }&\cincludegraphics{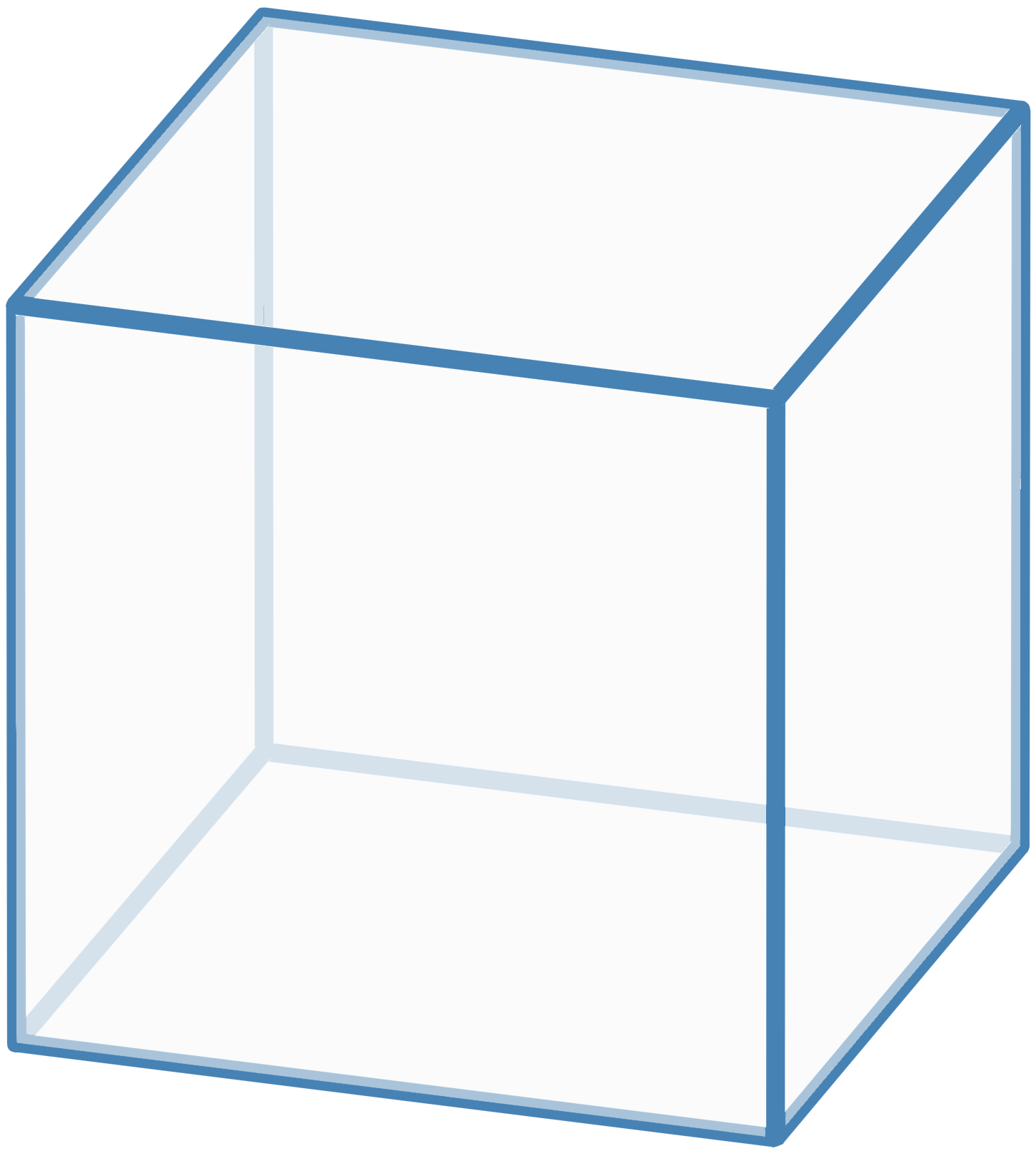} & &
\makecell{\\[7mm]\SymTwo\hspace*{-5mm}} & \makecell{Octahedron \\[2.5mm] (\yes , \yes, \yes , \yes , 0)} & \cincludegraphics{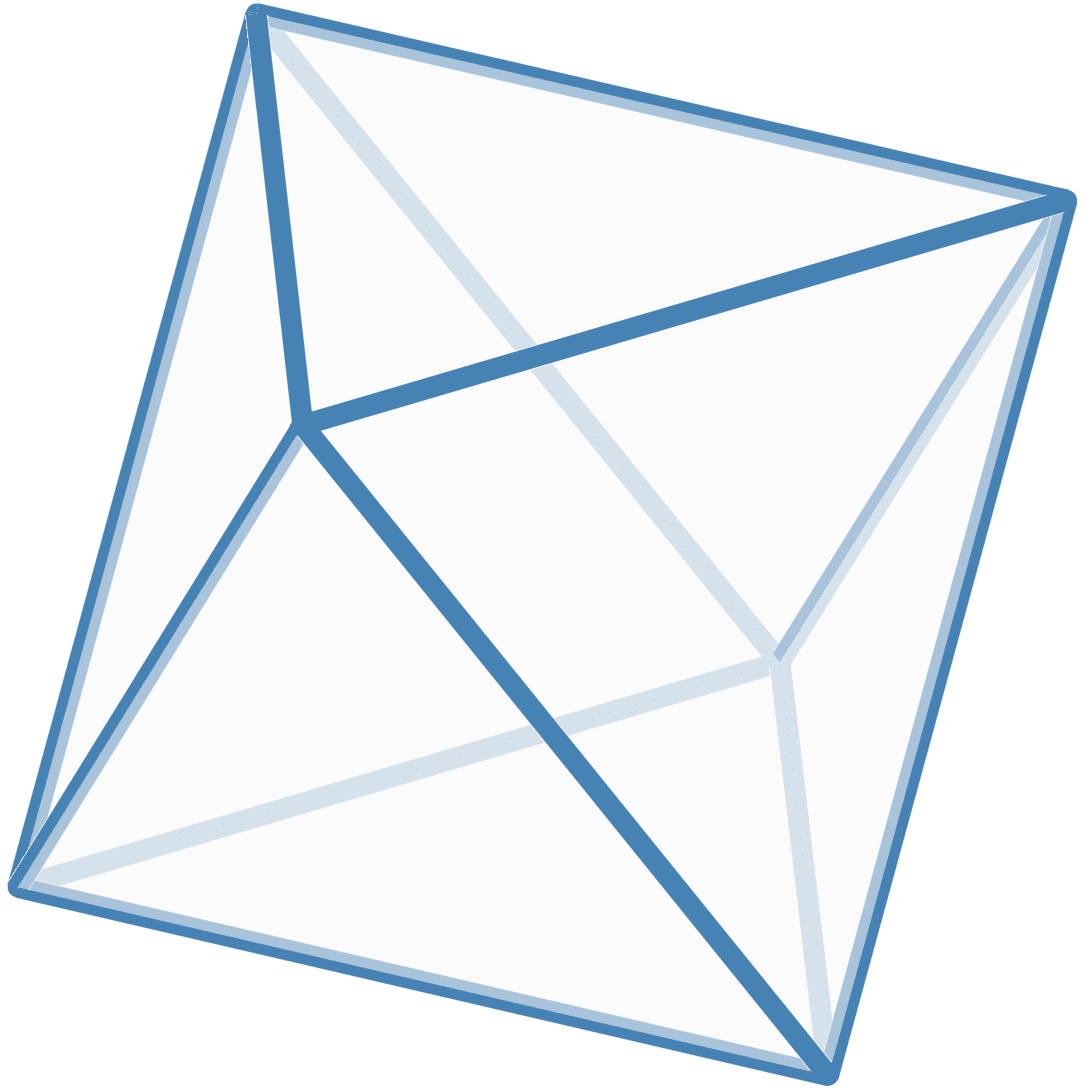} \\ \cline{1-3}\cline{5-7}

\makecell{\\[7mm]\SymThree\hspace*{-5mm}} & \makecell{Cuboctahedron\\[2.5mm] (\no , \no , \no , \no , 0+3)} &\cincludegraphics{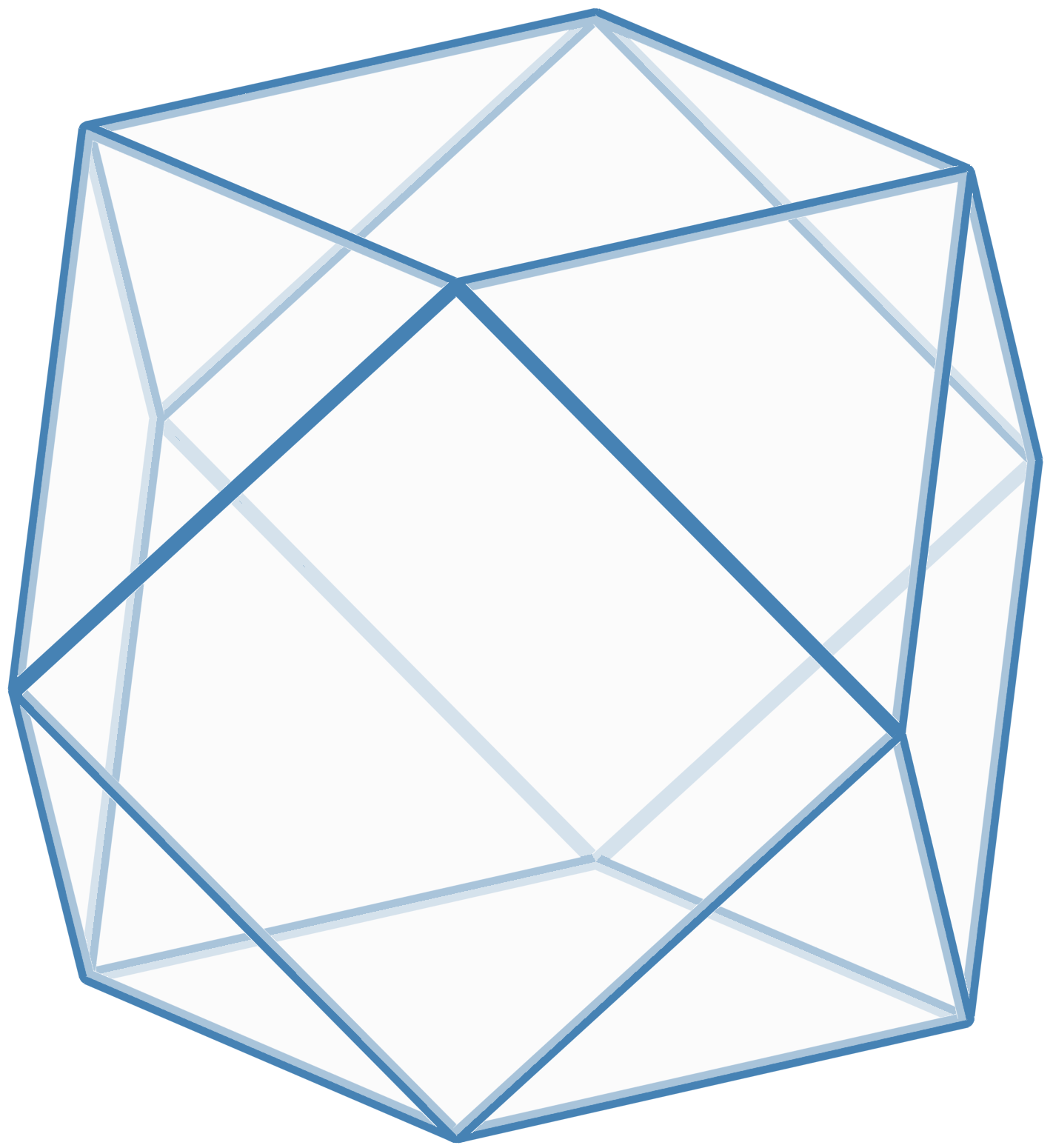} &  & &
\makecell{Truncated Cube\\[2.5mm](\yes , \yes , \yes , \yes , 0) } & \cincludegraphics{Images/truncatedcube.png} \\ \cline{1-3}\cline{5-7}

\makecell{\\[7mm]\SymOne\hspace*{-5mm}} & \makecell{Truncated Octahedron\\[2.5mm] (\no , \no , \no , \no , 5)} & \cincludegraphics{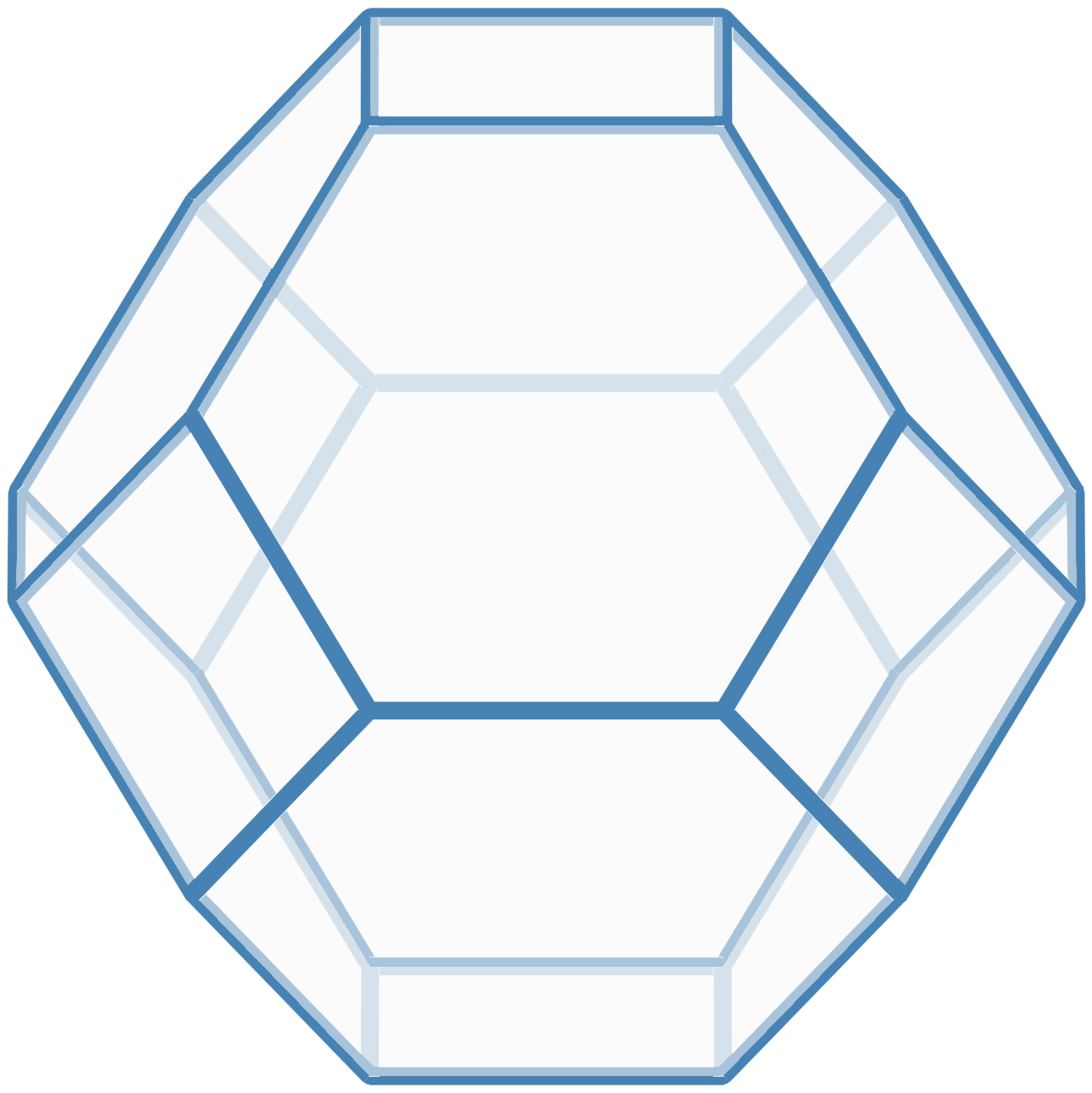} & & \makecell{\\[7mm]\SymThree\hspace*{-5mm}} &
\makecell{Rhombicuboctahedron\\[2.5mm] (\no , \no , \no , \no , 3+3)} & \cincludegraphics{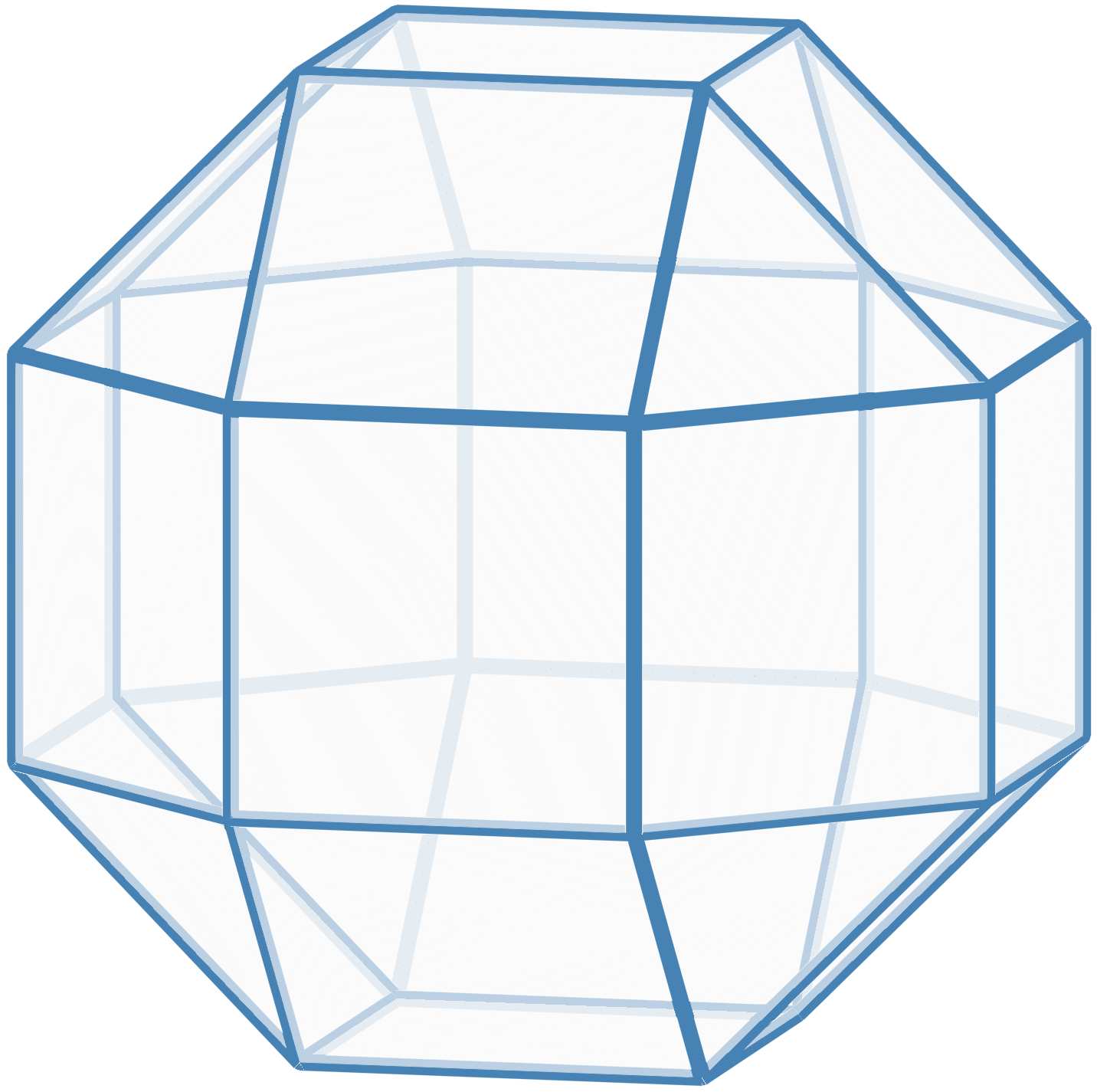}\\ \cline{1-3}\cline{5-7}

\makecell{\\[7mm]\SymOne\hspace*{-5mm}} & \makecell{Truncated Cuboctahedron\\[2.5mm] (\no , \no , \no , \no , 5)} & \cincludegraphics{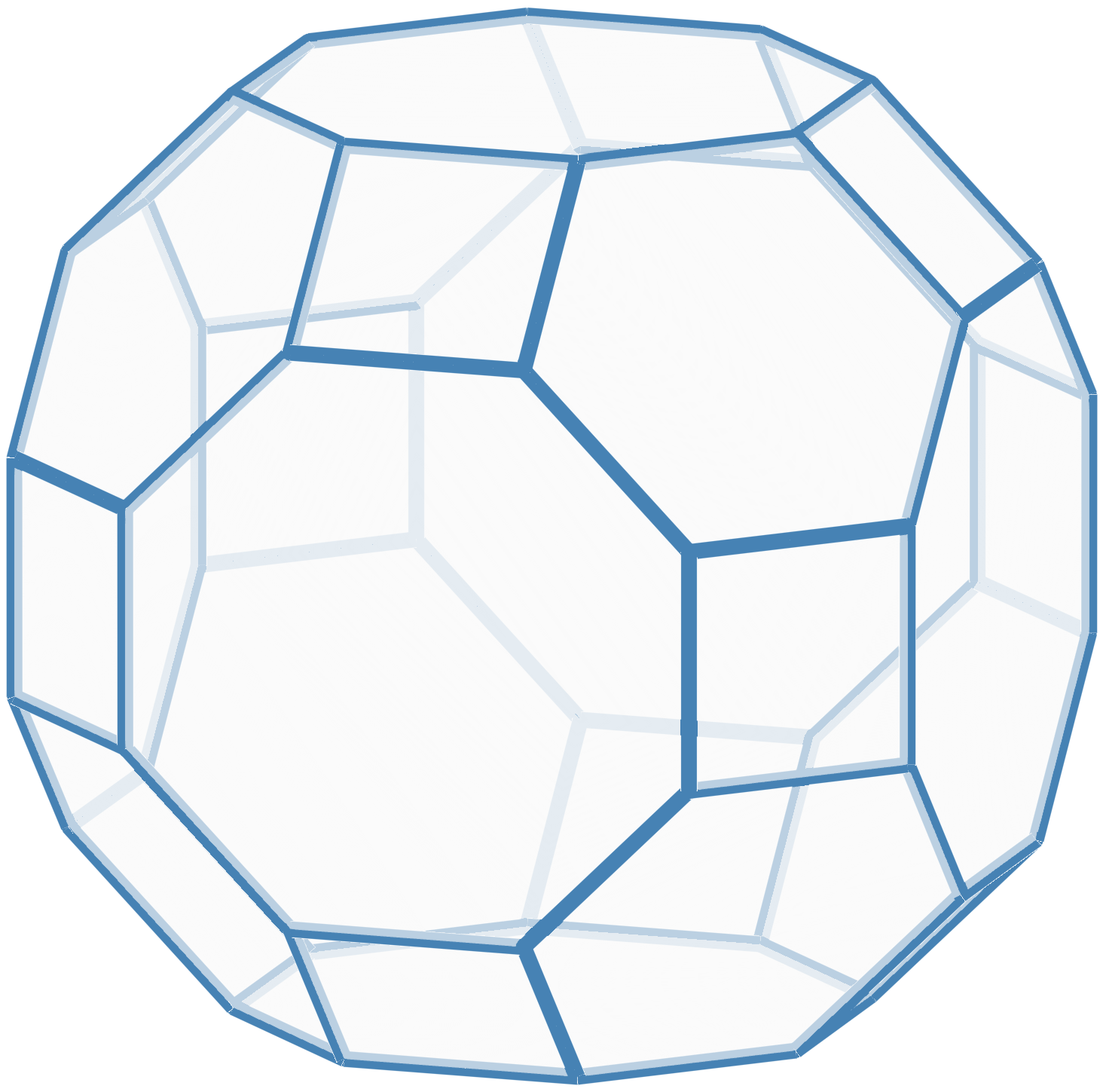} & & &
\makecell{Snub Cube\\[2.5mm](\yes , \yes , \yes , \yes , 0)} & \cincludegraphics{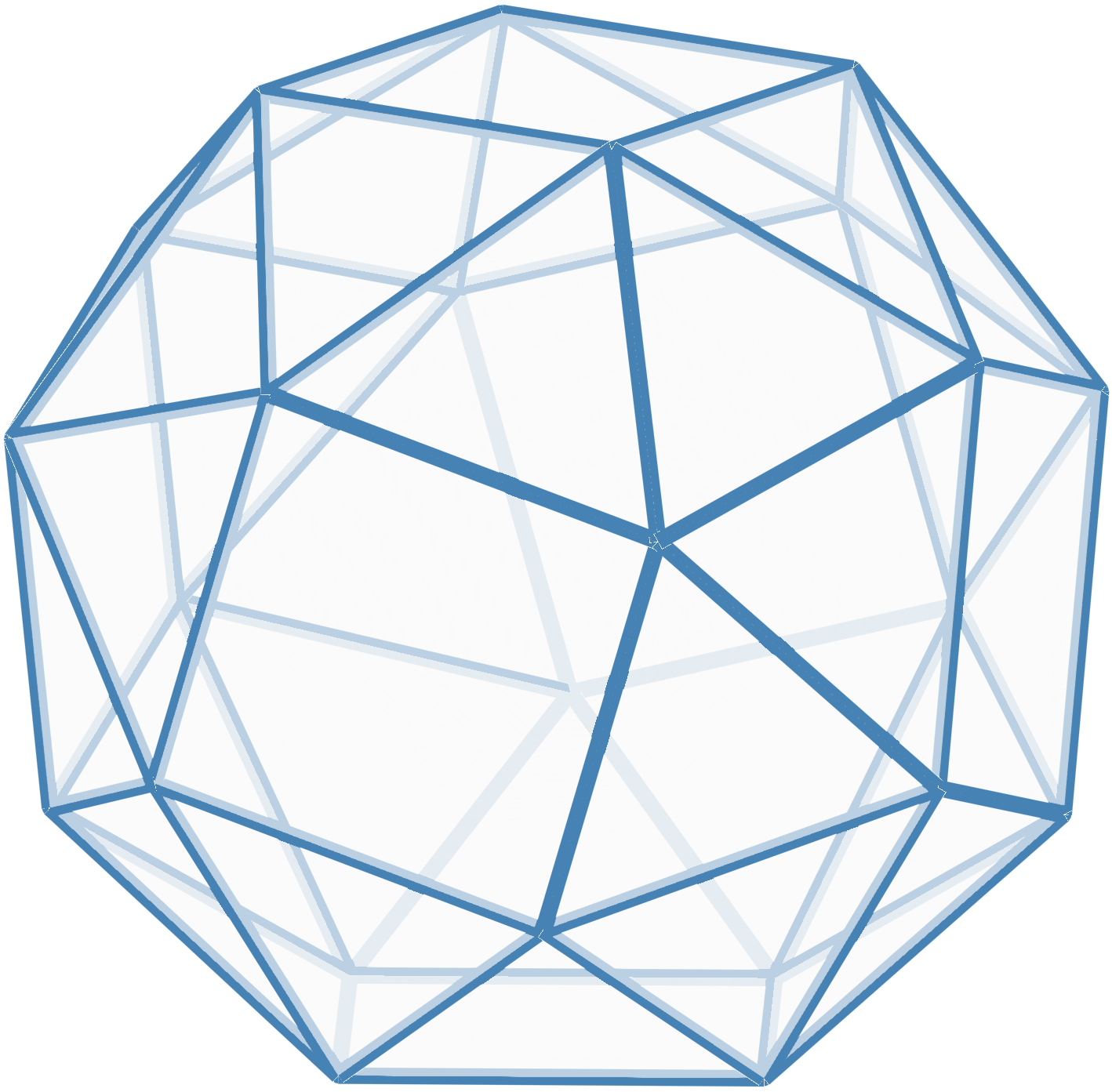}\\ \customdoubleline

\makecell{\\[7mm]\SymTwo\hspace*{-5mm}} & \makecell{Icosahedron \\[2.5mm] (\yes , \yes , \yes , \yes , 0)} & \cincludegraphics{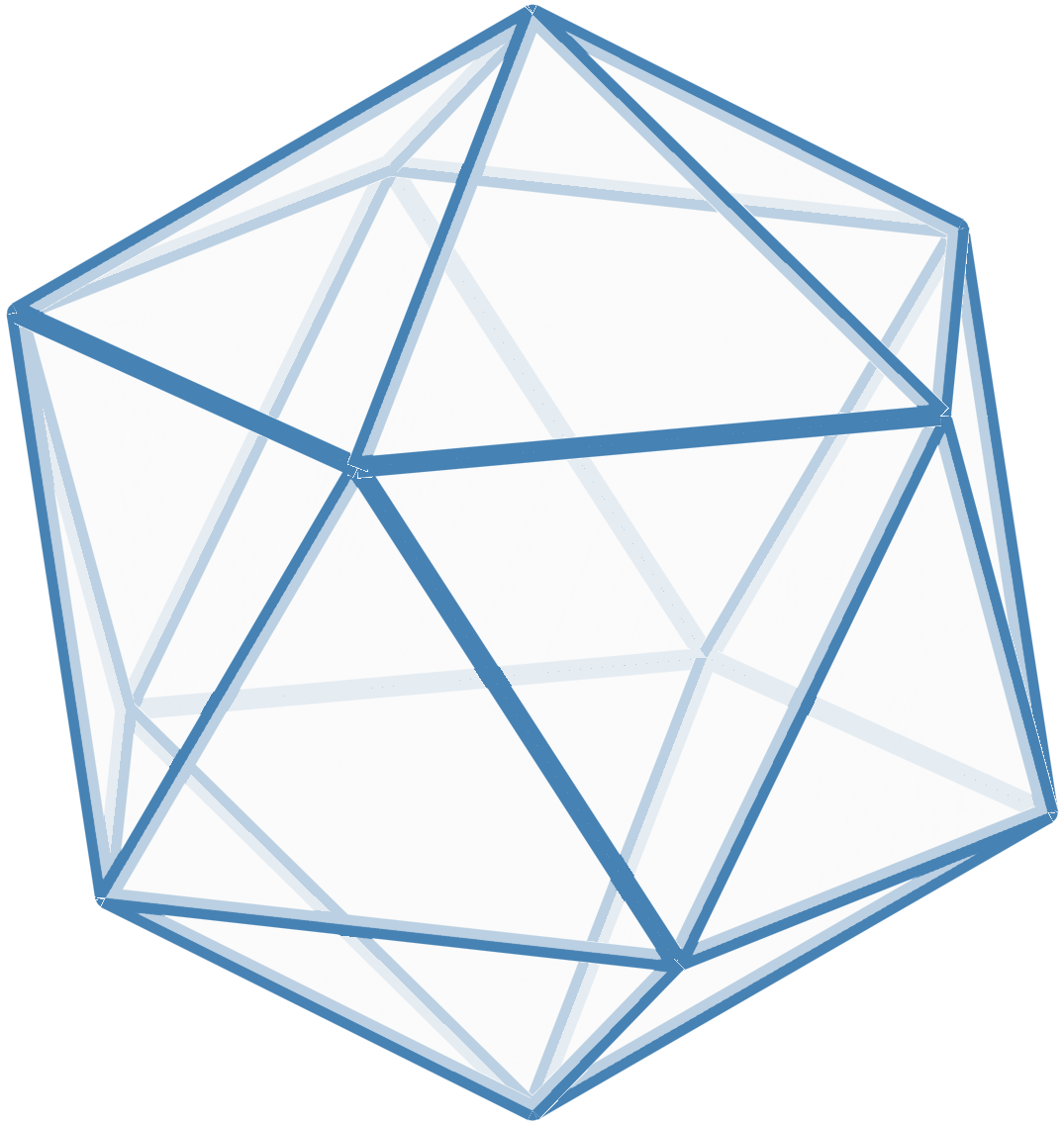} & & \makecell{\\[7mm]\Star\hspace*{-5mm}} &
\makecell{Dodecahedron\\[2.5mm](\no , \noBut , \yes , \yes , 5)} & \cincludegraphics{Images/dodecahedron.png}\\ \cline{1-3}\cline{5-7}

& \makecell{Icosidodecahedron\\[2.5mm](\yes , \yes , \yes , \yes , 0)} & \cincludegraphics{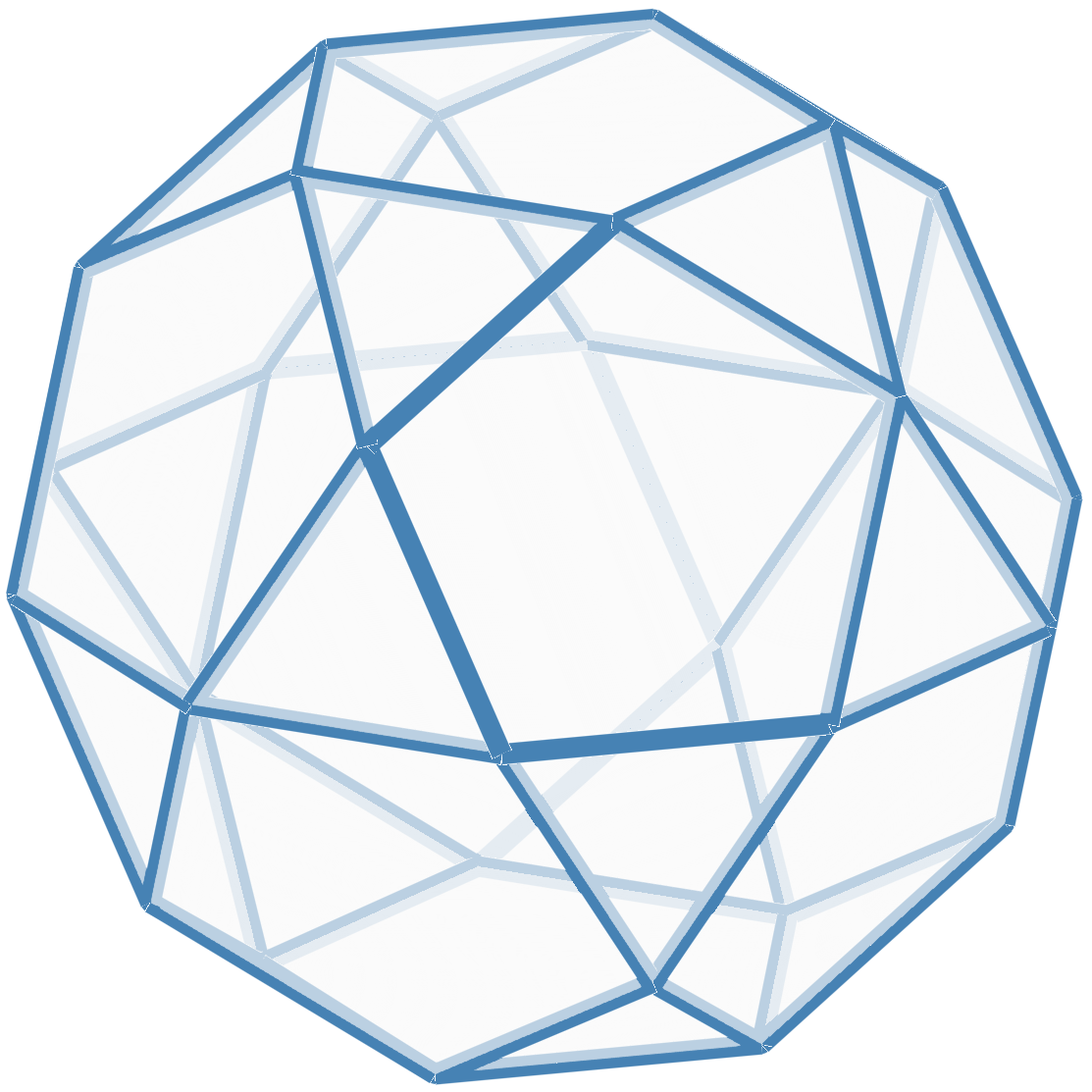} & & \makecell{\\[7mm]\Star\hspace*{-5mm}} &
\makecell{Truncated Dodecahedron\\[2.5mm](\no , \no , \no ,  \Ques , 4)} & \cincludegraphics{Images/truncatedDodecahedron.png}\\ \cline{1-3}\cline{5-7}

\makecell{\\[7mm]\Star\hspace*{-5mm}}
& \makecell{Truncated Icosahedron\\[2.5mm](\no , \noBut , \yes , \yes , 5)} & \cincludegraphics{Images/truncatedIcosahedron.png} & & \makecell{\\[7mm]\SymThree\hspace*{-5mm}} &
\makecell{Rhombicosidodecahedron\\[2.5mm] (\no , \no , \no , \no , 5+3)} & \cincludegraphics{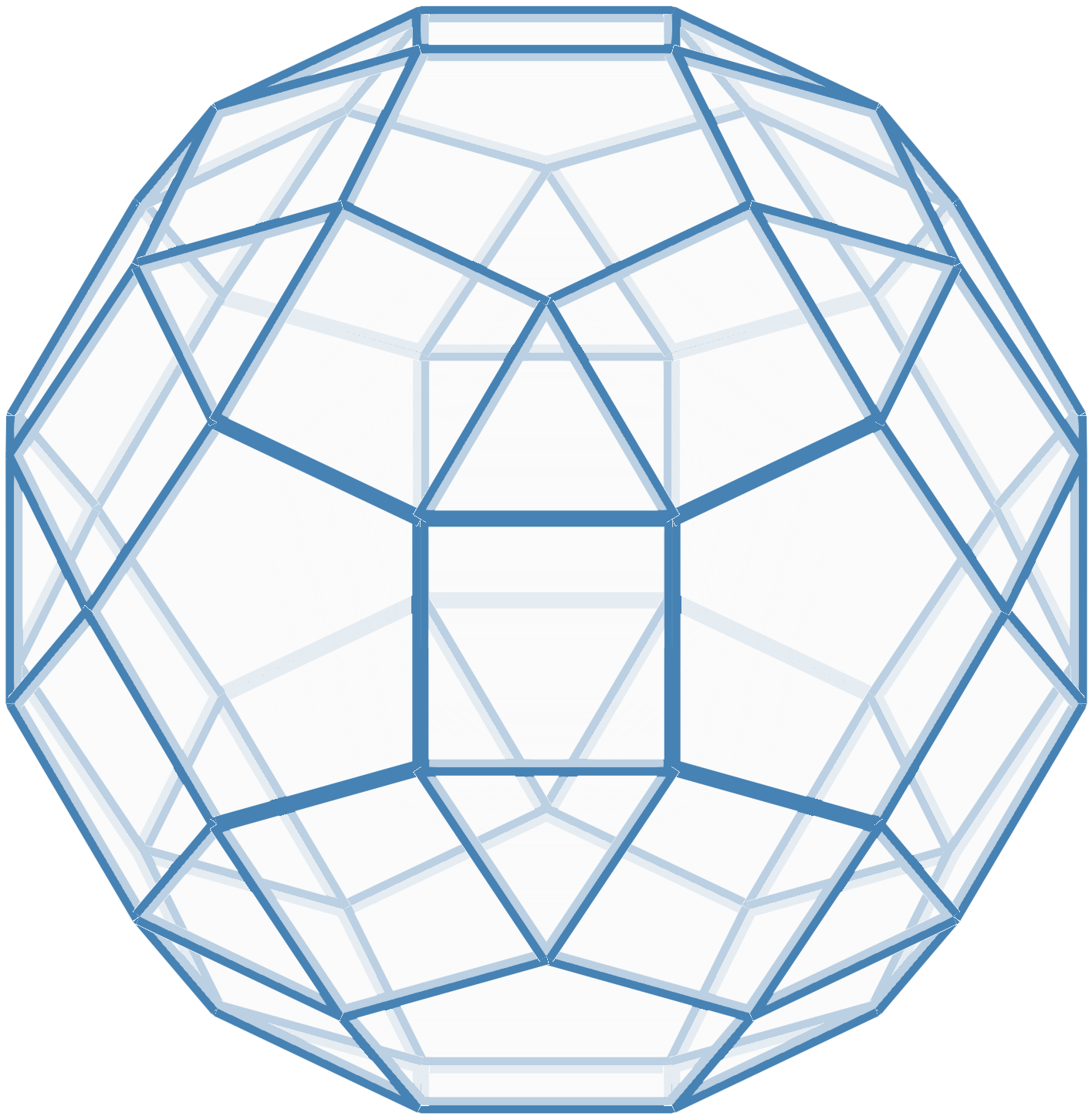}\\ \cline{1-3}\cline{5-7}

\makecell{\\[7mm]\SymOne}\hspace*{-5mm} & \makecell{Truncated Icosidodecahedron\\[2.5mm] (\no , \no , \no , \no , 5)} & \cincludegraphics{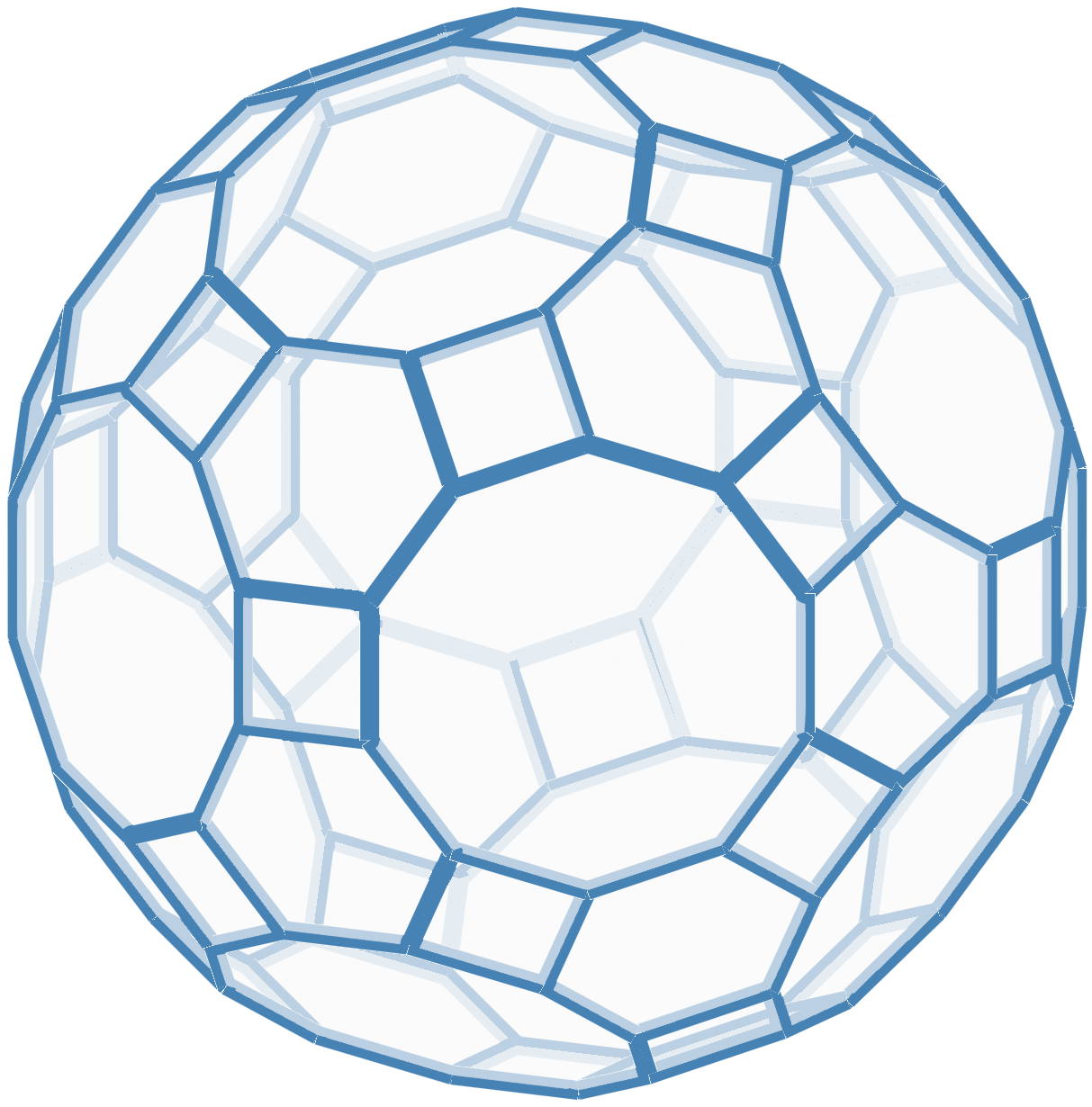} & & &
\makecell{Snub Dodecahedron\\[2.5mm](\yes , \yes , \yes , \yes , 0)} & \cincludegraphics{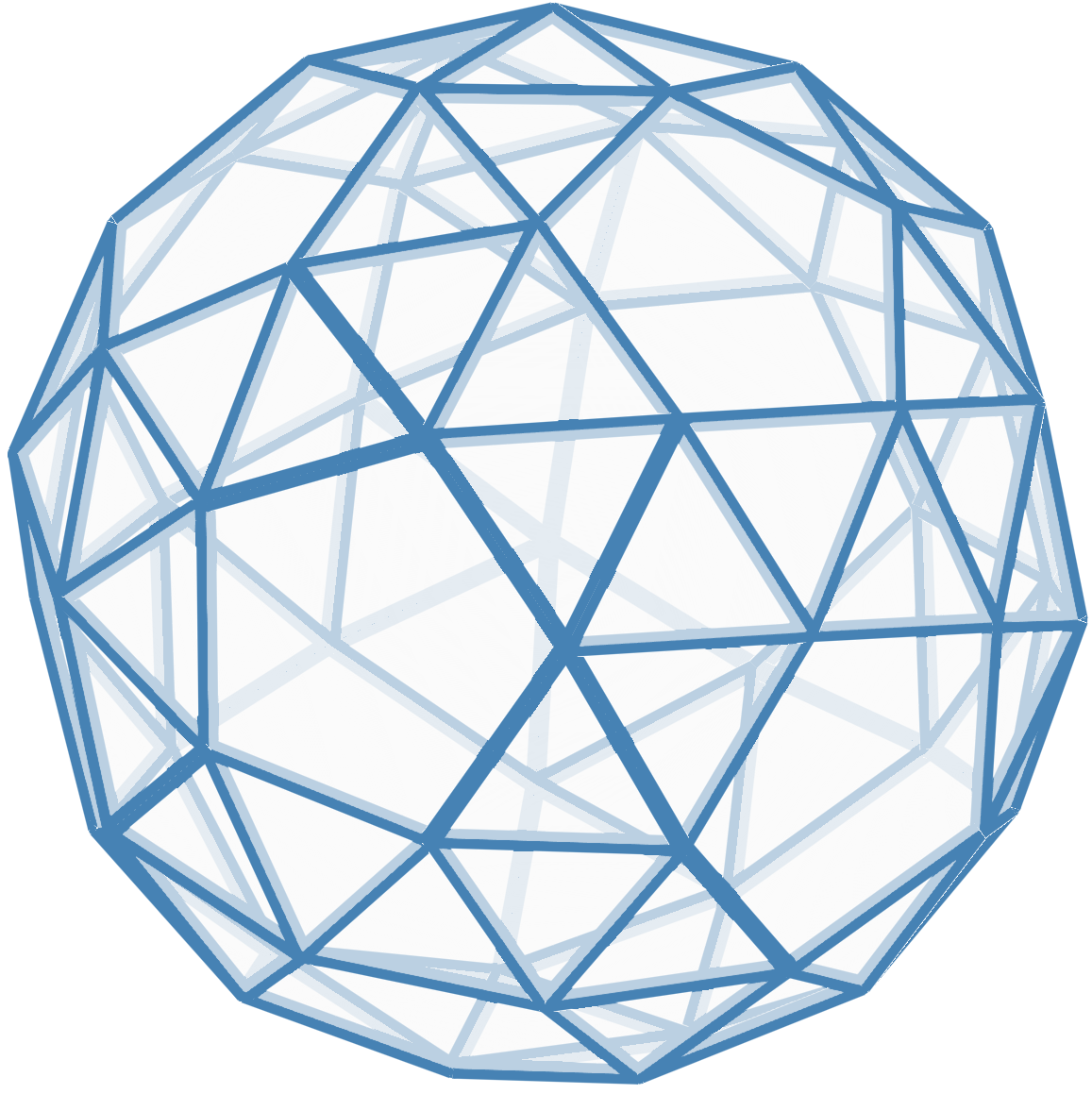}\\ \cline{1-3}\cline{5-7}

\end{tabular}
    \caption{Rigidity properties of the five Platonic and 13 Archimedean solids grouped by the polytopes' symmetry. 
    For an explanation of the cell contents, see \cref{sec:explain_table}.
    }
    \label{tab:comparison}
\end{table}

%% file: sec/table2.tex
\begin{table}[ht!]
\centering
\hspace*{-9mm}
\begin{tabular}{|clc|c|clc|}
\cline{1-3}

\makecell{\\[7mm]\SymTwo\hspace*{-5mm}} &
\makecell{Triakis Tetrahedron\\[2.5mm](\yes , \yes , \yes , \yes , 0)} & \cincludegraphics{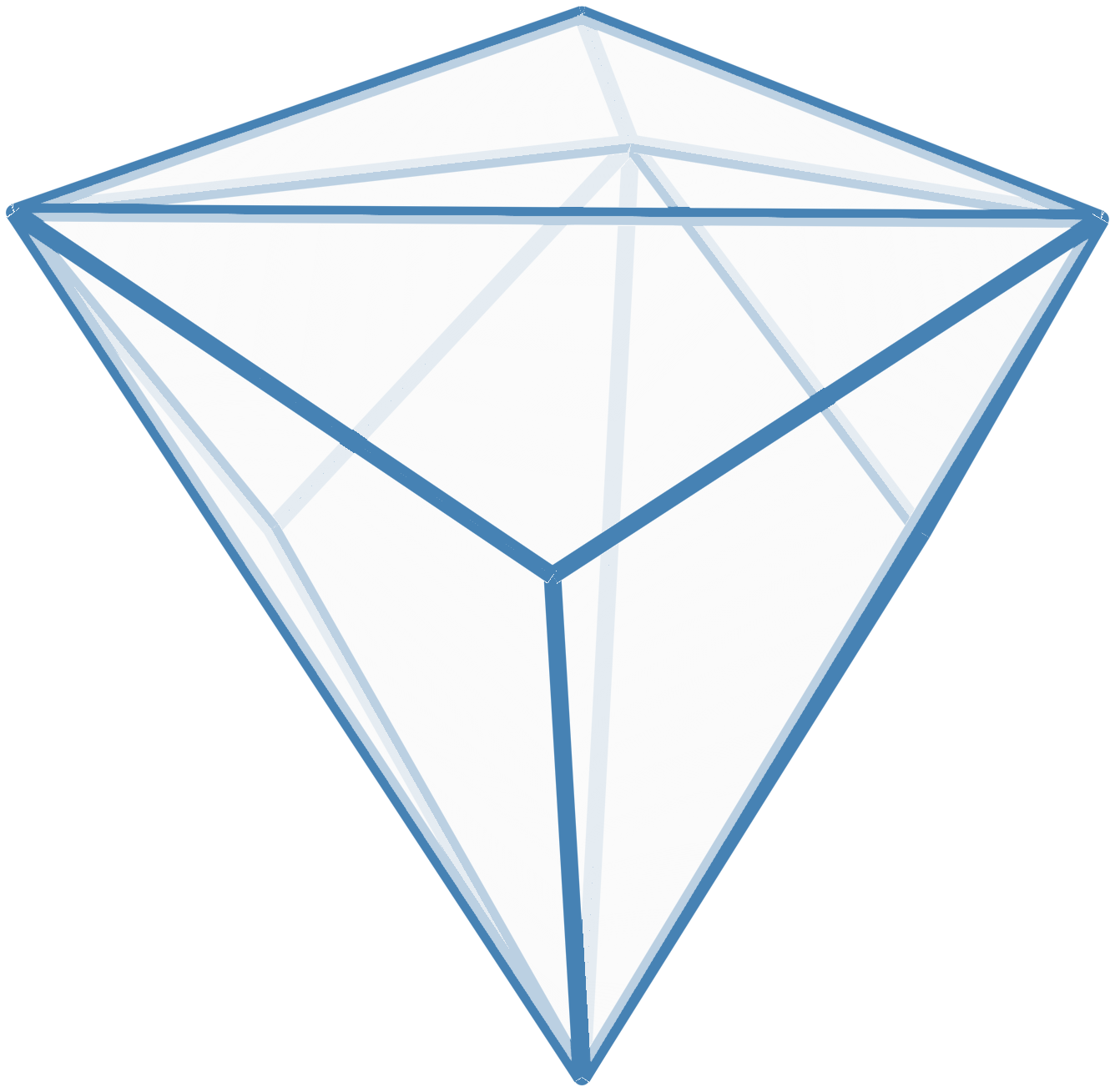} & \multicolumn{4}{c}{~} \\
\customsingletwo

\makecell{\\[7mm]\SymOne\rlap{\;\SymThree}\hspace*{-5mm}} & \makecell{Rhombic Dodecahedron \\[2.5mm](\no , \no , \no , \no , 5)}&\cincludegraphics{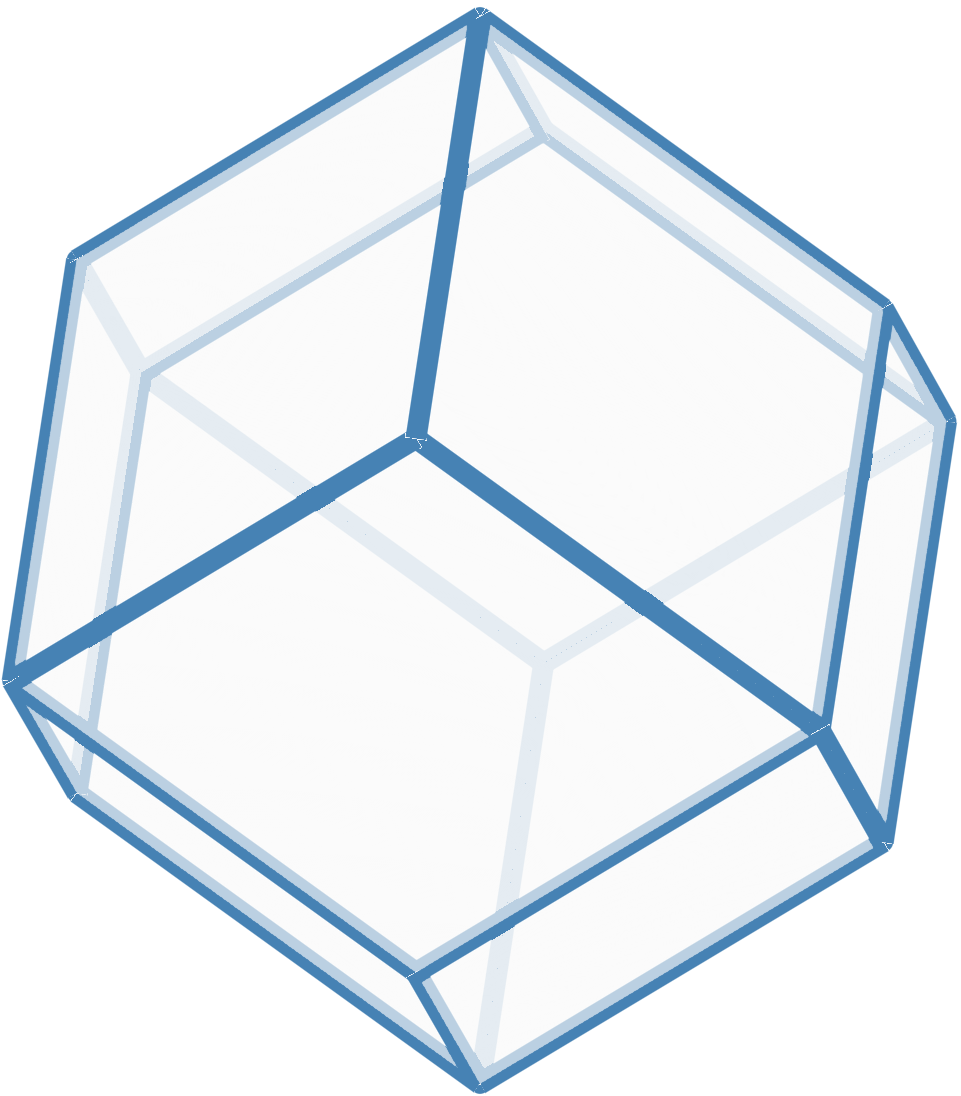} &  &
\makecell{\\[7mm]\SymTwo\hspace*{-5mm}} & \makecell{Triakis Octahedron \\[2.5mm](\yes , \yes , \yes , \yes , 0) } & \cincludegraphics{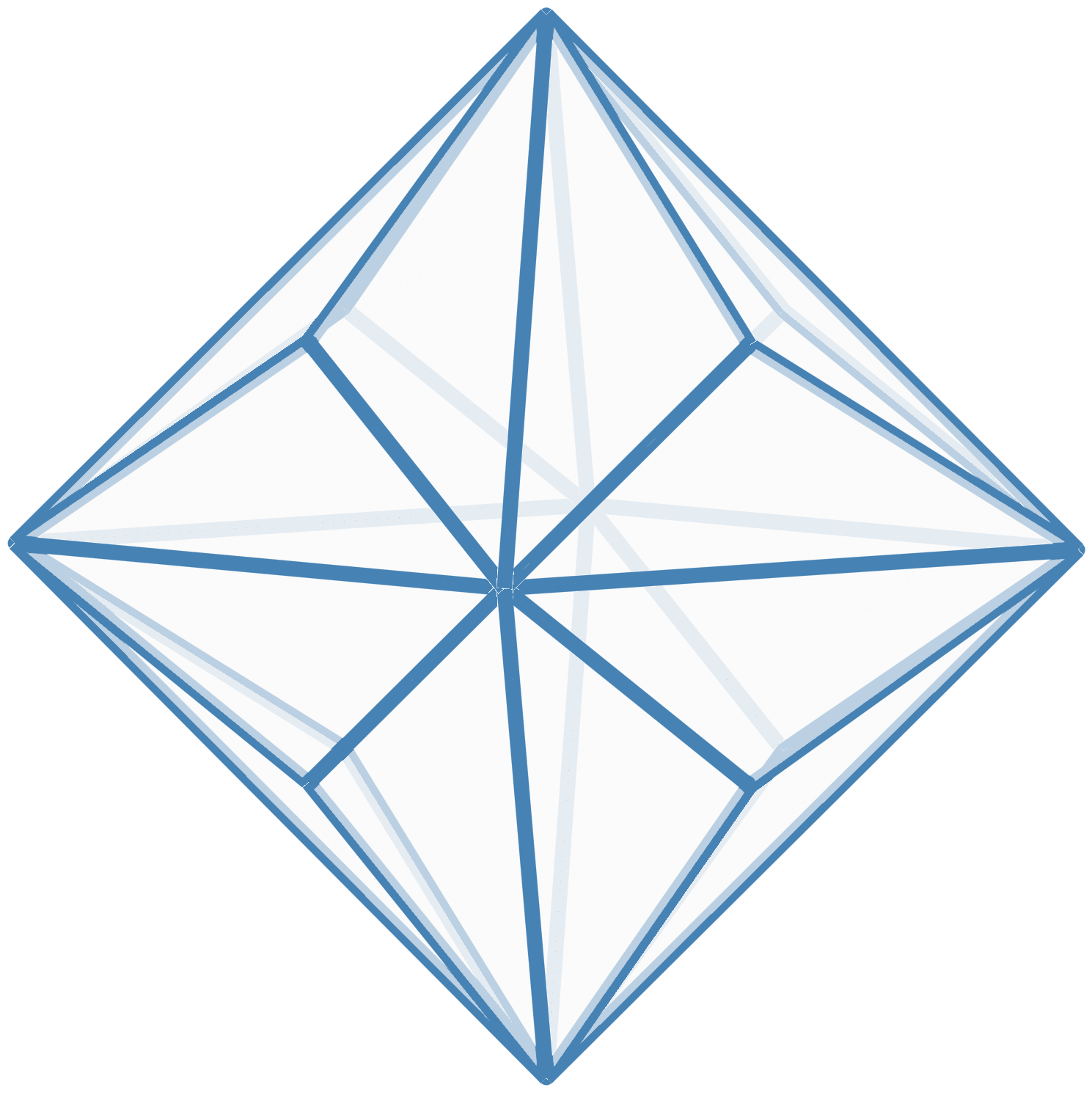} \\ \cline{1-3}\cline{5-7}

\makecell{\\[7mm]\SymTwo\hspace*{-5mm}} & \makecell{Tetrakis Hexahedron\\[2.5mm](\yes , \yes , \yes , \yes , 0)} & \cincludegraphics{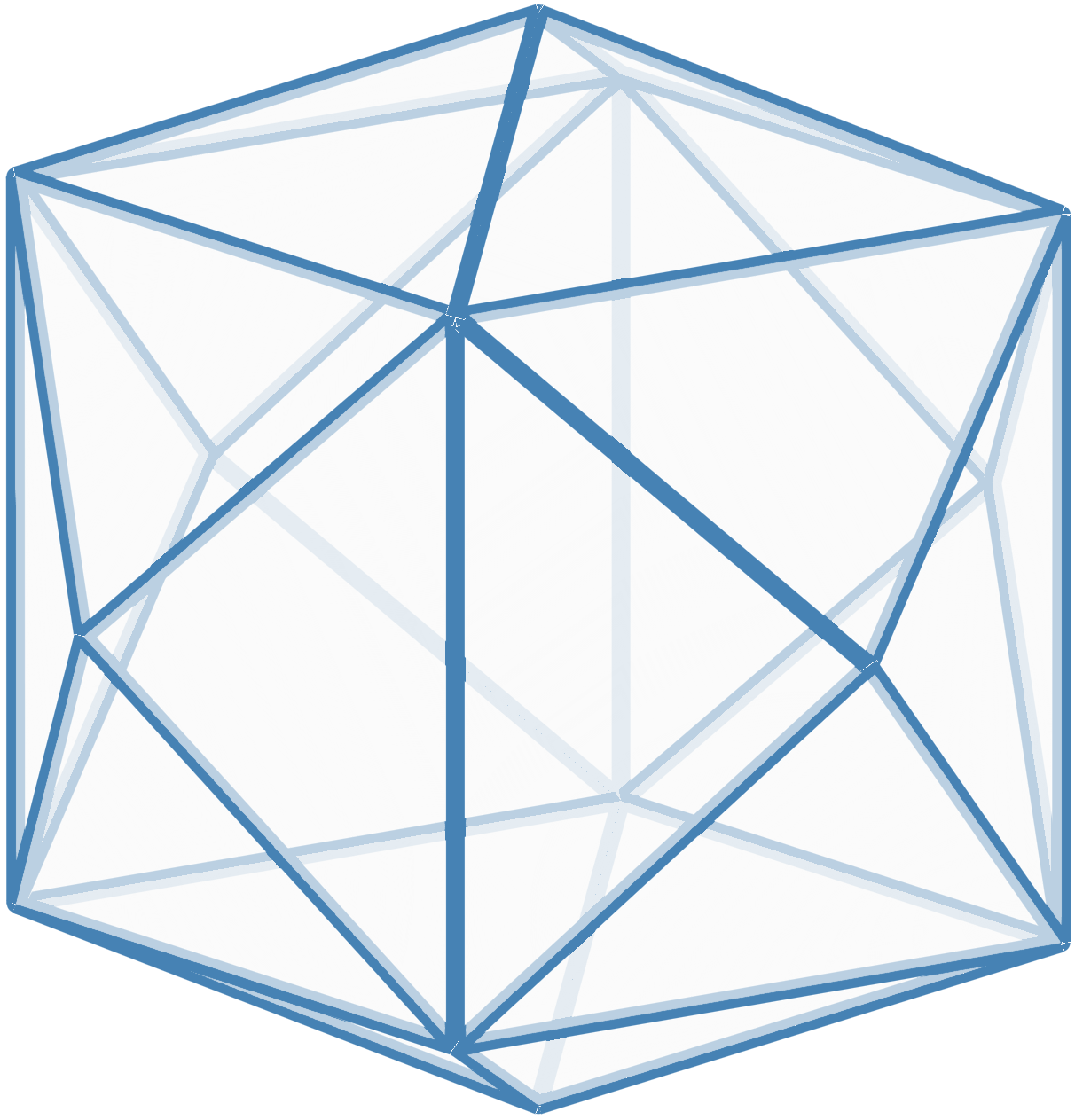} & & &
\makecell{Deltoidal Icositetrahedron\\[2.5mm](\yes , \yes , \yes , \yes , 0)} & \cincludegraphics{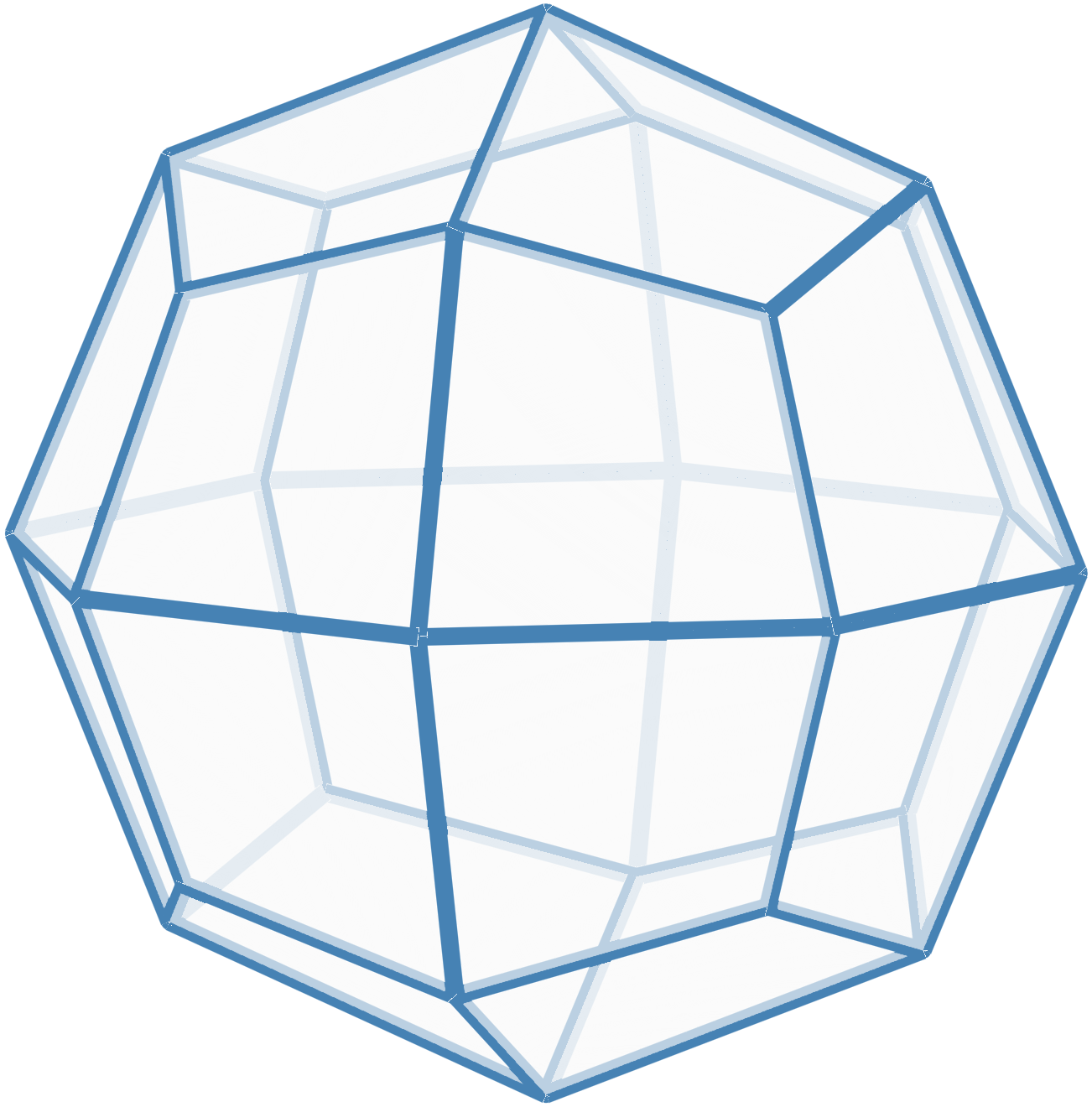}\\ \cline{1-3}\cline{5-7}

\makecell{\\[7mm]\SymTwo\hspace*{-5mm}} & \makecell{Disdyakis Dodecahedron\\[2.5mm](\yes , \yes , \yes , \yes , 0)} & \cincludegraphics{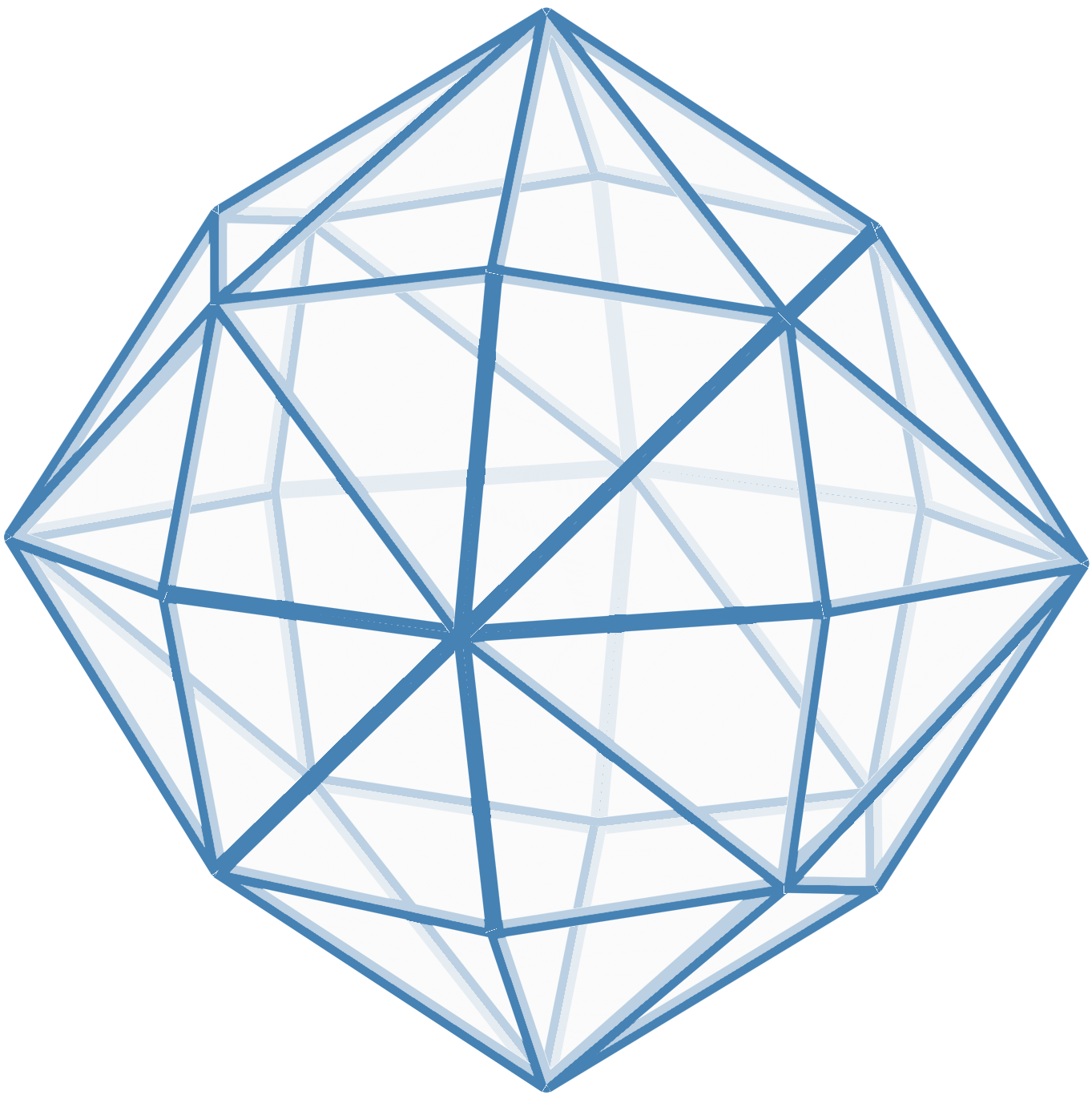} & & &
\makecell{Pentagonal Icositetrahedron\\[2.5mm](\yes , \yes , \yes , \yes , 0)} & \cincludegraphics{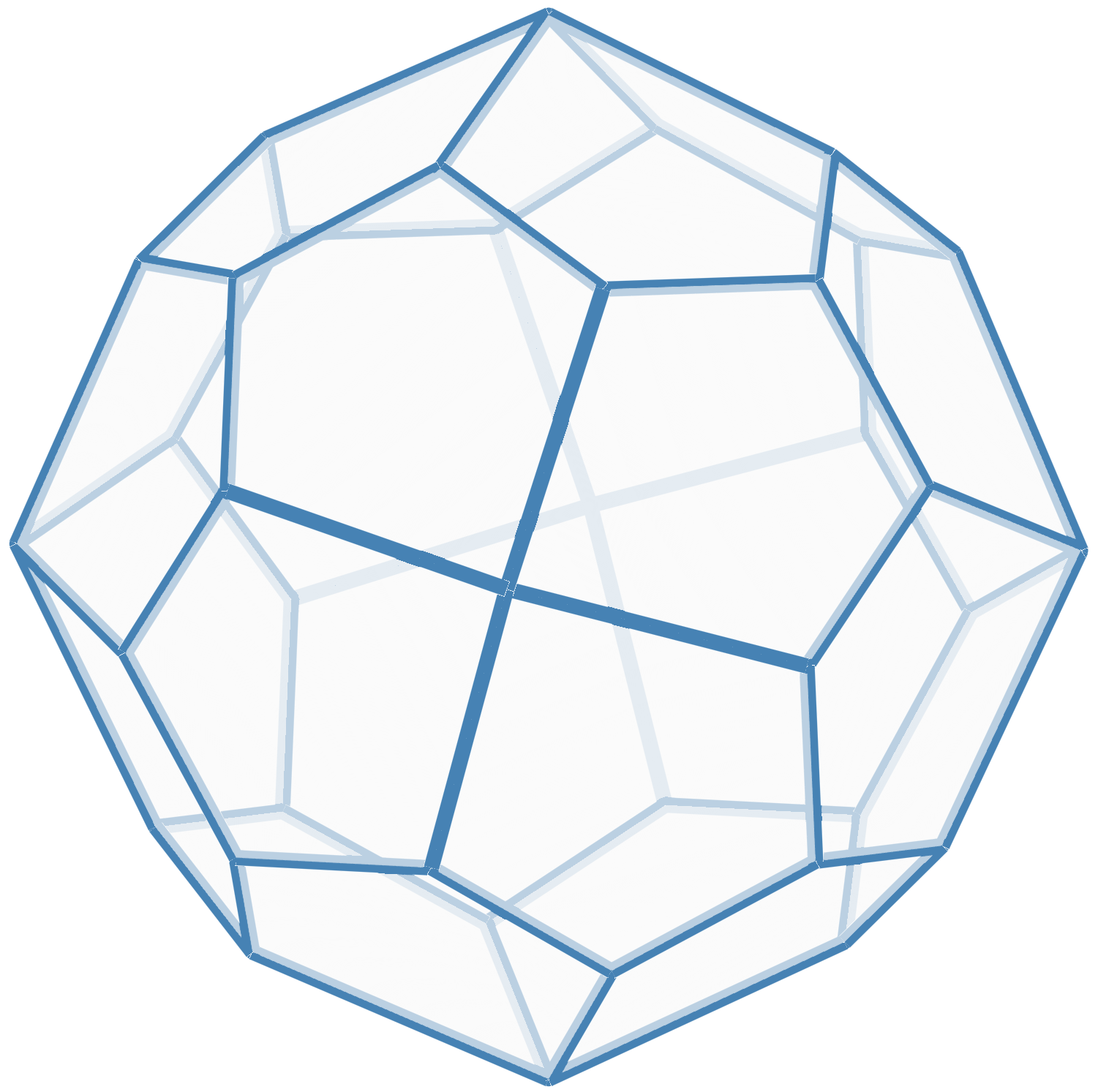}\\[2.5mm] \customdoublelinetwo

\makecell{\\[7mm]\SymOne\rlap{\;\SymThree}\hspace*{-5mm}} & \makecell{Rhombic Triacontahedron\\[2.5mm](\no , \no , \no , \no , 9)} & \cincludegraphics{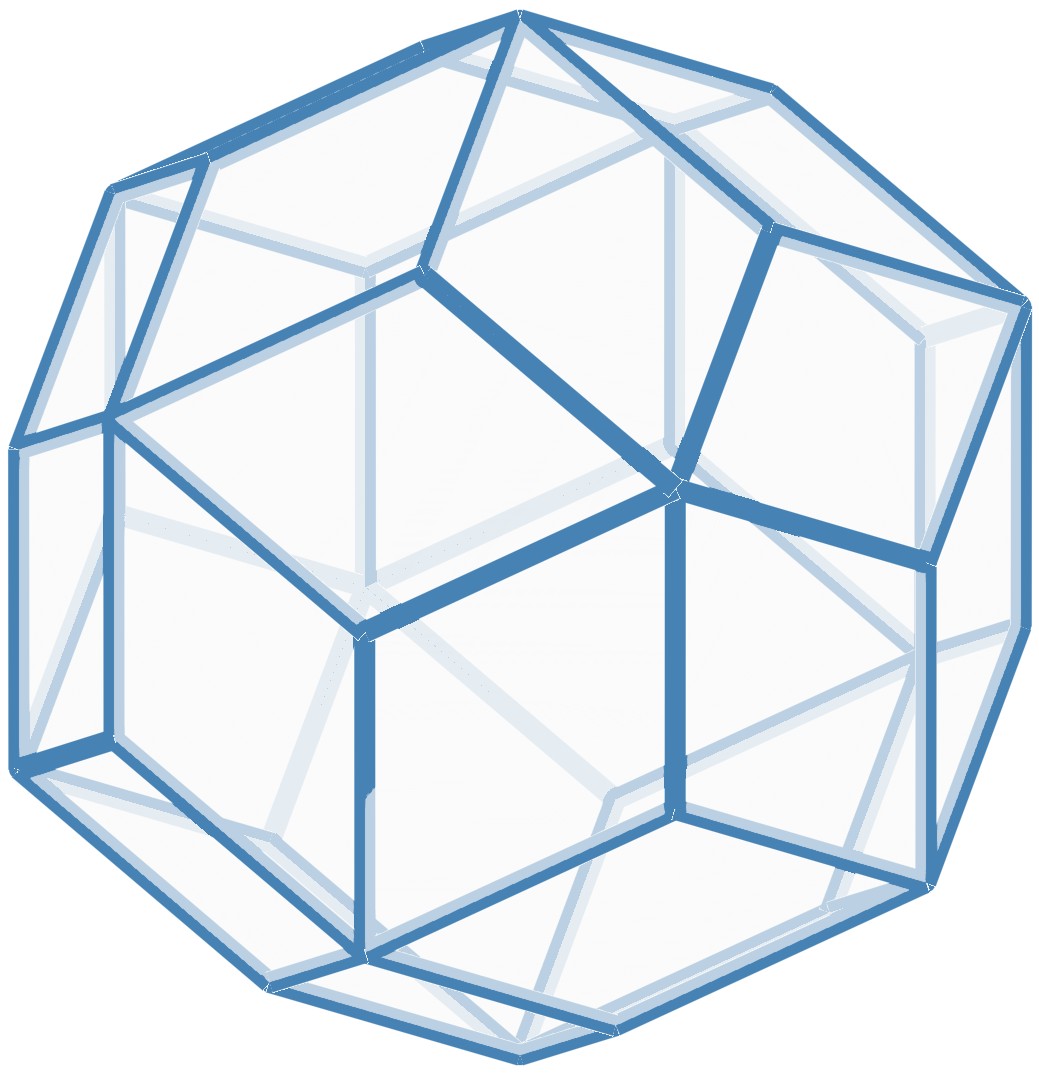} & & \makecell{\\[7mm]\SymTwo\hspace*{-5mm}} &\makecell{Triakis Icosahedron\\[2.5mm](\yes , \yes , \yes , \yes , 0) } & \cincludegraphics{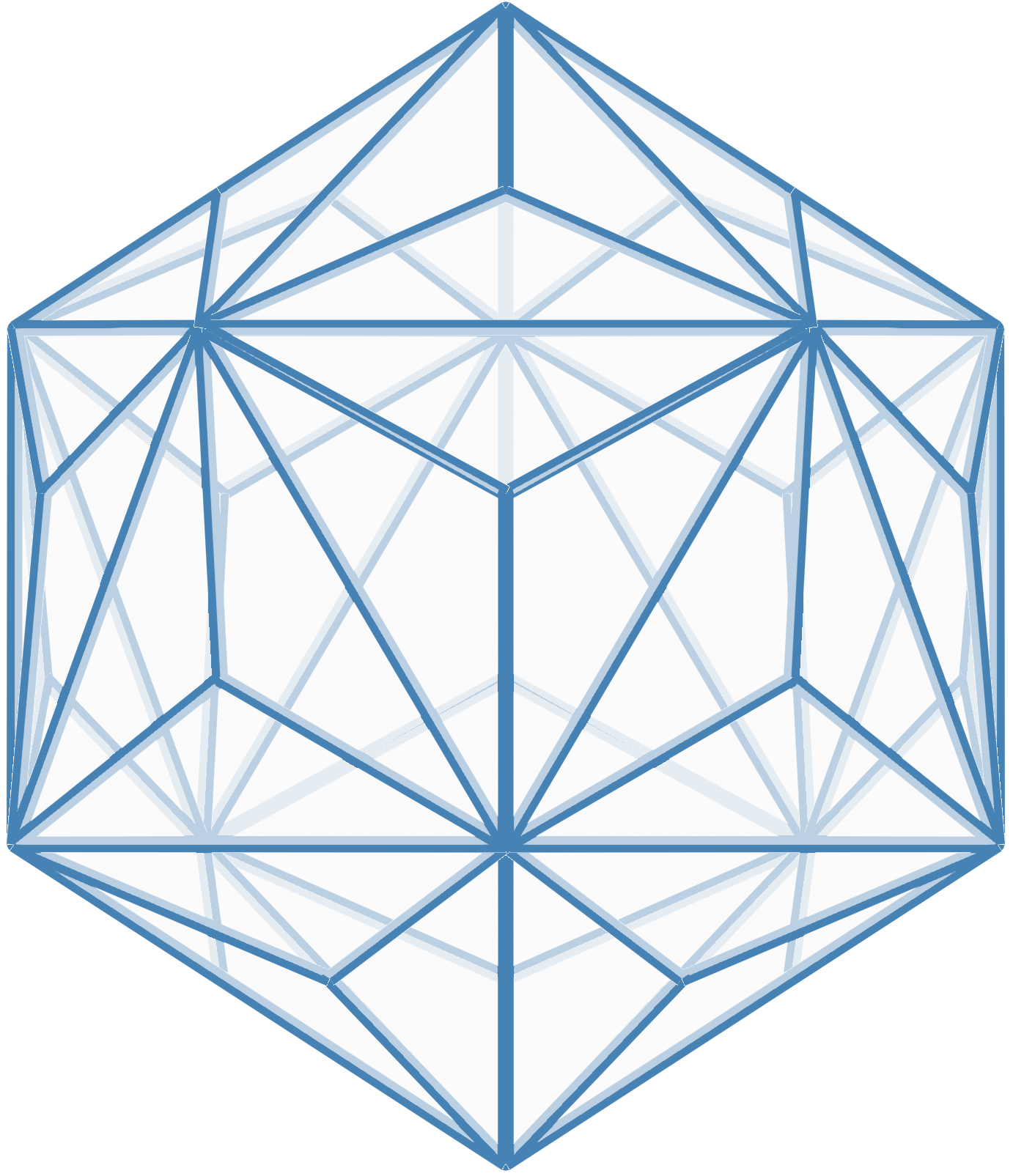} \\ \cline{1-3}\cline{5-7}

\makecell{\\[7mm]\SymTwo\hspace*{-5mm}} & \makecell{Pentakis Dodecahedron\\[2.5mm](\yes , \yes , \yes , \yes , 0)} & \cincludegraphics{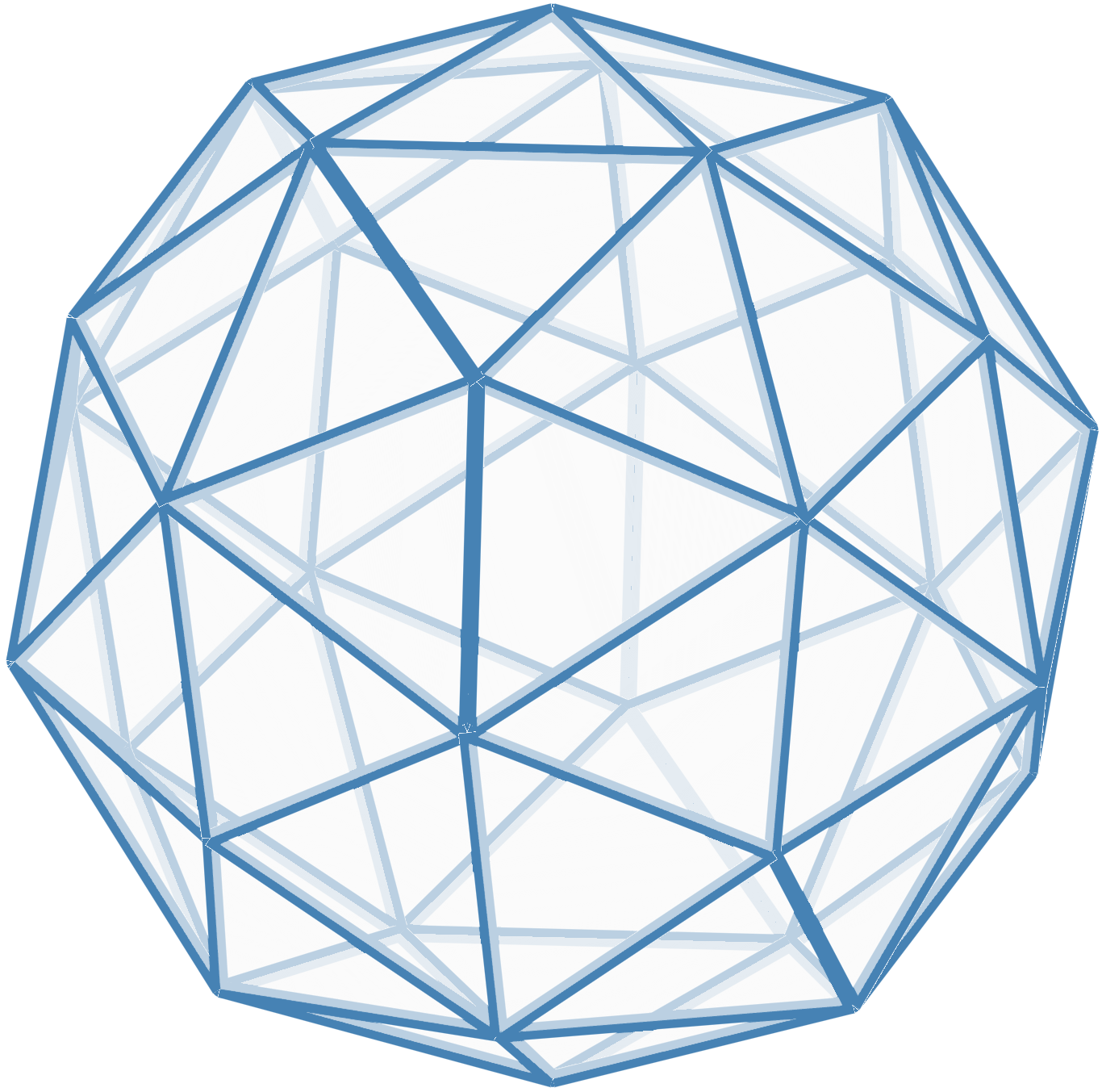} & & &
\makecell{Deltoidal Hexecontahedron\\[2.5mm](\yes , \yes , \yes , \yes , 0)} & \cincludegraphics{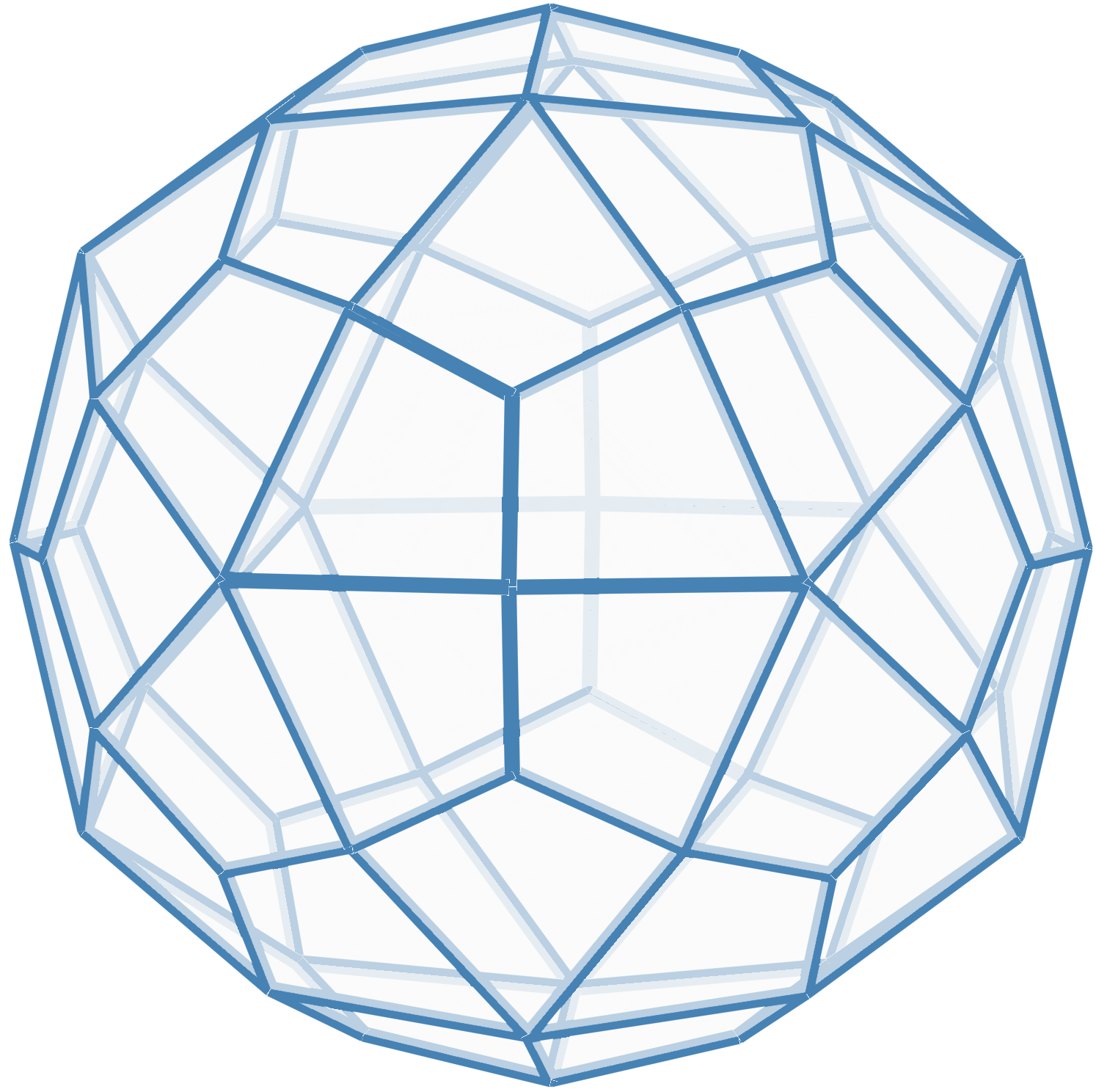}\\ \cline{1-3}\cline{5-7}

\makecell{\\[7mm]\SymTwo\hspace*{-5mm}} & \makecell{Disdyakis Triacontahedron\\[2.5mm](\yes , \yes , \yes , \yes , 0)} & \cincludegraphics{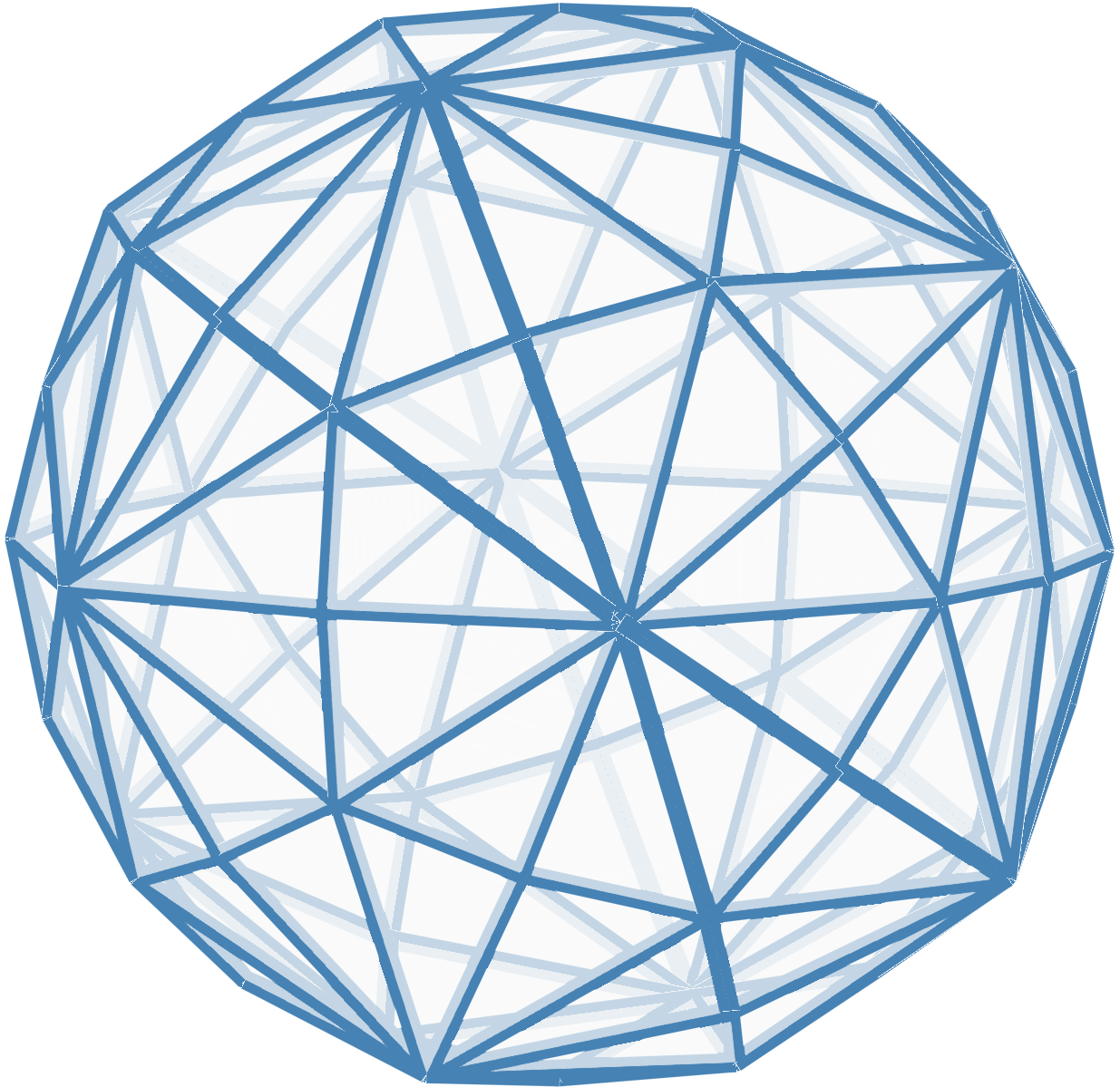} & & &
\makecell{Pentagonal Hexecontahedron\\[2.5mm](\yes , \yes , \yes , \yes , 0)} & \cincludegraphics{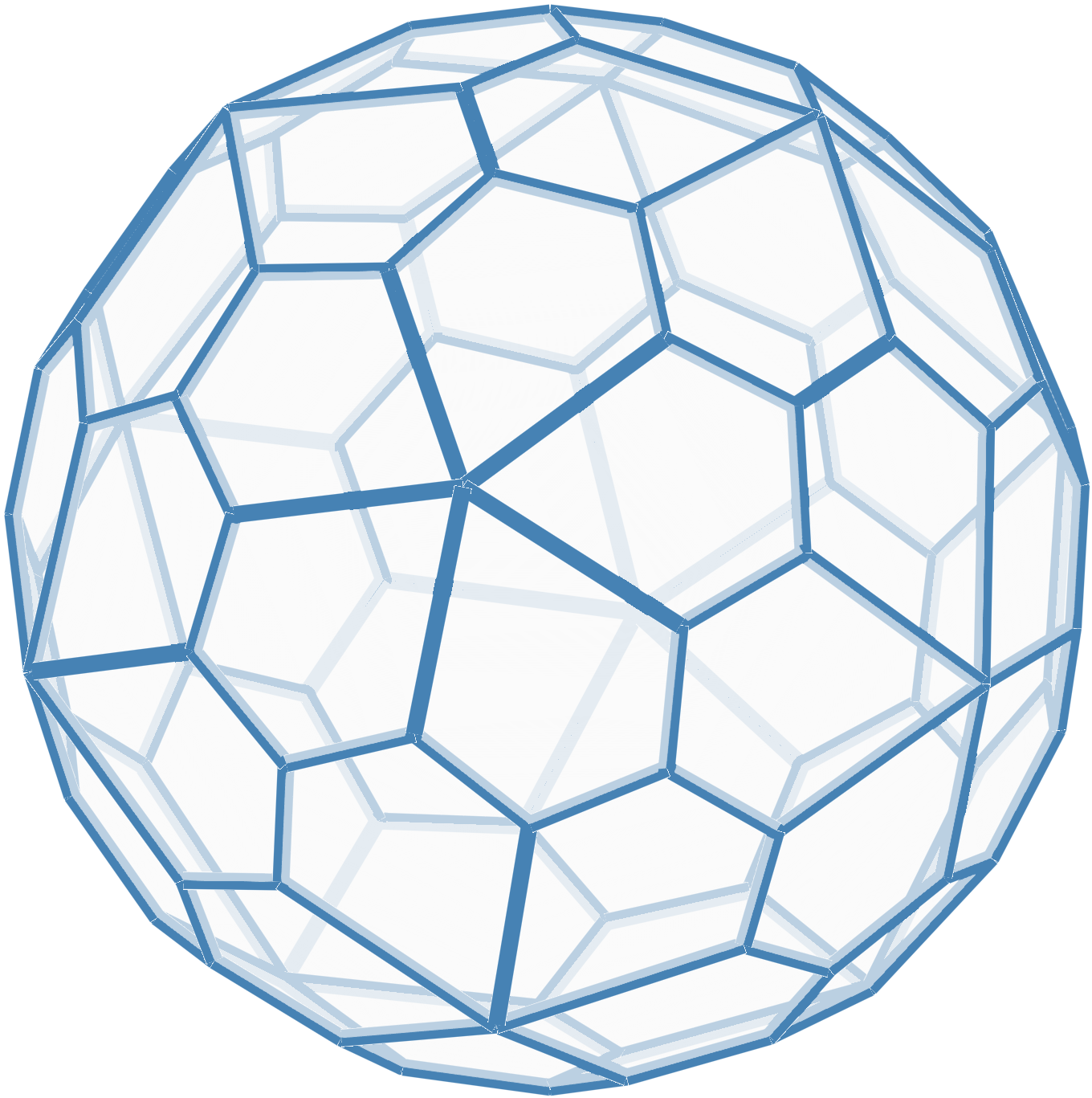}\\ \cline{1-3}\cline{5-7}

\end{tabular}
    \caption{Rigidity properties of the 13 Catalan solids grouped by the polytopes' symmetry. The polytopes are sorted according to the position of the dual Archimedean solid in \Cref{tab:comparison}. 
    For an explanation of the cell contents, see \cref{sec:explain_table}.
    }
    \label{tab:catalan-comparison}
\end{table}

%% file: sec/appendix.tex
\appendix

\section{Dimension bounds for first-order flex spaces}
\label{sec:flexible_polytopes}
\label{sec:dimension_estimates}

The goal of this section is to explain some of the experimentally observed dimension values of the first-order flex space in \cref{tab:comparison} and \ref{tab:catalan-comparison} and draw conclusions from their values.
For this we revisit the constructions for flexible polytopes introduced in \cite[Section 4]{himmelmannschulzewinter2025rigiditypolytopesedgelength}.

Recall that $\mreal^*(\mathcal P,P)$ is a semi-algebraic set. 
At regular points (\ie\ where the space is smooth), the first-order flex space can be understood as the tangent space to $\mreal^*(\mathcal P,P)$, and hence, their (local) dimensions agree.
At non-regular points the relation~is~more complicated.
The strategy in this section is to derive lower bounds on the dimension of the first-order flex space by constructing sub-manifolds $M\subseteq\mreal^*(\mathcal P,P)$.

\subsection{Generic Minkowski sums}
\label{sec:generic_Msum}

Given polytopes $P_1,...,P_m\subset\RR^d,m\ge 2$ in general orientation relative to each other, their Minkowski sum
$$P:= P_1+ \cdots + P_m$$
is a flexible polytope for which we can distinguish two types of flexes:
\begin{myenumerate}
    \item \Def{summand flexes} result from one of the summands $P_i$ being flexible.
    \item \Def{reorientation flexes} result from a relative reorientation of the summands. These exist even if all summands are rigid.
\end{myenumerate}
Both types of flexes are called \Def{Minkowski flexes} (as introduced in \cite[Section 4.1]{himmelmannschulzewinter2025rigiditypolytopesedgelength}).
The assumption that the summands are in general position is necessary, since coincidental coplanarities of faces between the summands can prevent both types of flexes from existing.
As a consequence, $\mreal^*(\mathcal P, P)$ contains a copy of
$$\mreal^*(\mathcal P_1,P_1)\times \cdots\times \mreal^*(\mathcal P_m,P_m) \times \RelO,$$
where $\RelO$ 
is a smooth open manifold that parametrizes the generic relative~orientations of the summands whose Minkowski sum is combinatorially equivalent~to~$P$.

Several polytopes in \cref{tab:comparison} and \ref{tab:catalan-comparison} are generic Minkowski sums (marked with~\SymThree).
This includes
\begin{itemize}
    
    \item the \emph{cuboctahedron}, which is the generic Minkowski sum of two simplices, each of which is rigid.
    
    \item the \emph{rhombicuboctahedron}, which is the generic Minkowski sum of a cube and an octahedron. The octahedron is (first-order) rigid. 
    The cube has~a~3-di\-mensional space of first-order flexes, all of which get inherited by the~Min\-kowski sum as summand flexes.
    
    \item the \emph{rhombicosidodecahedron}, which is the generic Minkowski sum of a dodecahedron and an icosahedron. Both summands are rigid, but the dode\-cahedron has a 5-dimensional space of first-order flexes, which too gets~inherited by the Minkowski sum (\cf\ \cref{sec:first_order}).
\end{itemize}
The polytopes marked with \SymOne\ \SymThree\ are further examples. 
They are \emph{generic~zonotopes} and will be discussed separately in \cref{sec:generic_zonotopes}.

All of the above polytopes are generic Minkowski sums of \textit{two} full-dimensional polytopes.
Hence, the relative orientations of their summands can be parametrized by orthogonal transformations. 
That is,
$\RelO$ is an open subset of $\OO(\RR^3)$ and~provides three dimensions of reorientation flexes.
Remarkably, in all three cases the dimension of the first-order flex space matches the lower bound we get from the above discussion.
As a consequence we find that, locally at their respective realizations, the Minkowski flexes are the only flexes for these polytopes.

\subsection{Zonotopes}
\label{sec:zonotopes}

A \Def{zonotope} $Z\subset\RR^d$ is a Minkowski sum of finitely many line~segments.
A flex of $Z$ that consists of zonotopes throughout is called a \Def{zonotope flex}.
Such are entirely described by a continuous reorientation of the generating line segments.
Note that such a reorientation needs to respect the linear dependencies between the segments in order to preserve the combinatorial type of the zonotope.

It was proven in \cite[Section 4.2]{himmelmannschulzewinter2025rigiditypolytopesedgelength} that zonotope flexes do always exist: the group $\GL(\RR^d)$ can be thought of as acting on $\mreal (\mathcal Z,Z)$ by changing the orientations of the generating line segments, while preserving linear dependencies and the line segments' lengths.
Since $\GL(\RR^d)$ is a Lie group, this action is smooth and generates flexes that~we call \Def{linear zonotope flexes}.

Let $\operatorname{Con}(\RR^d)\subset\GL(\RR^d)$ be the subgroup of \emph{conformal transformations}, 
\ie\ scalar multiples of orthogonal transformations.
Suppose that the line arrangement of $Z$ is \Def{essential}, \ie\ there are no proper subspaces $U_1,...,U_r\subset\RR^d$ with~$\dim U_1+\cdots+\dim U_r=d$ that contain all line segments.
One can check that $\operatorname{Con}(\RR^d)$ is precisely the stabilizer of $Z$~in $\mreal^*(\mathcal Z,Z)$.
In consequence, the action of $\GL(\RR^d)$ generates a submanifold $M\subseteq\mreal^* (\mathcal Z,Z)$ of dimension 
$$\dim\GL(\RR^d)-\dim\operatorname{Con}(\RR^d)=d^2-\big(\kern-1pt\textstyle\binom{d}{2}+1\big)=\binom{d+1}{2}-1.$$

For $d=3$ this implies that $\mreal^* (\mathcal Z,Z)$ contains a submanifold of dimension $\binom{4}{2}-1=5$.
The zonotopes in \cref{tab:comparison} and \ref{tab:catalan-comparison} are marked with \SymOne.
With the exception of the cube, all of them have essential arrangement, and hence have a space of first-order flexes of dimension at least five.
Remarkably, four of them also attain the bound with equality:
\begin{itemize}
    \item the \emph{rhombic dodecahedron}, 
    \item the \emph{truncated octahedron} (aka the  \emph{permutahedron}),
    \item the \emph{truncated cuboctahedron},
    \item the \emph{truncated icosidodecahedron}.
\end{itemize}
As a consequence, locally at their respective realizations, the linear zonotope flexes are the only flexes of these polytopes.

The rhombic dodecahedron is a \emph{generic zonotope} (marked with \SymOne\ \SymThree) and its dimension match will be explained in \cref{sec:generic_zonotopes}.
The other three zonotopes~also share a property that is potentially relevant: let $Z$ be any of these zonotopes,\nls then the matroid $\mathcal M(Z)$ of its line arrangement is \Def{linearly unique}. 
This means that $\mathcal M(Z)$ has a unique linear realization up to linear transformation (and scalar multiplication of the vectors).
This already implies that all zonotope flexes are necessarily linear zonotope flexes.
But we currently have no understanding of why all~flexes~of~these zonotopes should (locally) be zonotope flexes in the first place.
This is not necessary: for example, the hexagonal prism is a zonotope with a non-zonotopal flex.

\begin{question}
    \label{q:only_zonotopes}
    Given a zonotope $Z$, when are all its flexes zonotope flexes? Is this true if $\mathcal M(Z)$ is linearly unique?
    More precisely, if $\mathcal M(Z)$ is linearly unique~and~essential, do we have $\mreal^* (\mathcal Z,Z)\simeq \GL(\RR^d)/\operatorname{Con}(\RR^d) \simeq\operatorname{SPD}^*(\RR^d)$?
\end{question}

Here $\operatorname{SPD}^*(\RR^d)$ is the space of symmetric positive definite matrices with determinant one.
The isomorphism is obtained via $\GL(\RR^d)\ni M\mapsto M\T M$.

\subsection{Generic zonotopes}
\label{sec:generic_zonotopes}

A \Def{generic zonotope} is a generic Minkowski sum of~line segments.
It is both a zonotope (\cref{sec:zonotopes}) and a generic Minkowski sum~(\cref{sec:generic_Msum}).
For $d\ge 3$ these are precisely the polytopes whose 2-faces are~all paral\-lelograms.
\Cref{tab:comparison} and \cref{tab:catalan-comparison} contain three of them marked with \SymOne\ \SymThree. These~are the following ($m$ being the number of line segments):
\begin{itemize}
    \item the \emph{cube} ($m=3$),
    \item the \emph{rhombic dodecahedron} ($m=4$), 
    \item the \emph{rhombic triacontahedron} ($m=6$). 
\end{itemize}
Their Minkowski summands, \ie\ the line segments, are clearly rigid and there are no summand flexes.
Moreover, the dimension of $\RelO$ can be computed exactly by constraining the line segments so as to get rid of all isometries: restrict each line segment to be centered at the origin, and for $k\in\{1,...,d-1\}$ (assuming $m\ge d$), restrict the $k$-th line segment to lie in a fixed $k$-dimensional subspace. It has then $k$ degrees of freedom.
Summing up the degrees of freedom, we obtain
\begin{align*}
\dim \RelO 
&= 0 + 1+\cdots + (d-1) + \overbrace{(d-1) +\cdots +(d-1)}^{m-d}
\\ &= \binom{d}{2} + (d-1)(m-d)
= (d-1)m-\binom{d}{2}.
\end{align*}

For $d=3$ we obtain $\dim\RelO=2m-3$.
%
%
%
This matches~the~dimensions of the first-order flex spaces listed in the tables.
This can be explained:

\begin{theorem} \label{res:generic_zonotopes}
    If $Z\subset\RR^d$ is a generic zonotope, then $\mreal^* (\mathcal Z,Z)\simeq\RelO$ and all $Q\in \mreal (\mathcal Z,Z)$ are zonotopes.
\end{theorem}
\begin{proof}
    The case $d=1$ is trivial. 
    If $d=2$ then $Z$ is a centrally symmetric polygon  and the proof is straight-forward as well.
    %

    If $d\ge3$, and since $Z$ is generic, all 2-dimensional faces of $Z$ are parallelograms.
    If $Q\in \mreal^* (\mathcal Z,Z)$, then all 2-faces of $Q$ are 4-gons with the same edge lengths as in $Z$.
    In particular, all 2-faces of $Q$ are parallelograms. In conclusion, $Q$ is itself a generic zonotope.

    Since the elements of $\mreal^* (\mathcal Z,Z)$ that are generic zonotopes are exactly parame\-trized by $\RelO$, it necessarily covers the entire configuration space.
\end{proof}

\section{Approximating Deformations of Polytopes}
\label{sec:approx_continuous_motions}

When encountering a flexible polytope, we often want to explicitly see how a flex looks like. However, determining the flexibility of a bar-and-joint framework is NP-hard \cite{abbott2008NPhard}, so there is little hope that our more complicated setting will be any easier. Instead, we shift our attention to numerically approximating deformations. This is not only relevant for the approximation of flexes and for their visualization, but also for inducing deformations for rigid polytopes by perturbing the lengths of some edges (\cf\ \Cref{section:edge-length-perturbations}). We thus apply the techniques that are discussed in this section for generating animations of the flexible Platonic and Archimedean solids, which are provided in the Supplementary Material \cite{zenodo_suppMaterial}, and to approximate deformations of the rigid dodecahedron which are induced by modifying the length of a single edge in \Cref{section:dodec-edge-length-perturbations}.

Assume we are given a non-trivial first-order flex $\dot{P}$ that extends to a flex. In other words, $P+\dot{P}$ lies in the linearization of the reduced metric realization space. In Adrian-Himmelmann \cite{deformationpaths}, a method for approximating the corresponding deformations based on Riemannian optimization is implemented as a combination of the metric projection to the closest point and homotopy continuation. It is called the \Def{Euclidean distance retraction} as it projects $P+\dot{P}$ to its closest point on the constraint set. 

The closest point problem can be described by a constrained polynomial optimization problem, so we can reformulate it as a (polynomial) Lagrange multiplier system whose zeros characterize the optimization problem's critical points (\cf\ \cite{Nocedal2006}). Since we already know a solution to this polynomial system in $u=P$ -- namely, the realization $P$ itself -- we can apply homotopy continuation methods consisting of a predictor step (\eg\ Euler's method) and a corrector step (\eg\ Newton's method). Such schemes are often applied for solving polynomial systems \cite{sommese2005numericalsolutionpolynomial}. 

In general, a direct application of the Newton corrector method may diverge or converge to the wrong point. Nevertheless, for the problem of computing the Euclidean distance retraction, the convergence of the predictor-corrector scheme to the closest point to $P+\dot{P}$ can be shown, as long as the linear step $P+\dot{P}$ stays sufficiently close to the underlying constraint set \cite{heaton2025euclideandistanceretraction}.

Nevertheless, three questions remain:
\begin{myenumerate}
    \item how to ensure that an initial flex extends to a continuous motion,
    \item how to compute the subsequent flex in the case where the linear space of first-order flexes has a dimension greater than one, and
    \item how to traverse singularities?
\end{myenumerate}

The Euclidean distance retraction has been refined for geometric constraint systems in Adrian-Himmelmann \cite{deformationpaths}, where all of these questions are addressed using heuristics and theoretical tools. If we assume that there exists a first-order flex that extends to a continuous motion (\eg\ by assuming that the polytope $P$ is flexible), this flex necessarily needs to be a non-blocked flex (\cf\ \cref{res:sotest}). 
For this reason, assume we know a basis of first-order flexes $\dot{P} =(\dot{P}_1, \dots, \dot{P}_r)$ and a basis of equilibrium stresses $\bs \xi = (\bs\xi _1,\dots, \bs\xi_s)$. Following \Cref{section:second-order-rigidity-check}, we solve the homogeneous polynomial system $\nabla_\mu E_{\dot{P},\,\bs \xi}(\lambda) = 0$
over $\mathbb{P}_\mathbb{R}^{r-1}$ for $\lambda_1^*,\dots,\lambda_r^*$ we can obtain a non-trivial flex $\dot{P}=\sum_{i=1}^r \lambda_i^*{\dot{P}}_i$ that is not blocked by any equilibrium stress (\cf\ \Cref{prop:stress-energy-condition}). We take it as our initial flex, as it heuristically has a good chance of providing a flex that extends to a continuous motion, since we have no examples of second-order flexible polytopes that are rigid.

For the second question, we return to Riemannian geometry: On a smooth manifold, the \Def{parallel transport} \cite[Section 10.3]{boumal2020intromanifolds} is a means to compare vectors that live in distinct tangent spaces. This inspires our approach: Given the previous flex $\dot{P}_\text{old}$ corresponding to the realization ${P}_\text{old}$ for which ${P}_\text{old}+\dot{P}_\text{old}$ has been projected to obtain the current realization $P$, we first compute a basis of the space of first-order flexes that we collect as the rows of a matrix $A$. To find the closest analogue to the flex $\dot{P}_\text{old}$ in the new realization $P$, we compute a solution to the minimization problem 
\[\min_{x\in \mathbb{R}^{d(V+F)}}||Ax-\dot{P}_\text{old}||^2,\]
which can, for instance, be solved by computing the Moore-Penrose pseudoinverse of $A$, which provides a minimum norm least squares solution \cite[Proposition
7.2.1]{allaire2008numericallinearalgebra}. This approach guarantees that we pick the closest subsequent non-trivial first-order flex and that our deformation path will be as smooth as possible.

Finally, algebraic constraint sets are distinct from smooth manifolds by having the potential to contain singularities. In our setting, such points can be characterized by the rigidity matrix dropping rank. They notoriously pose issues for smooth approximation approaches. Nevertheless, homotopy continuation schemes come with techniques such as the \Def{singular endgame} to approach singular solutions (\cf\ \cite[Section 10]{sommese2005numericalsolutionpolynomial}). Conversely, escaping singularities poses issues even for such schemes, although this ability is essential for our purposes. To tackle this problem, we take inspiration from Holmes-Cerfon, Theran and Gortler \cite{Holmes-Cerfon2021AlmostRigidity}, in which they describe an acceleration-based method to escape singularities by keeping track of a local, orthogonal frame along the deformation path. This information allows us to reliably traverse nodal-like singularities and to change direction in cusp-like singularities. 

All of these individual techniques for approximating deformation paths have been implemented in the \textsc{Julia} package \textsc{DeformationPaths.jl} \cite{deformationpathszenodo},
where all Platonic, Archimedean and Catalan solids that are discussed in this article are provided as test sets. 

\section{Computational data and explanation of Supplementary Material}
\label{appendix:comp-data-dodecahedron}
\label{appendix:supp_material}

\subsection{Regular dodecahedron}
\label{appendix:SM_dodecahedron}

Let $\phi$ be the golden ratio, $\frac{1}{2}(1+\sqrt{5})$. We use the following coordinates for the 20 vertices:
$$
\begin{array}{llll}
 \bs{p}_1=(1,1,1) & \bs{p}_2=(1,1,-1) & \bs{p}_3=(1,-1,1) & \bs{p}_4=(1,-1,-1) \\
 \bs{p}_5=(-1,1,1) &\bs{p}_6=(-1,1,-1) & \bs{p}_7=(-1,-1,1) & \bs{p}_8=(-1,-1,-1)  \\
 \bs{p}_9=(0,\phi,\frac{1}{\phi}) & \bs{p}_{10}=(\frac{1}{\phi},0,\phi) & \bs{p}_{11}=(\phi,\frac{1}{\phi},0) & \bs{p}_{12}=(0,\phi,\frac{-1}{\phi})\\
  \bs{p}_{13}=(\frac{-1}{\phi},0,\phi) & \bs{p}_{14}=(\phi,\frac{-1}{\phi},0) & \bs{p}_{15}=(0,-\phi,\frac{1}{\phi}) & \bs{p}_{16}=(\frac{1}{\phi},0,-\phi)\\  \bs{p}_{17}=(-\phi,\frac{1}{\phi},0) & \bs{p}_{18}=(0,-\phi,\frac{-1}{\phi}) & \bs{p}_{19}=(\frac{-1}{\phi},0,-\phi) & \bs{p}_{20}=(-\phi, \frac{-1}{\phi},0) 
\end{array}$$
The edges are given by the following pairs of vertices: 
$$\begin{array}{cccccc}
  (1, 9) & (1, 11) & (2, 11) & (2, 12) & (9, 12) & (1, 10) \\(5, 9) & (5, 
  13) & (10, 13) & (8, 20) & (17, 20) & (6, 17)\\  (6, 19) & (8, 19) & (7, 
  13) & (5, 17) & (7, 20) & (4, 18) \\ (4, 14) & (3, 14) &(3, 15) & (15, 
  18) & (6, 12) & (16, 19)\\  (2, 16) & (4, 16) & (11, 14) & (3, 10) & (7, 
  15) & (8, 18)
\end{array}$$
From this information, one can recover the facets and the coordinates of the normal vectors. Our specific computation uses the following vertex ordering for the twelve facets: 
$$\begin{array}{ccc}
     f_1:(1,9,5,13,10)&f_2:(1,9,12,2,11)&f_3:(1,10,3,14,11)\\f_4:(2,11,14,4,16) & 
     f_5:(2,12,6,19,16)&f_6:(3,10,13,7,15)\\f_7:(3,14,4,18,15)&f_8:(4,16,19,8,18)&
     f_9:(5,9,12,6,17)\\f_{10}:(5,13,7,20,17)&f_{11}:(6,17,20,8,19)&f_{12}:(7,15,18,8,20)
\end{array}$$

Symbolic expressions for a basis $(\dot{P}_1,\dots,\dot{P}_5)$ of first-order flexes are provided in the \textsc{Mathematica} data file \texttt{Dodecahedron/symbolicInfVFMatrix.mx}, which can be found in this article's Supplementary Material \cite{zenodo_suppMaterial}. Similarly, symbolic expressions for a basis $({\bs\xi}_1,\dots,{\bs\xi}_5)$ of the corresponding equilibrium stresses are provided in \texttt{Dodecahedron/symbolicStrVFMatrix.mx}. 
The \textsc{Mathematica} notebook \texttt{Dodecahedron/DodecahedronRigidity.nb} follows the process detailed in \cref{sec:so_rigidity_dod}, checking first prestress stability and then second-order rigidity.

\subsection{Truncated icosahedron}
\label{appendix:SM_TIco} 

    Analogously, we show that the truncated icosahedron is second-order rigid. 
    The corresponding computation is performed in the \textsc{Mathematica} notebook \texttt{TruncIco/TICORigidity.nb}.

\subsection{Truncated dodecahedron}
\label{appendix:SM_TDod} 

    For the truncated dodecahedron, we can show that it is \emph{not} second-order rigid using similar techniques when applying \Cref{gbasis}. This polytope has a four-dimensional space of first-order flexes and equilibrium stresses. In the \textsc{Mathematica} notebook \texttt{TruncDod/TDODRigidity.nb}, we show that the polynomial ideal $\langle Q_1,Q_2,Q_3,Q_4\rangle$ is positive-dimensional, as it is homogeneous and has nontrivial solutions.